\documentclass{amsart}

\usepackage{graphics} 
\usepackage{epsfig} 
\usepackage{times} 
\usepackage{amsmath} 
\usepackage{amssymb}  
\usepackage{cases}
\usepackage{verbatim}
\usepackage{comment}
\usepackage{yfonts}
\usepackage{color}

\usepackage{mathtools}

\newtheorem{theorem}{Theorem}[section]
\newtheorem{lemma}[theorem]{Lemma}

\theoremstyle{definition}
\newtheorem{definition}[theorem]{Definition}
\newtheorem{example}[theorem]{Example}

\newtheorem{proposition}[theorem]{Proposition}

\theoremstyle{remark}
\newtheorem{remark}[theorem]{Remark}
\newtheorem{assumption}[theorem]{Assumption}

\def\cal{\mathcal}
\def\la{\langle}
\def\ra{\rangle}
\def\EE{{\mathbb E}}
\def\deltae{\epsilon}

\def\D{\partial}

\numberwithin{equation}{section}

\begin{document}

\title[Mean field social optimization: person-by-person optimality]{Mean field social optimization:
 feedback person-by-person optimality and the dynamic programming equation}

\author{Minyi Huang}
\address{School  of Mathematics and Statistics, Carleton University,
Ottawa, ON K1S 5B6, Canada, and GERAD, Montreal, QC H3T 2A7, Canada}
\email{mhuang@math.carleton.ca}

\author{Shuenn-Jyi Sheu}
\address{Department of Mathematics, National Central University, Chung-Li, Taiwan and Department of Mathematical Sciences, National Chengchi University, Taipei, Taiwan}
\email{sheusj@math.ncu.edu.tw}

\author{Li-Hsien Sun}
\address{Institute of Statistics, National Central University, Chung-Li, Taiwan}
\email{lihsiensun@ncu.edu.tw}

\subjclass[2000]{}

\date{May 2026. The first version appeared as a GERAD Technical Report G-2024-45, Aug 2024 (https://www.gerad.ca).}

\keywords{Controlled diffusion, mean field social optimization, person-by-person optimality, dynamic programming, master equation}

\begin{abstract}
We consider mean field social optimization in nonlinear diffusion models.
By  dynamic programming with a representative agent employing cooperative optimizer selection,
we derive a new Hamilton--Jacobi--Bellman (HJB) equation to be called the master equation of the value function. Under some regularity conditions, we establish $\epsilon$-person-by-person optimality of the master equation-based control laws, which may be viewed as a necessary condition for nearly attaining the social optimum. A major challenge in the analysis is to obtain tight estimates, within an error of $O(1/N)$, of the social cost having order $O(N)$. This will be accomplished by multi-scale analysis via constructing two auxiliary master equations. We illustrate  explicit solutions of the master equations  for the linear-quadratic (LQ) case, and give an application to systemic risk.
\end{abstract}

\maketitle

\tableofcontents

\section{Introduction}

Mean field social optimization studies decision problems involving a large number of agents which have a  common optimization objective and interact through coupling in their individual dynamics or costs, or both \cite{HCM12}. These problems are also referred to as large population optimal control \cite{DPT20}.  The cooperative behavior of the agents differs from the noncooperative behavior of the agents in mean field games \cite{CHM17}. The reader is referred to \cite{AM15,HWY21,HCM12,HY21,WZ17} for the analysis of social optima in a linear-quadratic (LQ) framework.
McKean--Vlasov optimal control has been studied in \cite{BCP18,CDJS22,CD15,DPT22,L17,PW17,TTZ21,WZ20},
 and may be heuristically interpreted as a limit form of large population optimal control. Cooperative mean field control has applications in
economic theory \cite{NM18}, collective choice problems \cite{SNM18}, multi-agent flocking \cite{PRT15},
and power systems \cite{CBM17}. The notion of social optima is also useful
for measuring (in)efficiency of mean  field games \cite{BT13,CR19}.

In this paper we analyze social optimization for a class of nonlinear models.
 Let $(\Omega, {\mathcal F}, \{{\mathcal F}_t\}_{t\ge 0}, {\mathbb P} )$ be the underlying filtered probability space satisfying the usual condition.
Consider a population of  agents ${\cal A}_i$, $1\le i\le N$, satisfying the stochastic differential equations (SDEs):
\begin{align}
dX_t^i=\ & f(X_t^i, u_t^i, \mu_{t}^{-i})dt+
\sigma dW_t^i
+\sigma_0 dW_t^0, \quad  1\le i\le N, \quad  t\ge 0, \label{sdeXi}
\end{align}
where $X_t^i$ and $u_t^i$ are agent ${\mathcal A}_i$'s   state and control, respectively, and
$$\mu_{t}^{-i}\coloneqq \frac{1}{N-1}\sum_{1\le j\le N, j\ne i}\delta_{ X^j_t}$$ is  the empirical distribution of all other $N-1$ agents, with $\delta_x$ being the Dirac measure at $x\in \mathbb{R}^n$. The initial states $\{X_0^i, 1\le i\le N\}$ are ${\mathcal F}_0$-measurable and are independent with finite second moment. The $N+1$ ${\mathcal F}_t$-adapted  standard Brownian motions $\{W^j, 0\le j\le N\}$, as the individual noises and the common noise, respectively, are independent  and   also independent of the initial states.
The dimensions of  $X_t^i$,  $u_t^i$,
$W^i_t$ and $ W^0_t$ are $n$, $n_1$, $n_2$, $n_3$, respectively, and $f$, $\sigma$ and $\sigma_0$ have compatible dimensions.
We restrict to the case of constant noise coefficients. Our approach can deal with  state and control dependent noise.

Agent ${\cal A}_i$ has its own cost
\begin{align}
&J_i(u^i(\cdot), {\mathbf u}^{-i}(\cdot)) =\EE\int_0^T L(X_t^i, u_t^i, \mu_{t}^{-i}) dt +\EE\ g(X_T^i, \mu_T^{-i}),\label{cosXi}
\end{align}
where  ${\mathbf u}^{-i}(\cdot)$  denotes the controls of all agents other than ${\cal A}_i$. To avoid heavy notation,  we do not include $t$ as an argument in $f, \sigma, \sigma_0$ and  $L$, but can treat the $t$-dependent case  without further difficulty.
The social cost is given by
\begin{align}\label{jsocN}
J_{\rm soc }^{(N)}(\mathbf{u})= \sum_{i=1}^N J_i,
\end{align}
where ${\mathbf u}\coloneqq (u^1, \cdots, u^N)$.

\subsection{Method, analytical challenges, and contributions}

One may attempt to
minimize  the social cost by directly solving an optimal control problem. This approach, however, becomes infeasible  when $N$ is  large.
Instead, we develop our solution by studying the optimizing behavior of a representative agent.
We exploit a simple but useful idea in team decision theory called person-by-person (PbP) optimality. The reader is referred to \cite{H80,SM76,WS00} for its characterization and to \cite{CA18} for its feedback form in decentralized stochastic control.  To explain the idea, let ${\mathcal J}(u^1, \cdots, u^N)$ be the team cost of $N$ agents with strategies $u^i$. Under a given information structure, if the team   attains its optimum by a joint strategy $(\hat u^1, \cdots, \hat u^N)$, then no agent can unilaterally take a new strategy to improve for the team. PbP optimality is a necessary condition for team optimality.

Our  PbP optimality-based approach studies the value function of a representative agent and applies dynamic programming in an extended state space including an individual state $x$ and a measure $\mu$ describing the mean field.
This method enables the agent to choose its optimizer in a feedback form $\hat u^i= \phi(t, x, \mu)$, which differs from \cite{SHM16}, where the control perturbation is
a non-anticipative process. In the end we obtain a special Hamilton--Jacobi--Bellman (HJB) equation, which will be called the master equation for the value function; some heuristic derivation has appeared in \cite{HSS20} without performance analysis. Master equations have been widely studied in mean field games \cite{BFY13,CDLL19,CD18} and also mean field optimal control \cite{BCP18,BFY13,DPT22,PW17}.

In the setting of social optimization (or large population optimal control), our approach above differs from some existing works, which approximate the $N$-agent optimal control problem by a McKean--Vlasov optimal control problem (see e.g. \cite{DPT22,DPT20,L17,WZ20}). As a result, their value function takes the form ${V}_{\rm mv}(t, \mu)$, only characterizing the performance of the population as a whole. However, in  practical cooperative decision scenarios, each constituent agent still desires to know its own performance; our approach responds to this need and specifies an individual's value function $V(t,x,\mu)$.
The authors in \cite{BCP18} formulate an optimal control problem combing a representative agent's state equation with the McKean--Vlasov dynamics, both controlled by the same process. They prove a dynamic programming principle without common noise, and the drift term in their master equation is structured differently from ours.

When the control laws $(\hat u^1, \cdots, \hat u^N)$ of $N$ agents have been determined from the master equation, an issue of central importance is their performance.
 For certain classes of models, the value function of the controlled McKean--Vlasov dynamics has been shown to be the limit of the value function of the $N$-agent optimal control problem as $N\to \infty$ \cite{CDJS22,DPT22}. However,
to our best knowledge, except  for LQ cases \cite{FHH23,HCM12,HN16,WZ17}, there has existed no past work on social optimization addressing the performance of the infinite population-based control laws when they are applied by finite populations. As it turns out, our PbP optimality-based approach provides a tractable framework for such performance analysis.
 We will establish  approximate PbP optimality  for the master equation-based control laws in optimizing $J^{(N)}_{\rm soc}$, which may be viewed as a necessary condition for achieving near social optimality. Naturally, it is desirable to bound the gap between the attained performance and exact PbP optimality as tightly as possible. A loose bound would not tell much about performance.

PbP optimality may be viewed as a special case of a Nash equilibrium where all agents take the social cost as their individual costs. For this reason, our performance estimate in terms of approximate PbP optimality is similar to the $\epsilon$-Nash equilibrium analysis in mean field games \cite{CHM17}. However, we confront a far more difficult problem, especially under model nonlinearity.    We may think of $J^{(N)}_{\rm soc}$ as a quantity of order $O(N)$ while expecting the agent in question  to minimize it up to an accuracy of $o(1)$.
   The estimate of the performance loss becomes very intricate since we face quantities of such drastically different scales as $O(N)$ and $o(1)$. In contrast, the $\epsilon$-Nash equilibrium analysis involves quantities of scales $O(1)$ and $o(1)$, for an agent's cost and its performance loss, respectively.
 Under some regularity  conditions pertaining to the solution of the master equation, we prove that
$\hat u^i$ is nearly PbP optimal, with an inadequacy (or regret) of at most $O(1/N)$ for optimizing $J^{(N)}_{\rm soc}$.

Our performance analysis depends on exploiting the multi-scale feature when an individual attempts to optimize the social cost, where the key idea is to decompose the social cost roughly  as the sum of a much larger macroscopic term and another term  that carries information on the representative agent and so will be called the instrumental value function. Meanwhile,  we need to introduce two auxiliary master equations. Such combined usage of functions describing phenomena of different scales, together with the construction of their equations,  has got much of its inspiration from the physics literature on computing tiny quantities which are superposed to very large ones \cite[p. 199-209]{ABS75} \cite[p. 195-198]{D92}
and also from perturbation methods of dynamical systems \cite[ch. 3]{H13}.

The paper is organized as follows.
Section \ref{sec:me} introduces the master equation of the value function. To prepare for performance analysis,  Section \ref{sec:PbP} introduces the instrumental value function and two auxiliary master equations.
The asymptotic PbP optimality theorem is proved in Section \ref{sec:PO}. Section \ref{sec:lq} illustrates explicit solutions of the master equations by linear-quadratic (LQ) models and a systemic risk example.

\subsection{Notation}
\label{sec:sub:no}

The Frobenius norm of a vector or matrix $v$ is denoted by $|v| $. For the function
$\psi(x)$ from $\mathbb{R}^n$ to $\mathbb{R}$, we denote the partial derivatives
 $\partial_{x_i} \psi$ and
$\partial_{x_ix_j} \psi\coloneqq \partial_{x_i} (\partial_{x_j} \psi )$. Denote $\partial_x \psi\coloneqq (\partial_{x_1} \psi , \cdots, \partial_{x_n}\psi)$ as a row vector. The Hessian matrix is
$
\partial_x^2 \psi \coloneqq  (\partial_{x_i x_j}\psi)_{1\le i,j\le n},
$
where $\partial_{x_i x_j}\psi$ is the $(i,j)$-th entry of the matrix.
If $\psi(x,y)$ is a scalar  function with $x, y\in \mathbb{R}^n$, we denote $$
\partial_{xy}\psi(x,y) \coloneqq   (\partial_{x_i y_j}\psi(x,y) )_{1\le i,j\le n}.
$$
The evaluation $\psi(x,y)|_{y=x}$ gives $\psi(x,x)$. When there is only one space variable $x$ in $\psi$ (as in $\psi(t, x, \mu)$), we use the short notation $\psi_x = \partial_x\psi $ and $\psi_{xx}=\partial_x^2 \psi$.

Let ${\cal P}_2({\mathbb R}^n)$ be the space of Borel probability measures
on ${\mathbb R}^n$ with finite second moment.
 On ${\cal P}_2({\mathbb R}^n)$, we endow the Wasserstein distance $W_2(\mu, \nu)$
 so that it becomes a complete metric space \cite{AGS05}.
 Let ${\mathcal P}_{\rm em}^k(\mathbb{R}^n)$ consist of all probability measures $\mu$ on $\mathbb{R}^n$ such that $\mu=\frac{1}{k}\sum_{i=1}^k{\delta_{{x}^i}}$ for some ${\mathbf x}=(x^1, \cdots, x^k) $ with ${ x}^i\in \mathbb{R}^n$, $1\le i\le k$.

For  a function $\psi:\mathbb{R}^n\to \mathbb{R}^k$ and a probability measure $\mu$, we write the integral $\int_{\mathbb{R}^n} \psi (x) \mu (dx) $ as  $ \la\mu ,\psi \ra $ or $\la \psi\ra_\mu$. We sometimes  write such form $\langle \mu(dy), \psi(t,y,\mu) \rangle$ to explicitly indicate the variable of integration.
  For two probability measures $\mu_1$ and $\mu_2$, $\la \mu_1-\mu_2, \psi \ra \coloneqq  \la\mu_1 ,\psi \ra - \la\mu_2 ,\psi \ra$.

Denote the agent index set ${\mathcal N}=\{1, \cdots, N\}$, and ${\mathcal N}_{-i}={\mathcal N}-\{i\}$. Throughout the paper, the agent indices $i, j\in {\mathcal N}$, etc, appearing in ($X^i_t$, $u_t^i$, $W_t^j$, $x^i$, $u^j$, etc.)  are always interpreted as a superscript, and not as an exponent.
The controls of the $N$ agents are written as ${\mathbf u}\coloneqq (u^1, \cdots, u^N).$ Further let ${\mathbf u}^{-i}$ be the controls of all agents except ${\mathcal A}_i$.
We follow \cite{CDLL19} (which uses the notation $\delta V/\delta \mu$) to define the measure differentiation for a function $V(t, x, \mu)$; for  fixed $(t, x)$, the derivative $\delta_\mu V$, as a Lebesgue measurable function from $\mathbb{R}^n$ to $\mathbb{R}$, is denoted by $\delta_\mu V(t, x, \mu; y)$ satisfying the normalization condition
$\langle \mu(dy), \delta_\mu V(t,x,\mu;y) \rangle=0  $, where $y\in \mathbb{R}^n$ after the semicolon  is the new variable arising from the differentiation. Higher order derivatives are similarly introduced and denoted in the form $\delta_{\mu\mu} V(t, x, \mu; y,z)$, etc. A function (such as $\psi(t,x,\mu)$) is said to be jointly continuous if it is continuous under the product topology of $[0,T]\times \mathbb{R}^n \times {\cal P}_2(\mathbb{R}^n)$.

Unless otherwise indicated, the partial derivative will be based on the actual appearance of the variables. For example, $\partial_x$  in  $\partial_x \delta_\mu V(t, y, \mu; x)$ means differentiation with respect to the variable after $\mu$, even if we initially introduce the form $\delta_\mu V(t, x, \mu; y)$. The functions $\chi$,  $\chi_1$, $\chi_2$, $\cdots$ (such as $\chi(t,\mu)$, $\chi_1(t,x,\mu)$, etc.) are reserved as a normalizing term (see Section \ref{sec:sub:lqme} for examples) when differentiating a function with respect to $\mu$. In various estimates we use $C, C_1, C_2, \cdots$
as generic constants which do not depend on $N$ and may change from place to place.

\section{The value function and its master equation}
\label{sec:me}

 Let $ \mathbb{U}$ be a nonempty closed subset of $ \mathbb{R}^{n_1}$.
 Here and hereafter, we reuse $C_a$ as a generic constant in various assumptions.
We make the following standing assumptions:

(A1) The map $f$:
$ \mathbb{R}^n\times \mathbb{U}\times
{\cal P}_2(\mathbb{R}^n)\to \mathbb{R}^{n}$ is continuous, and
 Lipschitz continuous in $(x,\mu)$,
uniformly with respect to $u$.  In addition,
$|f(x,u,\mu)|\le  C_{a}(1+|x|+|u|+\langle|y|\rangle_\mu)$.

(A2) $L: \mathbb{R}^n \times \mathbb{U}\times {\mathcal P}_2(\mathbb{R}^n) \to \mathbb{R}_+$  and $g: \mathbb{R}^n \times {\cal P}_2(\mathbb{R}^n)\to \mathbb{R}_+$ are  continuous, and
$$
L( x, u, \mu),\ g(x,\mu) \le C_a(1+|x|^2+|u|^2+\langle|y|\rangle^2_\mu),
$$
for all  $ (x,u,\mu)\in  \mathbb{R}^n \times \mathbb{U}\times {\cal P}_2(\mathbb{R}^n)$.

\subsection{The $\phi$-value function}

To explain the idea underlying the  derivation of  the master equation,
 we analyze a family of  $(N,\phi)$-indexed control problems for agent ${\cal A}_i$ in a population of $N$ agents.
 Consider a set of control laws
$\phi(t,X_t^k, \mu_t^{-k})$, $1\le k\le N$,  with a continuous function $\phi $ from $[0,T]\times \mathbb{R}^n \times {\mathcal P}_2 (\mathbb{R}^n )$   to ${\mathbb U}$ that ensures a unique solution to the SDE system \eqref{sdeXi}.

For each $\phi$,
let $V^{N,\phi}(t, x^i, \mu^{-i})$ denote the cost
$$J_i(t,x^i, \mu^{-i},\hat{\mathbf{u}}(\cdot))=\EE\int_t^T L(X_s^i, u_s^i, \mu_{s}^{-i}) ds +\EE g(X_T^i, \mu_T^{-i})$$
of agent ${\cal A}_i$ when all $N$ agents apply the control laws $\hat u^k_s= \phi (s,X_s^k, \mu_s^{-k} ) $, $1\le k\le N$, on $[t,T]$
with $X_t^i=x^i$ and
$\mu_{t}^{-i}=\mu^{-i}\in {\mathcal P}_{\rm em}^{N-1}(\mathbb{R}^n)$.
The choice of $\mu^{-i}\in{\mathcal P}_{\rm em}^{N-1}(\mathbb{R}^n) $ is arbitrary, which may be matched by appropriate initial states of all other agents ${\mathcal A}_j$, $j\ne i$ at time $t$.   Then $V^{N,\phi}$ is a well-defined function on $[0,T]\times \mathbb{R}^n
\times {\mathcal P}_{\rm em}^{N-1}(\mathbb{R}^n)$.
It is clear that this cost depends on $N$. For notational simplicity, we will
just write it as $V^{\phi}(t, x^i, \mu)$ by dropping the superscript $N$. We may denote the cost of ${\cal A}_i$ in this form due to symmetry of the other $N-1$ agents. Specifically, if we permutate the other agents' initial states, the cost of ${\cal A}_i$ remains the same.
 We shall call $V^\phi$ the $\phi$-value function for ${\cal A}_i$.

Now let the $N$ agents be assigned the initial time-state values  $(t, x^j, \mu^{-j} )$, $1\le j\le N$, respectively, which are subject to the constraint
$\frac{1}{N}\delta_{x^j} + \frac{N-1}{N} \mu^{-j} = \frac{1}{N}\sum_{k=1}^N \delta_{x^k} $ for all $ j$.
For the initial time $t\in [0,T]$, denote the controlled state processes
for agents ${\cal A}_j$, $j\ne i$,
\begin{align}
dX_s^j=& f(X_s^j,
 \phi(s, X_s^j,  \mu_s^{-j}), \mu_s^{-j}) ds
  +\sigma dW_s^j +\sigma_0 dW_s^0,\quad s\ge t. \label{sdexjt}
\end{align}
For agent ${\cal A}_i$, we have
\begin{align}
dX_s^i =\ &f(X_s^i, u_s^i, \mu_s^{-i})ds +\sigma dW_s^i
 +\sigma_0 dW_s^0,\quad  t\le s\le T,\label{sdexit}
\end{align}
where we take $u_s^i\equiv u^i\in \mathbb{U}$ for $s\in [t,t+\deltae)$ and $u_s^i=\phi(s,X_s^i, \mu_s^{-i})$ on $[t+\deltae, T]$.
Under \eqref{sdexjt}--\eqref{sdexit}, denote
\begin{align}
J_i(t,x^i, \mu^{-i},u^i(\cdot), \hat {\mathbf u}^{-i}(\cdot)) & \coloneqq\EE\int_t^T L(X_s^i, u_s^i, \mu_{s}^{-i}) ds +\EE g(X_T^i, \mu_T^{-i}) \notag \\
 &= \EE\Big[ \int_t^{t+\epsilon} L(X_s^i, u^i,
\mu_s^{-i})ds+
V^\phi({t+\epsilon}, X^i_{t+\epsilon}, \mu^{-i}_{t+\epsilon} ) \Big].\label{Jidel}
\end{align}

Using  dynamic programming, we select $u^i\in {\mathbb U}$ on the small interval $[t, t+\deltae]$ to minimize
$$\EE\Big[ \int_t^{t+\epsilon} L(X_s^i, u^i,
\mu_s^{-i})ds+
\sum_{k=1}^NV^\phi({t+\deltae}, X^k_{t+\deltae},
\mu^{-k}_{t+\deltae} )\Big]
$$
 and next take $N\to \infty$. This allows us to formally derive the master equation of the value function $V(t,x,\mu)$, as the limiting form of $J_i$, after finding the minimizer $\hat u^i$ which is further required to be equal to $\phi$; see details in  appendix A.

\subsection{The dynamic programming equation (or master equation)}

Denote the matrices
$
\Sigma \coloneqq  \sigma \sigma^T,\
\Sigma_0 \coloneqq \sigma_0 \sigma_0^{T}.
$
 We interpret $(x, u^i)$ as the state and control of the representative agent $ {\mathcal A}_i $ in an infinite population.
The master equation of the value function $V(t,x,\mu)$ takes the form
\begin{align}
- \partial_t V(t, x, \mu)
= \ &   V_x (t, x, \mu)
f(x, \hat u^i, \mu) + \frac{1}{2} {\rm Tr}
[ V_{xx}(t, x, \mu) (\Sigma+\Sigma_0)  ]
\label{MEVwoCN}   \\
&+ L(x, \hat u^i, \mu)  + \langle \mu(dy),  \partial_y\delta_\mu V(t,x, \mu; y) f( y, \phi(t, y, \mu),\mu)\rangle \notag  \\
&+\frac{1}{2} \langle \mu(dy), {\rm Tr}[
\partial_{y}^2 \delta_\mu
V(t, x, \mu; y) (\Sigma+\Sigma_0)]\rangle \notag \\
&+\langle \mu(dy),  {\rm Tr} [ \partial_{xy}\delta_\mu V(t,x, \mu;y)  \Sigma_0 ]\rangle \notag  \\
&+\frac{1}{2}\langle \mu^{\otimes 2}(dydz),  {\rm Tr}[\partial_{yz}\delta_{\mu\mu} V (t, x, \mu;y, z)
\Sigma_0]\rangle ,\notag\\
&\hskip -1cmV(T, x, \mu)=g(x,\mu),\qquad (t,x,\mu)\in[0,T]\times \mathbb{R}^{n}\times
 {\cal P}_2(\mathbb{R}^n), \label{vtcg}
\end{align}
where
 \begin{align}
 \hat u^i\coloneqq & \phi(t,x, \mu)
 \label{optu_noCN}
 \end{align} is selected with $\phi$ fulfilling the  condition
\begin{align}
 \Phi(t, x,  \phi(t, x, \mu),\mu, V(\cdot))=& \min_{u^i}\Phi(t, x, u^i,\mu,   V(\cdot) )
\end{align}
for
\begin{align}
\Phi(t, x, u^i,\mu,   V(\cdot) )
 \coloneqq &
 V_x (t, x, \mu)
f(x, u^i, \mu) + L(x, u^i, \mu)  \nonumber\\
&   +\langle
\mu(dy),  \partial_{x} \delta_\mu V(t, y, \mu; x) f(x,  u^i,\mu)   \rangle . \nonumber 
\end{align}

\begin{definition}\label{def:vsol}
We call the pair $(V, \phi)$, as mappings from $[0,T]\times \mathbb{R}^n\times {\mathcal P}_2(\mathbb{R}^n) $ to $\mathbb{R}$ and $\mathbb{U}$, respectively,  a solution of the master equation \eqref{MEVwoCN} if the following conditions are satisfied:

(i) $V$, $\partial_tV$, $V_{x}$ and $V_{xx}$ are each jointly continuous  on
 $[0,T]\times \mathbb{R}^n \times {\mathcal P}_2 (\mathbb{R}^n )$;

(ii)
$\delta_\mu V(t,x,\mu; y)$, $\delta_{\mu\mu} V(t, x, \mu; y,z)$, $\partial_y\delta_\mu V(t,x,\mu; y)$, $\partial_y^2\delta_\mu V(t,x,\mu; y)$, $\partial_{xy}\delta_\mu V(t,x,\mu; y)$, $\partial_{yz}\delta_{\mu\mu} V(t, x, \mu; y,z)$ are    jointly continuous in $t\in [0,T]$, $x,y,z\in \mathbb{R}^n$, and
 $\mu \in   {\mathcal P}_2(\mathbb{R}^n)  $;

(iii) $\phi$ is continuous on $[0,T] \times \mathbb{R}^n\times {\mathcal P}_2(\mathbb{R}^n)$ with
 $$
 |\phi(t, x, \mu)|\le C_\phi (1+|x|+\langle |y| \rangle_{\mu});
 $$

(iv) there exists a constant $C_a$  such that for all $(t, x, y,z, \mu)$,
\begin{align}
\begin{cases}
|\delta_\mu V(t,x,\mu;y)|, \ |\delta_{\mu\mu} V(t,x,\mu; y,z)|\le C_a (1+|x|^2+|y|^2  +|z|^2+\langle |y|\rangle_{\mu}^2 ) , \\
 |\partial_y\delta_\mu V(t,x,\mu;y)|\le C_a(1+|x|+|y|+  \langle |y| \rangle_{\mu} ),\\
|\partial_{y}^2 \delta_\mu V(t,x,\mu; y)|,\ | \partial_{xy}\delta_\mu V(t,x, \mu;y)|\le C_a , \\
 |\partial_{yz}\delta_{\mu\mu} V (t, x, \mu; y,z)|\le C_a;
\end{cases} \label{vgc1}
\end{align}

(v) $(V,\phi)$ satisfies \eqref{MEVwoCN} with terminal condition \eqref{vtcg}.
\end{definition}

 For the subsequent analysis, it is necessary to look for a solution of  equation \eqref{MEVwoCN} with further restrictions on $V$. We introduce  the following class of functions ${\cal C}_V$ consisting of all functions $V$ from $[0,T]\times \mathbb{R}^n \times {\mathcal P}_2 (\mathbb{R}^n )$     to $\mathbb{R}$ such that (a)
$V$ fulfills conditions (i) and (ii) in Definition \ref{def:vsol}, and furthermore,    $\delta_{\mu\mu\mu} V(t,x,\mu;y,z,w)$,
$\partial_x \delta_\mu V(t,x,\mu;y)$,
$\partial_x^2 \delta_\mu V$, $ \partial_t\delta_\mu V $,
$\partial_y\delta_{\mu\mu} V (t,x,\mu;y,z) $, $\partial_x\delta_{\mu\mu} V  $,
$\partial_y^2\delta_{\mu\mu} V  $, $\partial_{xy}\delta_{\mu\mu} V  $,
 $\partial_{y}\delta_{\mu\mu\mu} V$,
$\partial_{yz}\delta_{\mu\mu\mu} V$ are jointly continuous for $t\in [0,T]$, $x,y,z,w\in \mathbb{R}^n$, $\mu\in {\cal P}_2(\mathbb{R}^n)$;
(b) $V$ satisfies \eqref{vgc1} and moreover,
\begin{align} \label{agc}
\begin{cases}
|V(t,x,\mu)|\le C_a(1+|x|^2+\langle | y|\rangle^{2}_\mu ), \\
|V_x(t,x,\mu)|\le C_a(1+|x|+\langle |y| \rangle_{\mu}  ) ,  \\
 |V_{xx}(t,x,\mu)|\le C_a,\\
  |\partial_t\delta_\mu V(t, x, \mu;y)| \le C_a(1+|x|^2+|y|^2+\langle |y| \rangle^2_{\mu}   ), \\
  |\partial_x\delta_\mu V(t, x, \mu;y)| \le C_a(1+|x|+|y|+\langle |y| \rangle_{\mu}   ),   \\
   |\partial_x^2\delta_\mu V(t, x, \mu;y)|\le C_a,  \\
  |\partial_y\delta_{\mu\mu} V(t, x, \mu;y,z)|\le C_a(1+|x|+|y|+|z|+ \langle |y| \rangle_{\mu}  ) ,  \\
 |\partial_y^2\delta_{\mu\mu} V(t, x, \mu;y,z)| ,  \
  |\partial_{xy}\delta_{\mu\mu} V(t, x, \mu;y,z)|\le C_a,    \\
   |\delta_{\mu\mu\mu}(t, x,\mu;y,z,w)|\le C_a(1+|x|^2+|y|^2+|z|^2
+|w|^2+\langle |y| \rangle_{\mu}^2 ),    \\
 |\partial_{y}\delta_{\mu\mu\mu} V(t,x, \mu;y,z,w)|\le C_a (1+|x|+|y|+|z|+|w|+ \langle |y| \rangle_{\mu} ),   \\
  |\partial_{yz}\delta_{\mu\mu\mu} V(t, x, \mu;y,z,w)|\le C_a .
\end{cases}
\end{align}

The above additional growth conditions in \eqref{agc}
are mainly for ensuring  well-defined coefficients in the equation
of $U$ in Section \ref{sec:PbP} and for constructing a solution of $U$, which involves differentiating both sides of \eqref{MEVwoCN} and leads to various higher order derivatives.

\section{Close-loop systems and auxiliary master equations}
\label{sec:PbP}

We begin by introducing the following assumption.

\begin{assumption}\label{assum:meV}
The master equation \eqref{MEVwoCN} has a solution pair $(V, \phi)$ with $V\in {\cal C}_V$.
\end{assumption}

A question of central interest is what kind of social performance can be achieved  when the  master equation-based control law
 $\phi $ in \eqref{optu_noCN} is used by all $N$ agents. Note that $\phi(t, x, \mu)$ involves the measure-valued variable $\mu$. For
 its  implementation, we will use the actual process $\mu_t^{-i}$ in  the control laws of the $N$ agents:
\begin{align}
\hat u_t^i= \phi(t, X_t^i, \mu_t^{-i}), \quad 1\le i\le N.
\end{align}
Under the set of control laws $\hat u_t^i$, $1\le i\le N$,
the closed-loop system takes the form
\begin{align}
dX_t^i=\ & f(X_t^i, \phi(t, X_t^i,\mu_t^{-i}), \mu_{t}^{-i})dt+
\sigma dW_t^i +\sigma_0 dW_t^0 , \quad 0\le t\le T, \label{allXphi0} \\
&\quad 1\le i\le N. \notag
\end{align}

 To analyze the performance of the control law $\phi$ in terms of  PbP optimality, we need to consider a general control
 $u_t^1$ for agent ${\cal A}_1$ while the other agents take
 \begin{align}
\hat u_t^k= \phi(t, X_t^k,\mu^{-k}_t), \quad 2\le k\le N.
 \end{align}
The key question is by how much the social cost may be reduced by optimizing $u_t^1$.
This leads to  a control problem with social cost $J_{\rm soc}^{(N)}(u^1(\cdot), \hat {\mathbf u}^{-1} (\cdot)) $ and   dynamics
\begin{align}
dX_t^1=\ & f(X_t^1, u_t^1, \mu_{t}^{-1})dt+
\sigma dW_t^1+\sigma_0 dW_t^0,  \label{X1u} \\
dX_t^k=\ & f(X_t^k, \phi(t, X_t^k,\mu_t^{-k}), \mu_{t}^{-k})dt+
\sigma dW_t^k +\sigma_0 dW_t^0 ,   \label{Xkuhat} \\
& 0\le t\le T, \quad 2\le k\le N.   \notag
\end{align}
  To facilitate the performance analysis, we need to examine two functions $U$ and $\overline U$ below, which are  related to the asymptotic behavior of the social costs attained by $\phi$ as $N\to \infty$.

The social cost under the control law $\phi$ in \eqref{optu_noCN} for all agents will   serve as a benchmark performance level.
Based on \eqref{allXphi0},  we consider the state processes
on $[s,T]$:
\begin{align}
dX_t^i=\ & f(X_t^i, \phi(t,X_t^i,\mu_t^{-i}) , \mu_t^{-i}) dt +
 \sigma dW_t^i+\sigma_0 dW_t^0 , \quad s\le t\le T, \label{Xphis}  \\
& \quad 1\le i\le N, \notag
\end{align}
which are assigned  the initial condition $(s, x^1,\cdots, x^N)$.  The use of  general initial conditions will enable us to derive a partial differential equation of the cost
\begin{align} \label{JNphi}
J^{(N)}_{\rm soc, \phi}(s, x^1, \cdots, x^N)\coloneqq  \sum_{i=1}^NJ_i(s,x^i, \mu^{-i}, \phi)
\end{align}
subject to  dynamics  \eqref{Xphis}. Here $J_i(s, x^i,\mu^{-i}, \phi)$ means that the cost is evaluated with initial time $s$, initial state $(x^i, \mu^{-i})$, and the same control law $\phi$ for $N$ agents.

 We will look for a suitable representation of $J_{\rm soc, \phi}^{(N)}$.
By symmetry of all other agents,  agent ${\cal A}_1$ may write $J_{\rm soc, \phi}^{(N)}$  in \eqref{JNphi} as
$U^{N}_{\rm soc }(s, x^1, \mu^{-1})$, where the initial condition $\mu^{-1}$, as the empirical  probability distribution of the state values
$(x^2, \cdots, x^N   )  $, is sufficient for describing the behavior of all other agents as a whole.

We take a decomposition of the form
\begin{align}
U^{N}_{\rm soc }(t,x^1, \mu^{-1})= U^N(t, x^1, \mu^{-1}) + (N-1) \overline U(t, \mu^{-1}),
\label{U1nrep}
\end{align}
where the function $\overline U$, to be identified later,  is defined on $[0,T]\times {\mathcal P}_2({\mathbb R}^n)$ and does not depend on $N$.
Given  $U_{\rm soc }^{N}(t, x^1, \mu^{-1})$, one might choose different pairs $(U^N, \overline U)$ for the right hand side of
\eqref{U1nrep}. However, an appropriate choice of $\overline U$ is crucial in order to prevent uncontrolled growth of $U^N$ as $N\to \infty$. We  will take
\begin{align}
\overline U (t, \mu)\coloneqq \langle \mu(dy), V(t, y, \mu) \rangle, \quad
\mu\in {\cal P}_2(\mathbb{R}^n),
\label{defUUb}
\end{align}
which is well-defined in view of \eqref{agc}.
Our next step is to get the limit form of $U^N(t, x^1, \mu^{-1})$, denoted by $U(t, x, \mu)$, and determine the equation of $U$. Since $U$ aggregates the effect of a particular agent's state $x$ on the social cost $J_{\rm soc, \phi}^{(N)}$  which itself approaches infinity as $N\to \infty$, we shall call $U$ the instrumental value function (IVF) to distinguish it from the  value function $V$.

\subsection{Auxiliary master equations}

For $(t,x, \mu)\in [0,T]\times \mathbb{R}^n\times {\mathcal P}_2(\mathbb{R}^n)$,  define
 \begin{align*}
& f^*(t,x,\mu)\coloneqq f(x, \phi(t, x,\mu), \mu),\quad    L^*(t,x,\mu)\coloneqq L(x, \phi(t, x,\mu), \mu)  .
 \end{align*}
We introduce the following assumption:

\begin{assumption}\label{assum:flg}
Each of the functions $ \delta_\mu f^*$, $ \delta_{\mu\mu} f^*$, $ \delta_\mu L^*$, $ \delta_{\mu\mu} L^*$, $ \delta_\mu g$ and $ \delta_{\mu\mu} g$ is jointly continuous in its arguments, and
there exists a constant $C_a$ such that
\begin{align*}
&|\delta_\mu f^*(t, x, \mu;y)| \le C_a(1+|x|+ |y|+\langle |y| \rangle_{\mu}  ),   \\
& |\delta_{\mu\mu} f^*(t, x, \mu;y,z)| \le C_a(1+|x|+ |y|+|z|+\langle |y| \rangle_{\mu}   ),    \\
 &|\delta_{\mu} L^*(t, x, \mu; y) |\le C_a(1+|x|^2+ |y|^2+\langle |y| \rangle_{\mu}^2),\\
&|\delta_{\mu\mu} L^*(t, x, \mu; y,z) |\le C_a(1+|x|^2+ |y|^2+|z|^2+\langle |y| \rangle_{\mu}^2),\\
&|\delta_{\mu} g( x, \mu; y) |\le C_a(1+ |x|^2+|y|^2+\langle |y| \rangle_{\mu}^2),\\
& |\delta_{\mu\mu} g( x, \mu; y,z) |\le C_a(1+ |x|^2+|y|^2+|z|^2+\langle |y| \rangle_{\mu}^2),
\end{align*}
for all $t\in [0,T]$, $x,y,z\in\mathbb{R}^n $.
\end{assumption}

We introduce the following master equation for the instrumental value  function $U(t,x,\mu)$:
\begin{align}
0=\ &\partial_t U(t, x,\mu) + U_x(t,x,\mu) f^*(t,x,\mu)  \label{MEUa}   \\
&+ \frac{1}{2}{\rm Tr}[  U_{xx}(t, x, \mu)
(\Sigma+\Sigma_0)]+L^*(t, x, \mu)\notag \\
 & + \langle \mu(dy), \ \partial_y\delta_\mu U(t,x,\mu;y) f^*(t,y,\mu)\rangle \nonumber \\
& +\frac12 \langle \mu(dy), \   {\rm Tr}[ \partial_{y}^2 \delta_\mu U(t,x,\mu; y)( \Sigma  +\Sigma_0)]\rangle \nonumber \\
& + \langle \mu(dy),   \mbox{Tr}[ \partial_{xy} \delta_\mu U(t, x, \mu; y) \Sigma_0] \rangle   \notag  \\
& +\frac12 \langle \mu^{\otimes2}(dydz), \mbox{Tr}[ \partial_{yz}  \delta_{\mu\mu} U(t, x, \mu; y,z) \Sigma_0 ]     \rangle    \notag   \\
 & + \langle \mu(dy) , \   \delta_\mu L^* (t,y,\mu; x) - \delta_\mu L^*(t,y,\mu; y) \rangle \nonumber  \\
 & + \langle \mu(dy) , \  \partial_y\delta_\mu \overline U(t,\mu;y)    [\delta_\mu f^* (t,y,\mu; x) - \delta_\mu f^*(t,y,\mu; y)] \rangle \nonumber \\
  &+   \frac{1}{2}\langle \mu(dy),    \mbox{Tr}\{ [\partial_{yz}
\delta_{\mu\mu}
\overline U(t,\mu;y, z)]|_{z=y}\Sigma\}\rangle \notag,\\
&\hskip -0.5cm U(T, x, \mu)= g(x, \mu) + \langle \mu(dy), \delta_\mu g(y,\mu; x)- \delta_\mu g(y,\mu; y)  \rangle  ,  \label{UTm}  \\
&\qquad\qquad \qquad(t,x,\mu)\in[0,T]\times \mathbb{R}^n\times
 {\cal P}_2(\mathbb{R}^n),\notag
\end{align}
where $\overline U(t,\mu)$ is given by \eqref{defUUb}.

It will be helpful to explain how equation
\eqref{MEUa} is derived.
Our method is to follow the idea underlying Feynman--Kac's formula to derive an equation for
$U_{\rm soc}^{N}(t,x^1, \mu^{-1}) $  using the dynamics of $(X_t^1, \mu_t^{-1})$ in \eqref{Xphis}. Subsequently we separate terms of two different scales, depending on whether they are multiplied by $N-1$, to formally obtain two equations for $U^{N}$ and $\overline U$, respectively. Finally, taking $N\to \infty$, we obtain a limit form $U$ for $U^N$ and its  equation \eqref{MEUa}. See appendix D for details. By this procedure we have in fact identified the equation of $\overline U$ (see \eqref{MEUb}) first and will next show that \eqref{defUUb} gives a solution.

Our next step is to construct a particular  solution to  master equation \eqref{MEUa}.  To do this, we need to introduce another function $M$: $[0,T]\times {\mathcal P}_2(\mathbb{R}^2)\to \mathbb{R}$ by
 the following equation:
\begin{align}
0=&\ \partial_t M(t, \mu)+
\langle  \mu(dy) ,  \partial_y \delta_\mu  M(t, \mu; y) f^*(t, y, \mu)     \rangle \label{meG2}\\
& +\frac12 \langle \mu(dy), \mbox{Tr} [\partial_{y}^2\delta_\mu M (t,\mu;y) (\Sigma+\Sigma_0)]\rangle \nonumber \\
& +\frac12 \langle \mu^{\otimes2}(dydz), \mbox{Tr}[ \partial_{yz}\delta_{\mu\mu} M (t,\mu;y,z) \Sigma_0]\rangle \nonumber \\
& + \langle \mu^{\otimes2}(dydz), \partial_{y}\delta_\mu V (t,z,\mu;y) [\delta_\mu f^*(t,y,\mu;z) -\delta_\mu f^*(t,y,\mu;y)   ]\rangle, \nonumber \\
&+ \frac12 \langle \mu^{\otimes2}(dydw), \mbox{Tr}\{  [\partial_{yz}\delta_{\mu\mu} V (t,w,\mu;y,z)]|_{z=y} \Sigma\}\rangle,  \nonumber\\
&\hskip -0.5cm  M(T,\mu)=0, \quad (t,\mu)\in[0,T]\times
 {\cal P}_2(\mathbb{R}^n).
\end{align}

We have constructed this equation so as to use $M$ as an adjustment term in the representation of $U$ (see \eqref{defUU} below).
 With $V$ known, we may view \eqref{meG2} as a linear equation.

Let  the class ${\cal C}_M$ consist of all functions $M(t, \mu)$ from $[0,T]\times {\cal P}_2(\mathbb{R}^n)$ to $\mathbb{R}$ such that (i) $M(t, \mu)$, $\partial_tM(t, \mu)$ $\delta_\mu M(t, \mu;y)$, $\delta_{\mu\mu} M(t, \mu; y,z)$, $\partial_y\delta_\mu M(t, \mu;y)$, $\partial_y^2\delta_\mu M(t, \mu;y)$, $\partial_{yz} \delta_{\mu\mu} M(t, \mu;y,z)$  are jointly continuous; (ii) for  some constant $C_a,$
\begin{align*}
 & |\delta_\mu M(t,\mu;y) |,\  |\delta_{\mu\mu} M(t,\mu;y,z) |\le C_a(1+|x|^2+|y|^2+ \langle |y|^2 \rangle_\mu ),   \\
&|\partial_y\delta_\mu M(t,\mu;y)|\le C_a(1+|y|+\langle |y|\rangle_{\mu}), \\
& |\partial_y^2\delta_\mu M(t,\mu;y)|\le C_a ,  \ |\partial_{yz}\delta_{\mu\mu} M(t,\mu;y,z)|\le C_a.
\end{align*}

 We further introduce the following assumption.
\begin{assumption}\label{assum:meM}
There exists a solution $M\in {\cal C}_M$ to equation \eqref{meG2}.
\end{assumption}

\begin{theorem} Under Assumption \ref{assum:meV},
 $\overline U$ defined by \eqref{defUUb} satisfies
the following equation
\begin{align}
&0=  \partial_t\overline U(t,\mu) +\langle \mu,  L^*(t,y, \mu)    \rangle
+
 \langle \mu, \partial_y \delta_\mu \overline U (t, \mu; y) f^*(t, y, \mu)   \rangle  \label{MEUb} \\
& \qquad + \frac12 \langle \mu,  {\rm Tr} [\partial_{y}^2 \delta_\mu \overline U (t,\mu; y)(\Sigma +\Sigma_0)]  \rangle  \nonumber\\
&\qquad + \frac12 \langle  \mu^{\otimes2}(dydz),   {\rm Tr}[ \partial_{yz} \delta_{\mu\mu} \overline U(t,  \mu; y,z) \Sigma_0]      \rangle , \notag\\
&\overline U(T,\mu)=  \langle \mu, g(y,\mu) \rangle ,\quad (t,\mu)\in[0,T]\times
 {\cal P}_2(\mathbb{R}^n).
\end{align}

\end{theorem}

\begin{proof}
For $\overline U$ defined in \eqref{defUUb},
we  have
\begin{align}
\delta_\mu\overline U(t, \mu;y)= V(t,y,\mu) + \chi(t,\mu) +\langle \mu(dw), \delta_\mu V(t, w, \mu;y) \rangle, \label{oUdmu1}
\end{align}
where $\chi$ is a normalizing term.
We have
\begin{align}
\delta_{\mu\mu}\overline U(t, \mu; y,z) = &\ \delta_\mu V(t, y, \mu; z)+\delta_\mu\chi(t, \mu;z) +\delta_\mu V(t, z, \mu; y)+\chi_1(t,\mu,y) \label{oUdu2}\\
&  +\langle \mu(dw),  \delta_{\mu\mu} V (t,w,\mu;y,z) \rangle.\notag
\end{align}
Therefore,
\begin{align}
\partial_{yz}\delta_{\mu\mu}\overline U(t, \mu; y,z) =&\partial_{yz}\delta_\mu V(t, y, \mu; z) +\partial_{yz}\delta_\mu V(t, z, \mu; y) \label{oUdmu3}\\
& +\langle \mu(dw),  \partial_{yz}\delta_{\mu\mu} V (t,w,\mu;y,z) \rangle.\notag
\end{align}
The partial differentiation in the last term can get inside the integration by an application of the dominated convergence theorem.

Substituting these expressions into the right hand side of \eqref{MEUb}, we
obtain
\begin{align}
\Xi \coloneqq&\  \partial_t\langle \mu, V(t,y,\mu)\rangle +\langle \mu,  L^*(t,y, \mu)    \rangle
+ \langle\mu ,  V_y(t,y,\mu) f^*(t, y,\mu)\rangle
   \label{CXi}\\
&+
 \langle \mu^{\otimes2}(dydw), \partial_y \delta_\mu V (t,w, \mu; y) f^*(t, y, \mu)   \rangle \notag  \\
& + \frac12 \langle \mu,  {\rm Tr} [
V_{yy}(t,y,\mu)(\Sigma +\Sigma_0)]  \rangle \notag  \\
&  + \frac12 \langle \mu^{\otimes 2}(dydw),  {\rm Tr} [\partial_{y}^2  \delta_\mu V (t,w,\mu;y)(\Sigma +\Sigma_0)]  \rangle  \nonumber\\
& + \frac12 \langle  \mu^{\otimes2}(dydz),   \mbox{Tr}\{[ \partial_{yz} \delta_{\mu} V(t, y, \mu; z) + \partial_{yz}\delta_{\mu} V(t, z, \mu; y)] \Sigma_0\}      \rangle  \notag\\
& + \frac12 \langle  \mu^{\otimes3}(dydzdw),   \mbox{Tr}\{ \partial_{yz} \delta_{\mu\mu} V(t, w, \mu;y, z)  \Sigma_0\}      \rangle     \notag  .
\end{align}
 By a change of variables, we further apply Lemma \ref{lemma:trDyz} to write the constituent term
\begin{align}
 & \langle  \mu^{\otimes2}(dydz),   \mbox{Tr}[\partial_{yz}\delta_{\mu} V(t, z, \mu; y)\Sigma_0]      \rangle\notag \\
&=  \langle  \mu^{\otimes2}(dydz),   \mbox{Tr}[\partial_{yz}\delta_{\mu} V(t, y, \mu; z)\Sigma_0].
 \label{dVyz}
\end{align}
 Now after substituting \eqref{dVyz} into \eqref{CXi}, we show $\Xi=0$
by integrating both sides of \eqref{MEVwoCN} using $\mu(dx)$, where we have $\langle \mu(dx), \partial_t V(t,x,\mu)\rangle = \partial_t\langle \mu(dx),  V(t,x,\mu)\rangle$.
\end{proof}

  We now take
 \begin{align}
 U(t, x, \mu)=\ &  V(t, x, \mu) + M(t, \mu) \label{defUU}   \\
   & \ + \langle \mu(dy),  \delta_\mu  V(t, y,\mu; x)- \delta_\mu  V(t, y,\mu; y)      \rangle \notag
 \end{align}
 for $(t,x,\mu)\in [0,T]\times \mathbb{R}^n\times
 {\cal P}_2(\mathbb{R}^n) $, and  proceed to show that this function satisfies \eqref{MEUa}.

\begin{remark}
We explain the idea behind the representation \eqref{defUU}.
We write $U^{N}_{\rm soc }(t,x^1, \mu^{-1})= \sum_{i=1}^N V^\phi (t, x^i, \mu^{-i}) $ and next expand $V^\phi(t, x^i, \mu^{-i})$ around $\mu^{-1}$.
The resulting sum is compared with \eqref{U1nrep} to suggest the construction in  \eqref{defUU}.
\end{remark}

\begin{theorem}\label{theorem:U}
Under Assumptions \ref{assum:meV}, \ref{assum:flg} and \ref{assum:meM},
 $U$  defined by \eqref{defUU} is a solution to  master equation   \eqref{MEUa}
 with terminal condition \eqref{UTm}.
\end{theorem}
\begin{proof}
See appendix C.
\end{proof}

\subsection{A minimizer property for $\phi$}

The next lemma gives a useful inequality resulting from taking $u^i$ in place of $\phi(t, x, \mu)$ in equation \eqref{MEUa} for $U$.

\begin{lemma}\label{lemma:optimizer}
Suppose Assumptions \ref{assum:meV}, \ref{assum:flg} and \ref{assum:meM} hold and    let $U$ be given by \eqref{defUU}.
Then for each $u\in \mathbb{U }$, we have
\begin{align}
0\le \ &\partial_t U(t, x,\mu) + U_x(t,x,\mu) f(x,u,\mu)  \label{MEU}   \\
&+ \frac{1}{2}{\rm Tr}[  U_{xx}(t, x, \mu)
(\Sigma+\Sigma_0)]+L( x, u,\mu)\notag \\
 & + \langle \mu(dy), \ \partial_y\delta_\mu U(t,x,\mu;y) f^*(t,y,\mu)\rangle \nonumber \\
& +\frac12 \langle \mu(dy), \   {\rm Tr}[ \partial_{y}^2 \delta_\mu U(t,x,\mu; y)( \Sigma  +\Sigma_0)]\rangle \nonumber \\
& + \langle \mu(dy),   {\rm Tr}[ \partial_{xy} \delta_\mu U(t, x, \mu; y) \Sigma_0] \rangle   \notag  \\
& +\frac12 \langle \mu^{\otimes 2} (dydz), {\rm Tr}[ \partial_{yz}  \delta_{\mu\mu} U(t, x, \mu; y,z) \Sigma_0 ]     \rangle    \notag   \\
 & + \langle \mu(dy) , \   \delta_\mu L^* (t,y,\mu; x) - \delta_\mu L^*(t,y,\mu; y) \rangle \nonumber  \\
 & + \langle \mu(dy) , \  \partial_y \delta_\mu \overline U(t,\mu;y)    [\delta_\mu f^* (t,y,\mu; x) - \delta_\mu f^*(t,y,\mu; y)] \rangle \nonumber \\
  &+   \frac{1}{2}\langle \mu(dy),    {\rm Tr}\{ [\partial_{yz}
\delta_{\mu\mu}
\overline U(t,\mu;y, z)]|_{z=y}\Sigma\}\rangle,\notag
\end{align}
where $\overline U$ is given by \eqref{defUUb}.
\end{lemma}

\begin{proof}
There are only two  terms within \eqref{MEU} that depend on $u$.
By \eqref{defUU}, we have
\begin{align}
&U_x(t, x, \mu)= V_x(t,x,\mu) + \langle \mu(dy) , \partial_x\delta_\mu V (t, x, \mu;y)\rangle,
\end{align}
where differentiation $\partial_x  $ can go inside the integral by dominated convergence.
 We have
\begin{align*}
& U_x(t,x,\mu) f(x, u, \mu)+L(x,u,\mu)\\
& \quad =[ V_x(t,x,\mu) + \langle \mu(dy) , \partial_x\delta_\mu V (t, x, \mu;y)\rangle ]f(x,u,\mu)  +L(x,u,\mu)  \\
&\quad \ge [ V_x(t,x,\mu) + \langle \mu(dy) , \partial_x\delta_\mu V (t, x, \mu;y)\rangle ]f^*(t,x,\mu)
  +L^*(t,x,\mu),
\end{align*}
where the inequality is due to the choice of $\phi$ in \eqref{optu_noCN}. On the other hand, $U$ satisfies \eqref{MEUa} by Theorem \ref{theorem:U}.
The lemma follows from \eqref{MEUa} and the above inequality.
\end{proof}

\section{Person-by-person optimality}
\label{sec:PO}

\begin{assumption} \label{assum:Lip}
$\phi(t, x, \mu)$ in the solution pair $(V, \phi)$ for \eqref{MEVwoCN} is  Lipschitz continuous in $(x, \mu)\in\mathbb{R}^n\times {\cal P}_2(\mathbb{R}^n)$,
 uniformly with respect to $t$.
\end{assumption}

\begin{assumption}\label{assum:X0}
There exists a constant $C_X$ such that
$\sup_N \max_{1\le i\le N} \EE| X_0^i|^2
\le C_X$.
\end{assumption}

Under the standing assumption (A1) and Assumptions \ref{assum:meV} and \ref{assum:Lip},  $f^*(t, x, \mu)$ is Lipschitz continuous in $(x, \mu)$ with linear growth, i.e., for some constant $C_{f,\phi}$,
\begin{align}
&|f^*(t,x, \mu)-f^*(t,y,\nu)|\le C_{f,\phi}(|x-y|+W_2(\mu, \nu)), \label{fLipC} \\
&|f^*(t,x, \mu)|\le C_{f,\phi} (1+|x|+\langle|y| \rangle_\mu).
\end{align}

\begin{proposition} \label{prop:ex2}
Under Assumptions \ref{assum:meV}, \ref{assum:flg}, \ref{assum:meM},  \ref{assum:Lip} and \ref{assum:X0}, for $(X_t^1, \cdots, X_t^N)$ given by \eqref{allXphi0}, we have
$$
\sup_N \max_{1\le i\le N} \sup_{0\le t\le T}\EE|X_t^i|^2\le C_{T,X},
$$
where the constant $C_{T,X}$ only depends on $(T, C_X, C_{f,\phi}, \Sigma, \Sigma_0)$.
\end{proposition}

\begin{proof}
Consider the closed-loop system
\begin{align}
dX_t^i =f^*(t,X_t^i, \mu^{-i}_t)dt + \sigma dW^i_t +\sigma_0 dW^0_t, \quad 1\le i\le N.\label{sdexf5}
\end{align}
Take $\mu_1= \frac{1}{k}\sum_{j=1}^k \delta_{x^j}$ and $\mu_2= \frac{1}{k}\sum_{j=1}^k \delta_{y^j} $ with $x^j,y^j\in \mathbb{R}^n$. Taking a particular coupling of $\mu_1$ and $\mu_2$  as the distribution
assigning probability $1/k$ at each point $(x^j,y^j)\in \mathbb{R}^{2n}$, we obtain $W_2(\mu_1, \mu_2)
\le (\frac{1}{k}\sum_{j}|x^j-y^j|^2)^{1/2}\le \sum_j |x^j-y^j|$. So in view of \eqref{fLipC},  $f^*(t, X_t^i, \mu_t^{-i})$ is Lipschitz continuous in $(X_t^1, \cdots, X_t^N)$
ensuring that the SDE system \eqref{sdexf5} has a unique solution on $[0,T]$.
Applying It\^o's formula to \eqref{sdexf5} and next taking expectation, we have
\begin{align}
\EE|X_t^i|^2 &= \EE|X_0^i|^2 + \EE\int_0^t[ X_s^{iT}f^*(s, X_s^i, \mu_s^{-i})+   {\rm Tr} (\Sigma+\Sigma_0) ]ds\notag \\
   &\le \EE|X_0^i|^2+ T {\rm Tr} (\Sigma+\Sigma_0) +
  C_{f,\phi} \EE\int_0^t |X_s^i| \Big(1+|X_s^i| + \frac{1}{N-1}\sum_{k\ne i} |X^k_s|\Big)ds\notag\\
  & \le \EE|X_0^i|^2+ {T} {\rm Tr} (\Sigma+\Sigma_0) +  C_1 \int_0^t \Big(1+\EE|X_s^i|^2+ \frac{1}{N-1}\sum_{k\ne i}\EE |X_s^k|^2\Big)ds. \notag
\end{align}
Denote $z_t=\max_{1\le i\le N} \EE|X_t^i|^2$. We obtain
$
z_t\le C_2+C_1\int_0^t (1+ 2 z_s)ds.
$
By Gr\"onwall's inequality, we obtain the desired estimate
with $C_{T,X}$ as specified.
\end{proof}

For any fixed constant $K_0>0$,
let
${\cal U}_{\cal F^N}^{K_0}$
consist of all $ \mathbb{R}^{n_1}$-valued processes $u^1_t\coloneqq u^1(t, \omega)$ that satisfy the conditions: (i) $u^1(t,\omega)$ is adapted to the filtration ${\mathcal F}^N_t\coloneqq\sigma(X_0^k,W_s^k,\\ W_s^0, k=1, \cdots, N, s\le t)$, (ii)
$\EE\int_0^T|u^{1} (t, \omega)|^2dt\le K_0.$

\begin{remark}
Given $\hat {\mathbf u}^{-1}$,
if $u^1(t, x_1, \cdots, x_N)$ is a Lipschitz feedback control law yielding  a unique closed-loop state process $(X_t^1, \cdots, X_t^N)$, we may identify such a control law as the process $u^1(t, X_t^1, \cdots, X_t^N)$, which is adapted to the filtration ${\mathcal F}_t^N$.
\end{remark}

Denote $\hat u_t^1= \phi(t, X_t^1, \mu_t^{-1})$ in \eqref{allXphi0}. Then by Proposition \ref{prop:ex2}, $\hat u_t^1$ belongs to $ {\cal U}_{\cal F^N}^{K_0}$ if $K_0$ is sufficiently large.

\begin{theorem}[Asymptotic person-by-person optimality]
\label{theorem:pbpo}
 Under Assumptions \ref{assum:meV}, \ref{assum:flg}, \ref{assum:meM}, \ref{assum:Lip}   and \ref{assum:X0}, for \eqref{X1u}--\eqref{Xkuhat} and any fixed $K_0>0$, we have
$$
J^{(N)}_{\rm soc} (\hat u^1(\cdot) , \hat u^2(\cdot), \cdots, \hat u^N(\cdot))\le  \inf_{u^1(\cdot)\in {\cal U}_{\cal F^N}^{K_0}}J^{(N)}_{\rm soc} (u^1(\cdot) , \hat u^2(\cdot), \cdots, \hat u^N(\cdot))+\epsilon_N,
$$
where $\epsilon_N =O(1/N)$ and $J^{(N)}_{\rm soc}$ is defined by \eqref{jsocN}.
\end{theorem}

\subsection{Proof of Theorem \ref{theorem:pbpo}}
Throughout this subsection, we suppose that all assumptions in Theorem \ref{theorem:pbpo} hold.
Using It\^o's formula
we obtain the following lemma.

\begin{lemma}\label{lemma:Uintrep}
Let the processes $(X^1_t, \cdots, X_t^N)$ be given by \eqref{X1u}--\eqref{Xkuhat}
with initial states  $(X^1_0, \cdots, X_0^N)$ at $t=0$ and $u_t^1\in{\mathcal U}_{{\mathcal F}^N}^{K_0}$ for some fixed $K_0>0$ as in Theorem \ref{theorem:pbpo} ensuring a unique strong solution to \eqref{X1u}--\eqref{Xkuhat}, and denote the derivatives $\delta_\mu U(t, x, \mu; y)$, $\delta_{\mu\mu} U(t, x, \mu; y,z)$, $\delta_\mu\overline U(t, \mu;y)$, $\delta_{\mu\mu}\overline U(t, \mu; y,z)$ for $U$ and $\overline U$ defined by \eqref{defUU} and \eqref{defUUb}, respectively. Then  we have
\begin{align}
&\EE [U(T, X_T^1, \mu_T^{-1})  - U(0, X_0^1, \mu_0^{-1} )] \label{Urep}\\
& = \EE \int_0^T\Big\{\partial_t U(t, X^1_t, \mu_t^{-1}) + U_{x}(t, X^1_t, \mu_t^{-1}) f(X_t^1, u^1_t, \mu_t^{-1}) \notag \\
&\qquad +\frac{1}{2} {\rm Tr}[ U_{xx}(t, X^1_t, \mu_t^{-1}) (\Sigma+\Sigma_0)]\notag \\
&\qquad  + \frac{1}{N-1}\sum_{j\ge 2} \partial_{y}\delta_\mu U(t, X^1_t,\mu_t^{-1}; X^j_t)f^*(t, X_t^j, \mu_t^{-j}) \notag \\
&\qquad + \frac{1}{2(N-1)}\sum_{j\ge 2 }{\rm Tr}[ \partial_{y}^2\delta_\mu U(t, X^1_t,\mu_t^{-1}; X^j_t)(\Sigma+\Sigma_0)] \notag\\
&\qquad  + \frac{1}{N-1}\sum_{j\ge 2  }{\rm Tr} [\partial_{xy}\delta_\mu U(t, X^1_t,\mu_t^{-1}; X^j_t)\Sigma_0 ]\notag\\
&\qquad  +   \frac{1}{2(N-1)^2}\sum_{j\ge 2,k\ge 2  } {\rm Tr[}\partial_{yz}
\delta_{\mu\mu}
U(t, X^1_t,\mu_t^{-1};X_t^k, X^j_t)\Sigma_0 ]\notag\\
&\qquad  +   \frac{1}{2(N-1)^2}\sum_{k\ge 2  }{\rm Tr[} \partial_{yz}
\delta_{\mu\mu}
U(t, X^1_t,\mu_t^{-1};X_t^k, X^k_t)\Sigma] \Big\}dt\notag
\end{align}
and
\begin{align}
&\EE [\overline U(T,  \mu_T^{-1})  - \overline U(0, \mu^{-1} )]\label{bUrep} \\
&\quad =\EE\int_0^T \Big\{ \partial_t \overline U(t, \mu_t^{-1})+ \frac{1}{N-1}\sum_{j\ge 2} \partial_{y}\delta_\mu \overline U(t, \mu_t^{-1}; X^j_t)f^*(t, X_t^j, \mu_t^{-j}) \notag \\
&\qquad + \frac{1}{2(N-1)}\sum_{j\ge 2 }{\rm Tr}[ \partial_{y}^2\delta_\mu \overline U(t, \mu_t^{-1}; X^j_t)(\Sigma+\Sigma_0) ]\notag \\
&\qquad+   \frac{1}{2(N-1)^2}\sum_{j\ge 2,k\ge 2  }{\rm Tr[} \partial_{yz}
\delta_{\mu\mu}
\overline U(t, \mu_t^{-1};X_t^k, X^j_t)\Sigma_0 ]\notag \\
&\qquad+   \frac{1}{2(N-1)^2}\sum_{k\ge 2  } {\rm Tr}[ \partial_{yz}
\delta_{\mu\mu}
\overline U(t,\mu_t^{-1};X_t^k, X^k_t)\Sigma]  \Big\}dt. \notag
\end{align}
\end{lemma}

\begin{proof}
For agent ${\mathcal A}_1$, by It\^o's formula involving the measure flow
$\{\mu^{-1}_t, 0\le t\le T\}$ (see e.g. \cite{CDLL19,CD18,C21}),  we have
{\allowdisplaybreaks
\begin{align*}
dU(t, X_t^1, \mu_t^{-1}) =&\ [\partial_t U(t, X^1_t, \mu_t^{-1}) + U_{x}(t, X^1_t, \mu_t^{-1}) f(X_t^1, u^1_t, \mu_t^{-1})] dt \\
& +\frac{1}{2} {\rm Tr}[ U_{xx}(t, X^1_t, \mu_t^{-1}) (\Sigma+\Sigma_0)] dt \\
&  +  U_{x}(t, X^1_t, \mu_t^{-1})(\sigma dW^1_t+ \sigma_0 dW_t^0)  \\
&+ \frac{1}{N-1}\sum_{j\ge 2} \partial_{y}\delta_\mu U(t, X^1_t,\mu_t^{-1}; X^j_t)f^*(t, X_t^j, \mu_t^{-j}) dt \\
& + \frac{1}{2(N-1)}\sum_{j\ge 2 }{\rm Tr}[\partial_{y}^2\delta_\mu U(t, X^1_t,\mu_t^{-1}; X^j_t)(\Sigma+\Sigma_0 )]dt \\
&  + \frac{1}{N-1}\sum_{j\ge 2  } \partial_{y}\delta_\mu U(t, X^1_t,\mu_t^{-1}; X^j_t)(\sigma dW^j_t + \sigma_0 dW_t^0 )   \\
&  + \frac{1}{N-1}\sum_{k\ge 2  }{\rm Tr} [\partial_{xy}\delta_\mu U(t, X^1_t,\mu_t^{-1}; X^k_t)\Sigma_0 ]dt\\
&+   \frac{1}{2(N-1)^2}\sum_{j\ge 2,k\ge 2  } \mbox{Tr}[\partial_{yz}
\delta_{\mu\mu}
U(t, X^1_t,\mu_t^{-1};X_t^k, X^j_t)\Sigma_0 ]dt\\
&+   \frac{1}{2(N-1)^2}\sum_{k\ge 2  }\mbox{Tr}[ \partial_{yz}
\delta_{\mu\mu}
U(t, X^1_t,\mu_t^{-1};X_t^k, X^k_t)\Sigma ]dt.
\end{align*}
}
Integrating  both sides on $[0,T]$ and taking expectation, we obtain \eqref{Urep}. The proof of \eqref{bUrep}  is similar and  uses  the following relation
\begin{align}
d\overline U(t, \mu_t^{-1}) = & \ \partial_t \overline U(t, \mu_t^{-1})dt+\frac{1}{N-1}\sum_{j\ge 2} \partial_{y}\delta_\mu \overline U(t, \mu_t^{-1}; X^j_t)f^*(t, X_t^j, \mu_t^{-j}) dt \\
& + \frac{1}{2(N-1)}\sum_{j\ge 2 }{\rm Tr}[ \partial_{y}^2\delta_\mu \overline U(t, \mu_t^{-1}; X^j_t)(\Sigma+\Sigma_0) ]dt \notag \\
&  + \frac{1}{N-1}\sum_{j\ge 2  } \partial_{y}\delta_\mu
\overline U(t, \mu_t^{-1}; X^j_t) (\sigma dW^j_t  +\sigma_0  dW^0_t  )   \notag \\
&+   \frac{1}{2(N-1)^2}\sum_{j\ge 2,k\ge 2  }\mbox{Tr} [\partial_{yz}
\delta_{\mu\mu}
\overline U(t, \mu_t^{-1};X_t^k, X^j_t)\Sigma_0 ]dt\notag \\
&+   \frac{1}{2(N-1)^2}\sum_{k\ge 2  } \mbox{Tr} [\partial_{yz}
\delta_{\mu\mu}
\overline U(t,\mu_t^{-1};X_t^k, X^k_t)\Sigma]dt.\notag
\end{align}
\end{proof}

To obtain the crucial lower bound in  Theorem \ref{theorem:Jlb} below, we need to slightly rewrite the integrands in \eqref{Urep} and \eqref{bUrep} to relate to  equations \eqref{MEUa} and \eqref{MEUb} for $U$ and $\overline U$, respectively.
We proceed to introduce the following error terms:
\begin{align*}
&\xi^a_1\coloneqq \frac{1}{N-1}\sum_{j\ge 2} \partial_{y}\delta_\mu U(t, X^1_t,\mu_t^{-1}; X^j_t)[f^*(t, X_t^j, \mu_t^{-j}) - f^*(t, X_t^j, \mu_t^{-1})],\\
& \xi_2^a\coloneqq
 \frac{1}{2(N-1)^2}\sum_{k\ge 2  }\mbox{Tr} [\partial_{yz}
\delta_{\mu\mu}
U(t, X^1_t,\mu_t^{-1};X_t^k, X^k_t)\Sigma],
\end{align*}
and
\begin{align*}
 \xi^b&\coloneqq  \sum_{j\ge 2} \partial_{y}\delta_\mu \overline U(t, \mu_t^{-1}; X^j_t)\Big\{f^*(t, X_t^j, \mu_t^{-j})- f^*(t, X_t^j, \mu_t^{-1})\\
  & \qquad \qquad\qquad - \frac{1}{N-1} [   \delta_\mu f^*(t, X_t^j, \mu_t^{-1}; X_t^1) - \delta_\mu f^*(t, X_t^j, \mu_t^{-1}; X_t^j)] \Big\}.
\end{align*}
Moreover,  for $2\le j\le N$, we denote
\begin{align*}
\xi^L_j&\coloneqq
L^*(t, X_t^j, \mu_t^{-j})- L^*(t, X^j_t, \mu_t^{-1})\\
&\qquad - \frac{1}{N-1}[ \delta_\mu L^*(t, X^j_t, \mu^{-1}_t; X_t^1)-\delta_\mu L^*(t, X^j_t, \mu^{-1}_t; X_t^j)],\\
\xi^g_j&\coloneqq
g(X_T^j, \mu_T^{-j})- g(X_T^j, \mu_T^{-1})\\
&\qquad- \frac{1}{N-1}[ \delta_\mu g( X^j_T, \mu^{-1}_T; X_T^1)-\delta_\mu g( X^j_T, \mu^{-1}_T; X_T^j)].
 \end{align*}

\begin{theorem}  \label{theorem:Jlb}
For \eqref{X1u}--\eqref{Xkuhat} and every control $u^1$ as in Theorem \ref{theorem:pbpo},
we  have
\begin{align}
J_{\rm soc}^{(N)}(u^1(\cdot), \hat {\mathbf u}^{-1}(\cdot))&\ge \EE[U(0, X_0^1, \mu_0^{-1}) +(N-1) \overline U(0, \mu_0^{-1} ) ]\\
&\quad+\EE \int_0^T \Big( \xi_1^a +\xi_2^a+ \xi^b+ \sum_{j=2}^N \xi^L_j\Big) dt  +\EE \sum_{j=2}^N \xi_j^g,\nonumber
\end{align}
where the equality holds if $u_t^1=\phi(t, X_t^1, \mu_t^{-1})$ for all $t\in [0,T]$ in \eqref{X1u} for agent ${\mathcal A}_1$.
\end{theorem}

\begin{proof}
Denote
\begin{align*}
\Theta^a \coloneqq \ &\partial_t U(t, X^1_t, \mu_t^{-1}) + U_{x}(t, X^1_t, \mu_t^{-1}) f(X_t^1, u^1_t, \mu_t^{-1})  \\
& +\frac{1}{2} {\rm Tr}[ U_{xx}(t, X^1_t, \mu_t^{-1}) (\Sigma +\Sigma_0) ]\\
&+ \frac{1}{N-1}\sum_{j\ge 2} \partial_{y}\delta_\mu U(t, X^1_t,\mu_t^{-1}; X^j_t)f^*(t, X_t^j, \mu_t^{-1})  \\
& + \frac{1}{2(N-1)}\sum_{j\ge 2 }{\rm Tr}[ \partial_{y}^2\delta_\mu U(t, X^1_t,\mu_t^{-1}; X^j_t)(\Sigma+\Sigma_0 )] \\
&  + \frac{1}{N-1}\sum_{j\ge 2  }{\rm Tr}[ \partial_{xy}\delta_\mu U(t, X^1_t,\mu_t^{-1}; X^j_t)\Sigma_0 ]\\
&+   \frac{1}{2(N-1)^2}\sum_{j\ge 2,k\ge 2  } {\rm Tr}[\partial_{yz}
\delta_{\mu\mu}
U(t, X^1_t,\mu_t^{-1};X_t^k, X^j_t)\Sigma_0 ], \notag
\end{align*}
which consists of the first six lines  in the integrand of \eqref{Urep} but   replaces $\mu_t^{-j}$  by $\mu_t^{-1}$ within $f^*$ on the third line.
Then from Lemma \ref{lemma:Uintrep}, we have the relation
\begin{align}
&\EE [U(T, X_T^1, \mu_T^{-1})  - U(0, X_0^1, \mu_0^{-1} )] = \EE \int_0^T (\Theta^a +\xi^a_1+\xi^a_2)(t)  dt. \notag
\end{align}

 Similarly, based on  \eqref{bUrep}, we define
\begin{align*}
\Theta^b\coloneqq
  & (N-1)\Big\{ \partial_t \overline U(t, \mu_t^{-1})+ \frac{1}{N-1}\sum_{j\ge 2} \partial_{y}\delta_\mu \overline U(t, \mu_t^{-1}; X^j_t)f^*(t, X_t^j, \mu_t^{-1})  \\
& + \frac{1}{2(N-1)}\sum_{j\ge 2 }{\rm Tr}[ \partial_{y}^2\delta_\mu \overline U(t, \mu_t^{-1}; X^j_t)(\Sigma+\Sigma_0) ]\\
&+   \frac{1}{2(N-1)^2}\sum_{j\ge 2,k\ge 2  }{\rm Tr}[ \partial_{yz}
\delta_{\mu\mu}
\overline U(t, \mu_t^{-1};X_t^k, X^j_t)\Sigma_0 ]
\Big
\}, \notag
\end{align*}
and
\begin{align*}
& \bar \Theta^b_1 \coloneqq \sum_{j\ge 2} \partial_{y}\delta_\mu \overline U(t, \mu_t^{-1}; X^j_t)[f^*(t, X_t^j, \mu_t^{-j})- f^*(t, X_t^j, \mu_t^{-1})], \\
& \Theta_2^b \coloneqq  \frac{1}{2(N-1)} \sum_{k\ge 2 }{\rm Tr}[ \partial_{yz}\delta_{\mu\mu} \overline U(t, \mu_t^{-1}; X^k_t, X^k_t)\Sigma ].
\end{align*}
Using \eqref{bUrep},   we obtain
\begin{align*}
&(N-1)\EE [\overline U(T,  \mu_T^{-1})  - \overline U(0, \mu_0^{-1} )] = \EE \int_0^T(\Theta^b +\bar\Theta^b_1 + \Theta_2^b )(t) dt.
\end{align*}
Subsequently, we obtain
\begin{align}
&\EE [U(T, X_T^1, \mu_T^{-1})+(N-1) \overline U(T,  \mu_T^{-1})] \label{5UUb} \\
&\qquad   - \EE [U(0, X_0^1, \mu_0^{-1}) +(N-1) \overline U(0, \mu_0^{-1} ) ]\notag \\
&\qquad =\EE \int_0^T (\Theta^a+\xi_1^a+\xi_2^a+\Theta^b+ \bar \Theta^b_1+ \Theta^b_2) dt.\notag
\end{align}
We need to further decompose  $\bar \Theta^b_1$ as the sum  of a main component and a small error term.
For this purpose, denote
\begin{align*}
& \Theta^{b}_1 \coloneqq \frac{1}{N-1} \sum_{j\ge 2} \partial_{y}\delta_\mu \overline U(t, \mu_t^{-1}; X^j_t) [   \delta_\mu f^*(t, X_t^j, \mu_t^{-1}; X_t^1) - \delta_\mu f^*(t, X_t^j, \mu_t^{-1}; X_t^j) ].
 \end{align*}
Then we have
$\bar \Theta_1^b=\Theta_1^b+   \xi^b   ,$
where $\xi^b$ is viewed as a  small error term.
By \eqref{MEUb}, we have
 \begin{align*}
\Theta^{b}=&  -(N-1)\langle \mu_t^{-1}(dy), L^*(t, y, \mu_t^{-1}) \rangle.
\end{align*}

For \eqref{5UUb}, denote
$$
\Theta^*\coloneqq \Theta^a+\xi_1^a+\xi_2^a+\Theta^b+ \bar \Theta^b_1+ \Theta^b_2.
$$
Now we have
\begin{align*}
  \Theta^*
=& \Theta^a+\Theta^b+  \Theta^b_1+ \Theta^b_2+ \xi_1^a+\xi_2^a+\xi^b  \\
 =&\ \Big\{ \partial_t U(t, X^1_t, \mu_t^{-1}) + U_{x}(t, X^1_t, \mu_t^{-1}) f(X_t^1, u^1_t, \mu_t^{-1})  \\
& +\frac{1}{2} {\rm Tr}[ U_{xx}(t, X^1_t, \mu_t^{-1}) (\Sigma+ \Sigma_0)] \\
&+ \frac{1}{N-1}\sum_{j\ge 2} \partial_{y}\delta_\mu U(t, X^1_t,\mu_t^{-1}; X^j_t)f^*(t, X_t^j, \mu_t^{-1})  \\
& + \frac{1}{2(N-1)}\sum_{j\ge 2 }{\rm Tr}[ \partial_{y}^2\delta_\mu U(t, X^1_t,\mu_t^{-1}; X^j_t)(\Sigma+\Sigma_0) ] \\
&  + \frac{1}{N-1}\sum_{j\ge 2  }{\rm Tr}[ \partial_{xy}\delta_\mu U(t, X^1_t,\mu_t^{-1}; X^j_t)\Sigma_0 ]\\
&+   \frac{1}{2(N-1)^2}\sum_{j\ge 2,k\ge 2  } {\rm Tr}[\partial_{yz}
\delta_{\mu\mu}
U(t, X^1_t,\mu_t^{-1};X_t^k, X^j_t)\Sigma_0 ]\\
& + \frac{1}{N-1} \sum_{j\ge 2} \partial_{y}\delta_\mu \overline U(t, \mu_t^{-1}; X^j_t) [   \delta_\mu f^*(t, X_t^j, \mu_t^{-1}; X_t^1)\\
&\qquad\qquad\qquad\qquad\qquad\qquad\qquad - \delta_\mu f^*(t, X_t^j, \mu_t^{-1}; X_t^j) ]\\
& + \frac{1}{2(N-1)}\sum_{k\ge 2  } {\rm Tr}[ \partial_{yz}
\delta_{\mu\mu}
\overline U(t,\mu_t^{-1};X_t^k, X^k_t)\Sigma]\Big\} \qquad (\eqqcolon\Upsilon)
 \notag \\
 & -(N-1)\langle \mu_t^{-1}(dy), L^*(t, y, \mu_t^{-1}) \rangle + \xi_1^a+\xi_2^a +  \xi^b\\
  =& \Upsilon -(N-1)\langle \mu_t^{-1}(dy), L^*(t, y, \mu_t^{-1}) \rangle +\xi_1^a+\xi_2^a   + \xi^b.
\end{align*} 
Now  Lemma \ref{lemma:optimizer} implies, after setting $x=X_t^1$, $u=u_t^1$ and $\mu=\mu_t^{-1}$ in \eqref{MEU}, that
\begin{align}
&\Upsilon +  L(X_t^1, u_t^1, \mu_t^{-1}) \label{upsge} \\
&\quad+  \langle \mu_t^{-1}(dy) , \   \delta_\mu L^* (t,y,\mu_t^{-1}; X_t^1) - \delta_\mu L^*(t,y,\mu_t^{-1}; y) \rangle\ge0,
\notag
\end{align}
where the equality holds if
$u_t^1=\phi(t, X_t^1, \mu_t^{-1})$  in both  $\Upsilon $ and $L$.
It follows from \eqref{upsge} that
\begin{align}
 \Theta^*
\ge &- L(X_t^1, u_t^1, \mu_t^{-1})-  \langle \mu_t^{-1}(dy) , \   \delta_\mu L^* (t,y,\mu_t^{-1}; X_t^1) - \delta_\mu L^*(t,y,\mu_t^{-1}; y) \rangle \notag \\
& -(N-1)\langle \mu_t^{-1}(dy), L^*(t, y, \mu_t^{-1}) \rangle
  + \xi_1^a +\xi_2^a+ \xi^b. \label{Thexi}
\end{align}
We have
\begin{align}
\sum_{j=2}^N L^*(t, X_t^j, \mu_t^{-j})&=
 (N-1)\langle \mu_t^{-1}(dy), L^*(t, y, \mu_t^{-1}) \rangle\label{LLxi}  \\
&\hskip -1cm  +\langle \mu_t^{-1}(dy) , \   \delta_\mu L^* (t,y,\mu_t^{-1}; X_t^1) - \delta_\mu L^*(t,y,\mu_t^{-1}; y) \rangle + \sum_{j=2}^N \xi^L_j \notag
\end{align}
and, in view of the terminal conditions of $U$ and $\overline U$,
\begin{align}
\sum_{j=1}^N g(X_T^j, \mu_T^{-j})=   U(T, X_T^1, \mu_T^{-1})+(N-1) \overline U(T,  \mu_T^{-1}) +\sum_{j=2}^N \xi^g_j.\label{gxi}
\end{align}

Combining \eqref{Thexi} and \eqref{LLxi} yields
\begin{align}
\Theta^*(t) &\ge - L(X_t^1, u_t^1, \mu_t^{-1})-\sum_{j=2}^NL^*(t, X_t^j, \mu_t^{-j}) +\xi_1^a+ \xi_2^a + \xi^b + \sum_{j=2}^N \xi^L_j.\label{TheLxL}
\end{align}
By \eqref{5UUb} and \eqref{TheLxL}, we have
\begin{align}
&\EE [U(T, X_T^1, \mu_T^{-1})+(N-1) \overline U(T,  \mu_T^{-1})]\label{ULxi} \\
& \qquad -\EE [U(0, X_0^1, \mu_0^{-1}) +(N-1) \overline U(0, \mu_0^{-1} ) ]\notag\\
&\qquad \ge -\EE \int_0^T\Big[ L(X_t^1, u_t^1, \mu_t^{-1})+\sum_{j=2}^NL^*(t, X_t^j, \mu_t^{-j})\Big]dt\notag \\
&\qquad\qquad+\EE \int_0^T\Big
(\xi_1^a +\xi_2^a+ \xi^b+ \sum_{j=2}^N \xi^L_j\Big) dt
.\notag
\end{align}
By \eqref{gxi} and \eqref{ULxi},
 we have
\begin{align*}
&\EE \int_0^T \Big[ L(X_t^1, u_t^1, \mu_t^{-1})+\sum_{j=2}^NL^*(t, X_t^j, \mu_t^{-j})\Big] dt +\EE \sum_{j=1}^N g(X_T^j, \mu_{T}^{-j})\\
& \qquad \ge\EE [U(0, X_0^1, \mu_0^{-1}) +(N-1) \overline U(0, \mu_0^{-1} ) ]\\
&\qquad\qquad+\EE \int_0^T \Big( \xi_1^a +\xi_2^a + \xi^b+ \sum_{j=2}^N \xi^L_j
 \Big) dt +\EE   \sum_{j=2}^N \xi^g_j,
\end{align*}
where the equality holds if $u_t^1=\phi(t, X_t^1, \mu_t^{-1})$ for all $t$ in \eqref{X1u} for agent ${\mathcal A}_1$.
\end{proof}

\begin{lemma}\label{lemma:xi}
 For \eqref{X1u}--\eqref{Xkuhat},
we have
\begin{align}
& |\xi^a_1 |, |\xi^b|\le \frac{C}{N-1} (1+|X_t^1|^2 +\langle|y|^2\rangle_{\mu_t^{-1}}  ),   \label{xia1b}  \\
& |\xi^a_2 |\le \frac{C}{N-1} ,  \label{xia2}   \\
&|\xi_j^L|\le \frac{C}{|N-1|^2} (1+|X^1_t|^2+|X_t^j|^2+\langle|y|^2\rangle_{\mu_t^{-1}}  ),\label{xiLj}  \\
&|\xi_j^g|\le \frac{C}{|N-1|^2} (1+|X^1_T|^2+|X_T^j|^2+\langle|y|^2\rangle_{\mu_T^{-1}} ).\label{xigj}
\end{align}

\end{lemma}

\begin{proof}
By \eqref{defUU} we calculate $\partial_y \delta_\mu U(t, x, \mu;y)$
and use Assumption \ref{assum:meV} to  obtain
\begin{align}
|\partial_y\delta_\mu U(t, x, \mu;y)|\le C(1+|x|+|y|+\langle|y|\rangle_\mu).
\label{ymuUC}
\end{align}
Next for $\mu^{-j}\coloneqq \frac{1}{N-1}\sum_{k\ne j, k=1}^N \delta_{x^k}$, we have
\begin{align}
&f^*(t, x^j, \mu^{-j})-f^*(t, x^j, \mu^{-1}) \label{ffj1}  \\
&\qquad =
\int_0^1 \int_y\delta_\mu f^*(t, x^j, \mu^{-1} +s(\mu^{-j}-\mu^{-1}) ;y) (\mu^{-j}-\mu^{-1}) (dy)ds. \notag
\end{align}
Denote $\hat \mu_s\coloneqq  \mu^{-1} +s(\mu^{-j}-\mu^{-1})$
and $\Delta^{j,1} \mu\coloneqq \mu^{-j}-\mu^{-1}  $.
  Assumption \ref{assum:flg} implies
\begin{align}
\Big|\int_y \delta_\mu f^*(t, x^j, \hat \mu_s;y)  \Delta^{j,1} \mu (dy)\Big| &=\frac{1}{N-1} |\delta_\mu f^*(t, x^j, \hat \mu_s;x^1)- \delta_\mu f^*(t, x^j, \hat \mu_s;x^j) | \notag \\
&\le \frac{C}{N-1} \Big(1+|x^j|+|x^1|+\frac{1}{N-1}\sum_{k\ge 2} |x^k|\Big).\notag
\end{align}
Subsequently, by \eqref{ymuUC} and \eqref{ffj1}, we have
\begin{align}
|\xi_1^a|&\le \frac{C}{N-1}\Big(1+|X_t^1|^2+\frac{1}{N-1}\sum_{j\ge 2}|X_t^j|^2+ \langle |y|\rangle^2_{\mu_t^{-1}}\Big)\\
&\le \frac{C}{N-1}(1+|X_t^1|^2+ \langle |y|^2\rangle_{\mu_t^{-1}}).\notag
\end{align}

Next, we use the growth conditions of $V$,  as specified by $ {\cal C}_V$, and the semi-symmetry property \eqref{ssp} to establish
\begin{align}
|\partial_{yz}
\delta_{\mu\mu}
U(t, X^1_t,\mu_t^{-1};X_t^k, X^k_t)|\le C,
\end{align}
which implies the bound of $|\xi_2^a|$ in  \eqref{xia2}.
To get the bound of $|\xi^b|$, we write

\begin{align}
f^*(t, x^j, \mu^{-j})=& f^*(t, x^j, \mu^{-1}) +
\langle \Delta^{j,1} \mu(dy)  , \ \delta_\mu f^*(t, x^j, \mu^{-1} ; y)     \rangle \label{fTaylor} \\
& + \int_0^1\int_0^1\int_{y,z} s \delta_{\mu\mu} f^*(t,x^j, \mu^{-1}+ s\tau (\mu^{-j}-\mu^{-1}); y,z) \notag \\
&\qquad\qquad \qquad \cdot  \Delta^{j,1} \mu (dy)
 \Delta^{j,1} \mu  (dz)d\tau  ds.\notag
\end{align}
Denote $\hat\mu_{s,\tau}\coloneqq \mu^{-1}+ s\tau (\mu^{-j}-\mu^{-1}) $. Then
\begin{align*}
&\int_{y,z}\delta_{\mu\mu} f^*(t,x^j,\hat \mu_{s,\tau};y,z)\Delta^{j,1}\mu (dy)\Delta^{j,1}\mu (dz) \\
&=  \frac{1}{N-1} \int_z [\delta_{\mu\mu} f^*(t,x^j,\hat \mu_{s,\tau};x^1,z)- \delta_{\mu\mu} f^*(t,x^j,\hat \mu_{s,\tau};x^j,z) ]\Delta^{j,1}\mu (dz)\\
&=\frac{1}{(N-1)^2}\Big\{ [\delta_{\mu\mu} f^*(t,x^j,\hat \mu_{s,\tau};x^1,x^1)- \delta_{\mu\mu} f^*(t,x^j,\hat \mu_{s,\tau};x^j,x^1) ]\\
&\qquad\qquad\qquad- [\delta_{\mu\mu} f^*(t,x^j,\hat \mu_{s,\tau};x^1,x^j)- \delta_{\mu\mu} f^*(t,x^j,\hat \mu_{s,\tau};x^j,x^j) ]\Big\}.
\end{align*}
Consequently, under Assumption \ref{assum:flg}, we have
\begin{align}
&\Big|\int_{y,z}\delta_{\mu\mu} f^*(t,x^j,\hat \mu_{s,\tau};y,z)\Delta^{j,1}\mu (dy)\Delta^{j,1}\mu (dz)\Big| \label{fsdede}  \\
&\qquad \le \frac{C}{(N-1)^2} (1+|x^1|+|x^j| + \langle|y|\rangle_{\mu^{-1}} ).\notag
\end{align}
We  use Assumption \ref{assum:meV} to obtain
\begin{align}
|\partial_y \delta_\mu \overline U(t, \mu; y)|&\le
|\partial_y V(t, y,\mu)|+
|\langle \mu(dw), \partial_y\delta_\mu V(t, w, \mu; y)  \rangle | \label{ymubU}  \\
&\le C(1+ |y|+\langle |y|\rangle_\mu).\notag
\end{align}
By \eqref{fTaylor}, \eqref{fsdede} and \eqref{ymubU}, we obtain the bound of $|\xi^b|$ in \eqref{xia1b}.
Finally, we follow the method in \eqref{fTaylor} to obtain \eqref{xiLj} and \eqref{xigj}.
\end{proof}

\begin{lemma}\label{lemma:xm2}
 There exists a constant $C_{K_0}$ depending on $K_0$ such that
 for \eqref{X1u}--\eqref{Xkuhat},
we have
\begin{align}\label{ex2k}
\sup_N \sup_{u^1\in {\cal U}_{\cal F^N}^{K_0}} \max_{1\le i\le N, 0\le t\le T} \EE |X_t^i|^2\le C_{K_0}.
\end{align}
\end{lemma}

\begin{proof}
By It\^o's formula, we have
\begin{align*}
|X_t^1|^2=\ & |X_0^1|^2+ \int_0^t \Big[X_s^{1T}f(X_s^1,u^1_s, \mu^{-1}_s) +  \mbox{Tr} (\Sigma +\Sigma_0)\Big]ds\\
& +
\int_0^t X_s^{1T} (\sigma dW_s^1+\sigma_0dW^0_s).
\end{align*}
We have
\begin{align*}
|X_s^{1T}f(X_s^1,u^1_s, \mu^{-1}_s)| &\le C|X_s^1|\cdot
\Big(1+|X_s^1|+|u_s^1|+\frac{1}{N-1}\sum_{k=2}^N |X_s^k|\Big)\\
&\le C\Big(1+|X_s^1|^2+ |u_s^1|^2 + \frac{1}{N-1} \sum_{k=2}^N |X_s^k|^2\Big).
\end{align*}
Then it follows that
\begin{align}
\EE|X_t^1|^2 \le \EE|X_0^1|^2  + C\int_0^t \Big(1+ \EE|u_s^1|^2+\max_{1\le k\le N} \EE|X_s^k|^2\Big) ds.
\end{align}
Similarly,
for $k\ge 2$, we use Assumption (A1) and linear growth of
$\phi$ in Definition \ref{def:vsol} to obtain
\begin{align}
\EE|X_t^k|^2 \le \EE|X_0^k|^2  + C\int_0^t \Big(1+\max_{1\le k\le N} \EE|X_s^k|^2\Big) ds.
\end{align}
Next we follow the proof of Proposition \ref{prop:ex2} to  use Gr\"onwall's lemma to obtain  \eqref{ex2k}.
\end{proof}

{\bf Proof of Theorem \ref{theorem:pbpo}}.
By Theorem \ref{theorem:Jlb}, Lemmas \ref{lemma:xi} and \ref{lemma:xm2}, we obtain
\begin{align*}
&J^{(N)}_{\rm soc} (\hat u^1, \hat {\mathbf u}^{-1})= \EE [U(0, X_0^1, \mu_0^{-1})+(N-1) \overline U(0,  \mu_0^{-1})]+ O(1/N), \\
& J^{(N)}_{\rm soc} ( u^1, \hat {\mathbf u}^{-1})\ge \EE [U(0, X_0^1, \mu_0^{-1})+(N-1) \overline U(0,  \mu_0^{-1})]- O(1/N)
\end{align*}
for all $u^1(\cdot)\in  {\cal U}_{\cal F^N}^{K_0}$. Then Theorem
\ref{theorem:pbpo} follows readily.  \qed

\section{Explicit solutions in  LQ models}
\label{sec:lq}

This section uses  LQ mean field social optimization models
    to illustrate explicit solutions of the master equations of $V$, $M$ and $U$.
The individual agent has the dynamics
\begin{align}\label{xilq}
dX^i_t=(AX_t^i+Bu_t^i+G X_t^{(-i)}) dt
+D dW_t^i+D_0 dW_t^0,\quad 1\le i\le N,
\end{align}
where the initial states are independent with $\EE |X_0^i|^2<\infty$.
The individual  cost is given by
\begin{align}\label{jilq}
J_i=\EE\int_0^T
(|X_t^i-\Gamma X_t^{(-i)}-\eta|_Q^2
+u_t^{iT} R u_t^i) dt+\EE |X_T^i-\Gamma_f X_T^{(-i)}-\eta_f|_{Q_f}^2,
\end{align}
where $X_t^{(-i)}\coloneqq \frac{1}{N-1}\sum_{j=1, j\ne i}^N X_t^j$ and $|x|_Q^2 \coloneqq x^TQx$.
We have symmetric matrices $Q\ge 0$ and $R>0$.
Denote the social cost $J_{\rm soc}^{(N)}=\sum_{i=1}^N J_i$.

\subsection{The master equation of the value function} \label{sec:sub:lqme}
Denote $\bar x \coloneqq \langle  \mu(dy),y \rangle\in \mathbb{R}^n $ as a function of $\mu$.
For the model \eqref{xilq}--\eqref{jilq}, corresponding to \eqref{optu_noCN} we determine
\begin{align*}
\Phi =\   &  V_x (t, x, \mu)(Ax+Bu^i+G\bar x)+   |x-\Gamma \bar x-\eta|_Q^2 + u^{iT}R u^i\\
&+\langle \mu(dy), \partial_{x}\delta_\mu V (t, y, \mu; x) (Ax+Bu^i+G\bar x)    \rangle.
\end{align*}
The minimizer of $\Phi$ takes the form
\begin{align}
&\hat u^i =\phi(t,x, \mu)= -\frac{1}{2} R^{-1}B^T \big[ V_x^T(t,x,\mu)+\langle \mu(dy),  \partial^T_{x}\delta_\mu V(t,y,\mu; x)\rangle \big].\nonumber
\end{align}
The master equation \eqref{MEVwoCN} becomes
\begin{align*}
&\hskip -0.5cm-\partial_t V(t, x, \mu)\\
=\ &  V_x(t, x, \mu) (Ax+G\bar x)
 +\frac{1}{2} {\rm Tr} [ V_{xx}(t,x,\mu) (DD^T+D_0D_0^T)]\\
& + |x-\Gamma \bar x-\eta|_Q^2 -\frac{1}{4}  V_{x}BR^{-1}B^T  V_{x}^T\\
&+\frac{1}{4} \langle \mu(dy),   \partial_{x}\delta_\mu V(t,y,\mu; x)\rangle \cdot BR^{-1}B^T
\cdot\langle\mu(dy),  \partial^T_{x}\delta_\mu V(t,y,\mu; x)\rangle  \\
& +\langle \mu(dy),   \partial_y\delta_\mu V(t,x,\mu; y) [Ay+B\phi(t,y,\mu)+G\bar x]\rangle \\
&+\frac{1}{2}\langle \mu(dy), {\rm Tr} [\partial_{y}^2 \delta_\mu V(t,x,\mu; y)
(DD^T+D_0D_0^T) ]\rangle \\
& +\langle \mu(dy),  {\rm Tr} [ \partial_{xy}\delta_\mu V(t,x, \mu;y) D_0D_0^T ]\rangle \notag  \\
&+\frac{1}{2}\langle \mu^{\otimes2}(dydz), {\rm Tr}[\partial_{yz}\delta_{\mu\mu} V (t, x, \mu; y,z)
D_0D_0^T]\rangle,\notag
\end{align*}
where $V(T, x,\mu)=|x-\Gamma_f \bar x -\eta_f|_{Q_f}^2$.

We take the ansatz
\begin{align}
V(t,x,\mu) =\ & x^{T} P(t) x
+\bar x^T \Lambda(t) \bar x + 2 x^{T} H(t) \bar x  \label{Vqsol}   \\
& +2 x^T S(t) +2\bar x^{T} \theta(t)+r(t), \nonumber
\end{align}
where the functions  $P(t)$, $\Lambda(t)$, $\cdots$, $r(t)$  are to be determined, with $P(t)$ and $\Lambda (t)$ being symmetric.
We calculate
\begin{align}
& \delta_{\mu}V(t,x,\mu;y)= 2\bar x^T\Lambda y+2x^THy+2y^{T}\theta+\chi(t,x,\mu) ,  \notag \\
& \delta_{\mu\mu} V(t,x, \mu;y,z)=2z^{T}\Lambda y+\delta_\mu \chi(t,x,\mu;z)  +\chi_1(t,x,y,\mu),  \notag
\end{align}
and derive the ordinary differential equation  (ODE) system:
\begin{align}
& \dot P = -A^T P -PA +PBR^{-1} B^T P -Q, \label{odeP}  \\
& \dot \Lambda  = -(\Lambda +H) BR^{-1}B^T (\Lambda+H^{T}) \label{odeLam} \\
 &\quad -\Lambda [A+G-BR^{-1} B^T(P+\Lambda+ H+H^T)]\notag \\
 & \quad -[A+G-BR^{-1} B^T(P+\Lambda +H+H^T)]^T \Lambda\notag \\
 & \quad +H^TBR^{-1} B^T H- H^TG - G^TH  -\Gamma^T Q\Gamma ,\notag \\
& \dot H  = -A^T H +P BR^{-1} B^T H  \label{odeH}  \\
& \quad -H [A+G-BR^{-1} B^T(P+\Lambda +H+H^T)]-PG +Q\Gamma  ,\notag \\
& \dot S = -A^T S+ PBR^{-1} B^TS+H BR^{-1} B^T(S+\theta) +Q\eta,\label{odeS} \\
& \dot \theta = H^T BR^{-1} B^TS +\Lambda    BR^{-1} B^T(S+\theta) \label{odetheta} \\
& \quad -[A+G-BR^{-1} B^T(P+\Lambda+ H+H^T)]^T\theta\notag \\
& \quad  -(\Lambda+H ) BR^{-1} B^T\theta-G^TS -\Gamma^T Q \eta ,\notag \\
& \dot r = S^T BR^{-1} B^T S
+ \theta^T BR^{-1}B^T (2S+\theta) \label{oder} \\
&\quad -{\rm Tr}[(P+\Lambda+H+H^T)(D_0D_0^T)]-{\rm Tr} (PDD^T)   -\eta^T Q\eta,\notag
 \end{align}
where the terminal conditions are
\begin{align*}
&P(T)= Q_f, \quad \Lambda(T)= \Gamma_f^T  Q_f\Gamma_f , \quad  H(T)=-Q_f \Gamma_f,\\
& S(T)=-Q_f\eta_f,   \quad   \theta(T)=\Gamma_f Q_f\eta_f ,\quad    r(T)= \eta_f^T Q_f \eta_f.
\end{align*}
 We determine the  control law
 \begin{align}
 \phi(t,x,\mu) =&
 -R^{-1} B^T[Px +(\Lambda+H+H^T )
 \bar x +S+\theta ].  \label{meui}
 \end{align}

\begin{theorem}\label{theorem:lqv}
The ODE system \eqref{odeP}--\eqref{oder} has a unique solution
$(P,\Lambda, \cdots, r)$ on $[0,T]$.
\end{theorem}

\begin{proof}
We first obtain a unique solution $P$ on $[0,T]$.
Denote $Z\coloneqq P+\Lambda +H+H^T$. By the ODEs of $(P, \Lambda, H)$, we can show that $Z$ satisfies
the following equation
\begin{align}
\dot Z=-(A+G)^T Z-Z(A+G) +ZBR^{-1} B^TZ -(I-\Gamma)^TQ (I-\Gamma), \label{odeZP}
\end{align}
where $Z(T) =  (I-\Gamma_f^T) Q_f(I- \Gamma_f)$. We start by solving the standard Riccati equation \eqref{odeZP}  to get a unique solution $Z(t)$ on $[0,T]$.
 Setting $P+\Lambda+H+H^T$  in \eqref{odeH} as $Z$, we determine $H$ using a linear equation, and further obtain $\Lambda =Z-P-H-H^T $.
 Subsequently, we solve two linear equations  to obtain $(S, \theta)$. Next, $r$ is determined from a linear equation.

The above calculation gives a solution of \eqref{odeP}--\eqref{oder}   on $[0,T]$, which is clearly unique by the local Lipschitz continuity  of the vector field of the ODE system.
\end{proof}

\subsection{The equation of $M$}

Given the control law $\phi$ in \eqref{meui},
 we determine
\begin{align*}
f^*(t, x,\mu)=\ & (A-BR^{-1} B^T P)x +\widehat G \bar x  - BR^{-1} B^T (S+\theta),
\end{align*}
where
\begin{align} \label{defwhg}
\widehat G(t)\coloneqq G- BR^{-1} B^T (\Lambda(t) +H(t)+H^T(t)).
\end{align}
Now master equation \eqref{meG2} reduces to the following form:
\begin{align}
0=&\ \partial_t M(t, \mu)+
\langle  \mu(dy) ,  [\partial_y \delta_\mu  M(t, \mu; y)] f^*(t, y, \mu)     \rangle \label{meG2a}\\
& +\frac12 \langle \mu(dy), \mbox{Tr} [\partial_{y}^2\delta_\mu M (t,\mu;y) (DD^T +D_0D_0^T)]\rangle \nonumber \\
& +\frac12 \langle \mu^{\otimes2}(dydz), \mbox{Tr}[ \partial_{yz}\delta_{\mu\mu} M (t,\mu;y,z) D_0D_0^T]\rangle \nonumber \\
& - \langle \mu^{\otimes2}(dydz), [\partial_{y}\delta_\mu V (t,z,\mu;y)]\cdot [\delta_\mu f^*(t,y,\mu;y)-\delta_\mu f^*(t,y,\mu;z) ]\rangle \nonumber \\
&+   \mbox{Tr}( \Lambda DD^T),  \nonumber\\
&\hskip -0.5cm  M(T,\mu)=0, \quad (t,\mu)\in[0,T]\times
 {\cal P}_2(\mathbb{R}^n).
\end{align}
We take the ansatz
\begin{align}
M(t,\mu)= \langle \mu, y^T \Pi_1^o(t) y\rangle +\bar x^T \Pi_2^o(t) \bar x +2\bar x^T \theta^o(t) +r^o(t),
\end{align}
where $ \Pi_1^o(t) $ and $\Pi_2^o(t)$ are symmetric.
This gives
\begin{align*}
&\delta_\mu M(t, \mu;y)=   y^T \Pi_1^o(t) y +2\bar x^T \Pi_2^o(t) y
+2y^T \theta^o(t)+\chi(t,\mu),\\
&\delta_{\mu\mu}M(t, \mu; y,z)= 2z^T \Pi_2^o y+ \delta_\mu \chi(t,\mu;z)+ \chi_1(t,y,\mu).
\end{align*}

We derive the following linear ODE system:
\begin{align*}
\dot\Pi_1^o =\ &-\Pi_1^o (A-BR^{-1}B^T P) -(A-BR^{-1}B^T P)^T \Pi_1^{o}  - (H\widehat G +\widehat G^TH^T),\\
\dot \Pi_2^o =\ & - \Pi_2^o (A+\widehat G-BR^{-1}B^TP) - (A+\widehat G-BR^{-1}B^TP)^T \Pi_2^o\\
& - \Pi_1^o\widehat G - \widehat G^T \Pi_1^o +H\widehat G +\widehat G^T
H^T, \\
\dot \theta^o =\ &  - (A+\widehat G-  BR^{-1}B^TP)^T \theta^o+
(\Pi_1^o+\Pi_2^o)BR^{-1}B^T (S+\theta),\\
 \dot r^o =\ &  2\theta^{oT} BR^{-1}B^T (S+\theta) - {\rm Tr}[ (\Pi_1^o+\Lambda) DD^T + (\Pi_1^o+\Pi_2^o)(D_0D_0^T)] ,
\end{align*}
with the terminal conditions
$\Pi_1^o(T)=\Pi_2^o(T)=0$, $ \theta^o(T)=0$ and $  r^o(T)=0$, where
the functions $(P,\Lambda, H, \widehat G, S, \theta )$ have been determined from Theorem \ref{theorem:lqv} and \eqref{defwhg}.
The next proposition is obvious.
\begin{proposition}\label{prop:M}
The ODE system of $(\Pi_1^o,\Pi_2^o, \theta^o, r^o)$ has a unique solution on
$[0,T]$.
\end{proposition}

\begin{remark}
In view of  Theorem \ref{theorem:lqv} and Proposition \ref{prop:M},
for the model \eqref{xilq}--\eqref{jilq}, we may verify Assumptions
\ref{assum:meV}, \ref{assum:flg}, \ref{assum:meM}, and \ref{assum:Lip}.
\end{remark}

\subsection{The solution of $ U$}

We use \eqref{defUU} to construct a solution of $U$
in the linear quadratic case  and obtain
\begin{align}
U(t,x,\mu)=\ & x^TPx+ \langle \mu, y^T (\Pi_1^o -H-H^T) y  \rangle +\bar x^T (\Pi_2^o - \Lambda)\bar x \label{ULQ} \\
&+ 2x^T (\Lambda+H+H^T)\bar x +2x^T(S+\theta)+2\bar x^T \theta^o + r+r^o, \notag
\end{align}
where $ \bar x \coloneqq \langle y\rangle_\mu   $.
On the other hand, we may also directly solve \eqref{MEUa} by looking for quadratic solutions of the following form
\begin{align}
 U(t, x, \mu) & =   x^T \Pi_1^\diamond(t) x + \langle \mu, y^T \Pi_2^\diamond(t) y\rangle+
 \bar x^T \Pi_3^\diamond(t) \bar x  \label{ULQ2} \\
& \quad + 2x^T \Pi_4^\diamond(t) \bar x+ 2 x^T S^\diamond(t) +  2\bar x^T \theta^\diamond(t) +r^\diamond(t),\notag
\end{align}
where $\Pi_1^\diamond(t)$, $\Pi_2^\diamond(t)$ and $\Pi_3^\diamond(t) $
 are symmetric matrix functions of $t$.

\begin{theorem}
For the  model \eqref{xilq}--\eqref{jilq}, the master equation \eqref{MEUa} has a unique  solution $U$ within the class of quadratic solutions, which coincides with the representation \eqref{ULQ}.
\end{theorem}

\begin{proof}
By use of \eqref{ULQ2}, we derive an ODE system for $(\Pi_1^\diamond, \Pi_2^\diamond, \cdots, r^\diamond)$. We further show solvability  the ODEs one by one, in the following order to take
$\Pi^\diamond_1 =P$, $\Pi_2^\diamond =\Pi_1^o -H-H^T$,  $\Pi_4^\diamond=\Lambda+H+H^T  $, $ \Pi_3^\diamond=\Pi_2^o-\Lambda $,  $S^\diamond = S+\theta  $, $ \theta^\diamond=\theta^o$,
and $r^\diamond = r+r^o $. The solution is clearly unique by the local Lipschitz continuity property of the vector field of the ODE system of $(\Pi_1^\diamond,
\Pi_2^\diamond, \cdots, r^\diamond  )$.
\end{proof}

\subsection{An example from systemic risk}
\label{sec:sr}

In the system of inter-bank lending and borrowing,
the state processes of $N$ banks,  as the log-capitalization defined in \cite{CFS15}, have dynamics
\[
dX^i_t=u_t^idt+\sigma(\sqrt{1-\rho^2}dW^i_t +\rho dW_t^0),\quad 0\le \rho \le 1,
\]
with the initial value $X^{i}_0$,
where $W_t^0$, $W^1_t$, $\cdots$, $W_t^N$ are independent standard Brownian motions.
The $N$ banks obtain the social optimal strategy through minimizing
\[
J^{(N)}_{\rm soc}({\mathbf u}(\cdot))=\sum_{i=1}^NJ_i({\mathbf u}(\cdot)) ,\quad {\mathbf u}\coloneqq(u^1, \cdots, u^N),
\]
where
\[
J_i({\mathbf u}(\cdot))=\EE\int_0^T L(X_t^i,u_t^i  , X_t^{(-i)})dt+\EE g(X^i_T, X_{T}^{(-i)})
\]
with $X_t^{(-i)}\coloneqq \frac{1}{N-1}\sum_{j\ne i} X_t^j $, running cost
\[
L(x^i ,u^i , z)= (u^i)^2+2qu^i(x^i-z)
+{\epsilon_0}(x^i-z)^2,
\quad q^2\leq\epsilon_0,
\]
and  terminal cost
$
g(x^i,z)={c}(x^{i}-z )^2.
$
The parameters $q$, $\epsilon_0$ and $c$ are positive with $q^2\leq\epsilon_0$  to ensure convexity  of the running cost $L(x^i, u^i,z)$.
See \cite{CFS15} for an interpretation of $q$ in terms of incentive for lending/borrowing.

\subsubsection{The direct solution}

Denote
${\mathbf x}=(x^1,\cdots , x^N)^T$. The value function   is

\begin{eqnarray}
\nonumber&&U^N_{\rm soc }(t,{\mathbf x})=\inf_{{\mathbf u(\cdot)}}\sum_{i=1}^N\EE\left[\int_t^T L( X^{i}_s ,u^i_s ,  X_s^{(-i)})ds+g(X^{i}_T,X_T^{(-i)})\right].
\end{eqnarray}
where $(t,{\mathbf x})\in [0,T]\times{\mathbb R}^N$.
By solving the HJB equation of $U_{\rm soc}^N$, we obtain
\begin{align}
U^N_{\rm soc }(t,{\mathbf x})={\mathbf  x}^T {\mathbf P}(t) {\mathbf x} + r(t), \notag
\end{align}
where ${\mathbf P}(t)$ has $N$ diagonal entries $\pi_1$ and off-diagonal entries  all equal to $\pi_2$, with
 \begin{align*}
 &0=\dot \pi_1 - (\pi_1+q)^2  -\Big(\pi_2 -\frac{q}{N-1}\Big)^2(N-1) + \frac{\epsilon_0 N}{N-1},\\
&0=\dot \pi_2 - 2(\pi_1 +q)\Big(\pi_2-\frac{q}{N-1}\Big) -\Big(\pi_2-\frac{q}{N-1}\Big)^2(N-2)- \frac{\epsilon_0 N}{(N-1)^2} ,
 \end{align*}
 where $\pi_1(T)= cN/(N-1)$ and $\pi_2(T)=-cN/(N-1)^2$.
The optimal control law of the $i$-th bank is
\begin{align}\label{sroc}
\hat u^{N,i}(t, {\mathbf x})=- (\pi_1+q) x^i -\Big(\pi_2  - \frac{q}{N-1} \Big)\sum_{k\ne i, k=1}^N x^k.
\end{align}
Denote the ODE
\begin{align}
0=\dot P_d -(P_d+q)^2+\epsilon_0, \quad P_d(T)= c,
\end{align}
which has a unique solution on $[0,T]$. We can further show $\sup_{0\le t\le T}|\pi_1(t)-P_d(t)|=O(1/N)$. By checking the  ODE of $(N-1) \pi_2$, we obtain  $\sup_{0\le t\le T}|(N-1)\pi_2(t)+P_d(t)|=O(1/N)$.
As $N\to \infty$, the optimal control law \eqref{sroc} takes the limiting form
\begin{equation}\label{PO_systemic}
\hat u^i(t,x^i,\mu)=[P_d(t)+q](\bar x-x^i),
\end{equation}
where $\bar x\coloneqq \langle \mu, y \rangle$.
Comparing (\ref{PO_systemic}) with the open-loop and closed-loop Nash equilibria proposed in \cite{CFS15}, the social optimum has the same solution as the mean field game has in the limit with $N\rightarrow\infty$.

\subsubsection{The master equation-based control }

By   \eqref{optu_noCN}, we obtain the minimizer
\begin{align}
\hat u^i= -\frac12V_x (t,x,\mu)-q(x-\bar x) -\frac12
\langle \mu(dy) , \partial_x \delta_\mu V(t,y,\mu; x)   \rangle , \label{srmeu}
\end{align}
where $\bar x\coloneqq \langle \mu, y \rangle$.
We take the ansatz $V(t, x, \mu)=P(t) x^2 +\Lambda(t) \bar x^2+2H(t) x\bar x+r(t)$ and  derive
\begin{align*}
& 0= \dot P -(P+q)^2+\epsilon_0 ,  \\
& 0= \dot \Lambda +(\Lambda +H)^2-(H-q)^2  -2\Lambda (P+\Lambda +2H)+\epsilon_0 , \\
& 0= \dot H -(P+q)(H-q)-H(P+\Lambda +2H) -\epsilon_0   ,\\
&0 =\dot r  +\sigma^2 P+ \sigma^2 \rho^2 (\Lambda +2H ),
\end{align*}
where $P(T)=c$, $\Lambda (T)=c$, $H(T)= -c$, $r(T)=0$.
Clearly, $P=P_d$ on $[0,T]$.
Denote $Z_1\coloneqq \Lambda +2H$, and we use the equations of $\dot \Lambda$ and $\dot H$ to write the ODE:
\begin{align}\label{Z1ode}
0=\dot Z_1 +(Z_1-q)^2-2(P+Z_1)(Z_1-q) -\epsilon_0, \quad Z_1(T)= -c,
\end{align}
 for which we can show that $Z_1=-P$ is a solution. By uniqueness of the solution of \eqref{Z1ode}, it follows that $\Lambda + 2H=-P$ on $[0,T]$. Subsequently, we determine $H$ from a linear equation using the equation of $\dot H$ after setting $\Lambda +2H$ as $-P$; this in turn determines  $\Lambda$.
 The above solution $(P, \Lambda, H)$ is unique.
The  control law \eqref{srmeu} now becomes
\begin{align*}
\hat u^i(t,x,\mu)&= -\{[P(t)+q]x+[\Lambda(t) +2H(t)-q] \bar x\}\\
&= [P(t)+q](\bar x-x),
\end{align*}
which agrees with \eqref{PO_systemic}.

\section{Concluding remarks}

Our performance estimate via multi-scale analysis has been based on existence of sufficiently well behaved solutions of two master equations for $V$ and $M$. For future work, it will be of interest to develop  existence results following such methods as in \cite{CDLL19,CCD22} for more specific models.

\bibliographystyle{amsplain}

\begin{thebibliography}{99}


\bibitem{ABS75}
R. Adler, M. Bazin, and M. Schiffer.
{\it Introduction to General Relativity}, New York: McGraw--Hill, 1975.



\bibitem{AGS05}
 L. Ambrosio, N. Gigli, and G. Savare. {\it  Gradient flows in metric spaces and the Wasserstein
spaces of probability measures}, Lectures in Mathematics, ETH Zurich: Birkhauser, 2005.




\bibitem{AM15}
J. Arabneydi and A. Mahajan.
Team-optimal solution of finite number of mean-field coupled LQG
subsystems. {\it Proc.  54th IEEE CDC},
 Osaka, Japan, Dec. 2015, pp. 5308--5313.





\bibitem{BT13}
M. Balandat and C. J. Tomlin.
On efficiency in mean field differential games. {\it Proc. American Control Conference},
Washington, DC, June 2013, pp. 2527--2532.


\bibitem{BCP18}
E. Bayraktar, A. Cosso, and H. Pham. Randomized dynamic programming principle and Feynman--Kac
representation for optimal control of McKean--Vlasov dynamics. {\it  Trans. American Mathematical Society},
370(3):2115--2160, 2018.



\bibitem{BFY13}
A. Bensoussan, J. Frehse, and P. Yam.
{\it Mean Field Games and Mean Field Type Control Theory},
New York: Springer,   2013.




\bibitem{CHM17}
P. E. Caines, M. Huang, and R. P. Malham\'e, Mean Field Games. In {\it Handbook of Dynamic Game Theory},
T. Basar and G. Zaccour Eds., pp. 345--372,  Berlin: Springer,  2017.





\bibitem{CDJS22}
P. Cardaliaguet, S. Daudin, J. Jackson, and P. Souganidis.
An algebraic convergence rate for the optimal control of McKean--Vlasov dynamics,        arXiv:2203.14554,  2022.


\bibitem{CDLL19}
P. Cardaliaguet, F. Delarue, J.-M. Lasry, and P.-L. Lions. {\it  The master equation and the
convergence problem in mean field games}, Princeton University Press, 2019.

\bibitem{CR19}
P. Cardaliaguet and C. Rainer. On the (in)efficiency of {MFG} equilibria. {\it SIAM Journal on Control and Optimization},
57(4):2292--2314, 2019.


\bibitem{C21}
T. Cavallazzi. It\^o--Krylov formula for a flow of measures. arXiv:2110.05251, 2021.


\bibitem{CA18}
C.D. Charalambous and N. U. Ahmed.
Centralized versus decentralized optimization of distributed stochastic differential decision systems with different information
structures Part II: applications.
{\it IEEE Trans. Autom. Control}, 63(7):1913--1928,  2018.


\bibitem{CD15}
R. Carmona and F. Delarue. Forward-backward stochastic differential equations and controlled
McKean--Vlasov dynamics, {\it  Annals of Probability}, { 43}(5):2647--2700, 2015.



\bibitem{CD18}
R. Carmona and F. Delarue. {\it Probabilistic Theory of Mean Field Games
with Applications}, vol I and II, Cham: Springer, 2018.




\bibitem{CFS15}
R. Carmona, J.-P. Fouque, and L.-H. Sun.
Mean field games and systemic risk.
{\it Communications in Mathematical Sciences}, 13(4):911--933, 2015.

\bibitem{CCD22} J.-F.
Chassagneux, D. Crisan, and F. Delarue. {\it  A probabilistic approach to classical solutions of the master
equation for large population equilibria}, Memoirs of AMS, 2022.



\bibitem{CBM17}
Y. Chen, A. Busic, and S. P. Meyn.
State estimation for the individual and the
population in mean field control with
application to demand dispatch.
{\it IEEE Trans. Autom. Control}, 62(3):1138--1149, Mar. 2017.


\bibitem{D92}
R. A. d'Inverno. {\it Introducing Einstein's relativity}.  New York: Oxford University Press, 1992.


\bibitem{DPT22}
M. F. Djete, D. Possamai,  and X. Tan. McKean--Vlasov optimal control: the dynamic programming
principle, {\it Ann. Probab.},  50(2):791--833, 2022.

\bibitem{DPT20}
M. F. Djete, D. Possamai,  and X. Tan.
McKean--Vlasov optimal control: limit theory and equivalence between different formulations, {\it Math. Oper. Res.}, 47:791--833, 2022.





\bibitem{FHH23}
X. Feng, Y. Hu, and J. Huang.
A unified approach to linear-quadratic Gaussian
mean-field team: homogeneity, heterogeneity and quasi-exchangeability,
{\it Ann. Appl. Probab.}, 33(4):2786--2823, 2023.







\bibitem{HW08}
E. Hairer and G. Wanner, {\it Analysis by Its History}, New York: Springer, 2008.


\bibitem{H13}
M. H. Holmes.
{\it Introduction to perturbation methods},
2nd ed., New York: Springer, 2013.


\bibitem{HWY21}
J. Huang, B.-C. Wang, and J. Yong.
Social optima in mean field linear-quadratic-Gaussian control with volatility uncertainty. {\it SIAM J. Control Optim.},
59(2):835--856,  2021.


\bibitem{HCM12}
M. Huang, P.E. Caines, and R.P. Malham\'e. Social optima in mean field LQG control: centralized and decentralized strategies. {\it IEEE Trans. Autom. Control}, 57(7):1736--1751,  2012.






\bibitem{HN16}
M. Huang and S.L. Nguyen. Linear-quadratic mean field teams with a
major agent, {\it  Proc. 55th IEEE CDC},  Las Vegas, NV, Dec. 2016, pp. 6958--6963.

\bibitem{HSS20}
M. Huang, S.-J. Sheu, and L.-H. Sun. Mean field social optimization: feedback person-by-person optimality and the master equation. {\it Proc. 59th IEEE CDC}, Jeju Island, Korea, Dec. 2020, pp. 4921--4926.

\bibitem{HY21}
M. Huang and X. Yang. Linear quadratic mean field social optimization: asymptotic solvability and decentralized control. {\it Applied Math. Optim.},  84:1969--2010, 2021.


\bibitem{H80}
Y.-C. Ho. Team decision theory and information structures. {\it
Proc. IEEE},  68(6):644--654, June 1980.










\bibitem{L17}
D. Lacker.
Limit theory for controlled McKean--Vlasov dynamics.
{\it SIAM J. Control Optim.},  55:1641--1672, 2017.









\bibitem{NM18}
G. Nuno and B. Moll.
Social optima in economies with heterogeneous agents.
{\it Review of
Economic Dynamics}, 28:150--180, 2018.


\bibitem{PW17}
H. Pham and X. Wei. Dynamic programming for optimal control of stochastic McKean--Vlasov dynamics.
{\it SIAM Journal on Control and Optimization}, 55(2):1069--1101, 2017.


\bibitem{PRT15}
B. Piccoli, F. Rossi, and E. Trelat.
Control to flocking of the kinetic Cucker--Smale model.
{\it SIAM J. Math. Anal.}, 47(6):4685--4719, 2015.


\bibitem{SNM18}
R. Salhab, J. L. Ny, and R. P. Malham\'e. Dynamic collective choice: social optima. {\it IEEE Trans. Autom.
Control}, 63(10):3487--3494, Oct. 2018.

\bibitem{SM76}
W. E. Schmitendorf and G. Moriarty. A sufficiency condition for coalitive Pareto-optimal solutions.
{\it J. Optim. Theory  Appl.}, 18(1):93--102,   1976.

\bibitem{SHM16}
N. Sen, M. Huang, and R. P. Malham\'e. Mean field social control with decentralized strategies and optimality characterization. {\it Proc. the 55th IEEE CDC}, Las Vegas, NV, Dec. 2016, pp. 6056--6061.




\bibitem{TTZ21}
M. Talbi, N. Touzi, and J. Zhang. Dynamic programming equation for the mean field optimal stopping
problem. arXiv:2103.05736, 2021.



\bibitem{V09}
C. Villani, {\it Optimal Transport: Old and New}, Berlin: Springer, 2009.


\bibitem{WS00}
P. R. de Waal and J. H. van Schuppen. A class of team problems with
discrete action spaces: optimality conditions based on
multimodularity. {\it SIAM J. Control Optim.},  38(3):875--892, 2000.


\bibitem{WZ17}
B.-C. Wang and J.-F. Zhang.
Social optima in mean field linear-quadratic-Gaussian models with Markov jump parameters,
{\it SIAM J. Control Optim.}, 55:429--456, 2017.




\bibitem{WZ20}
C. Wu and J. Zhang.
Viscosity solutions to parabolic master equations and
McKean--Vlasov SDEs with closed-loop controls,
{\it Ann. Appl. Probab.}, 30(2):936--986, 2020.






\end{thebibliography}

\section*{Appendix A: A formal derivation of the master  equation of $V$}

\label{sec:AppA}

\renewcommand{\theequation}{A.\arabic{equation}}
\setcounter{equation}{0}
\renewcommand{\thetheorem}{A.\arabic{theorem}}
\setcounter{theorem}{0}
\renewcommand{\thesection}{A.\arabic{theorem}}
\setcounter{theorem}{0}
\renewcommand{\thesubsection}{A.\arabic{subsection}}
\setcounter{subsection}{0}

This appendix considers a more general model with diffusion coefficients $\sigma (X_t^i, u_t^i, \mu_t^{-i})$ and
 $\sigma_0(X_t^i, u_t^i, \mu_t^{-i})$ before specializing to the form in \eqref{sdeXi}. This will give us more insights into the dynamic programming method.        Below we accordingly denote $\Sigma(x,u,\mu)\coloneqq (\sigma\sigma^T)(x,u,\mu)$
and
$\Sigma_0(x,u,\mu) \coloneqq (\sigma_0\sigma_0^T)(x,u,\mu)$.

Now we take  initial time $t\in [0,T)$ and initial states $(x^1, \cdots, x^N)$. For the feedback control law $\phi$,   denote the controlled state processes
for agents ${\cal A}_j$, $j\ne i$, as
\begin{align}
dX_s^j=& f(X_s^j,
 \phi(s, X_s^j,  \mu_s^{-j}), \mu_s^{-j}) ds
  +\sigma(X_s^j, \phi(s, X_s^j,  \mu_s^{-j}), \mu^{-j}_s) dW_s^j
  \label{dxjapp} \\
  &+\sigma_0(X_s^j,\phi(s, X_s^j,  \mu_s^{-j}), \mu_{s}^{-j}) dW_s^0,\quad s\ge t. \notag
\end{align}
For agent ${\cal A}_i$, we have
\begin{align} \label{dxiapp}
dX_s^i =\ &f(X_s^i, u_s^i, \mu_s^{-i}) +\sigma(X_s^i, u_s^i, \mu^{-i}_{s}) dW_s^i
 +\sigma_0(X_s^i, u_s^i, \mu_{s}^{-i}) dW_s^0,\quad s\ge t,
\end{align}
where we take $u_s^i\equiv u^i\in \mathbb{U}$ for $s\in [t,t+\epsilon)$ and $u_s^i=\phi(s, X_s^i, \mu_s^{-i})$ on $[t+\epsilon, T]$.

\subsection{The control perturbation of ${\cal A}_i$}
\label{sec:pbp}

Note that
$V^\phi(t, x^i, \mu)$ is defined on $[0,T]\times \mathbb{R}^n \times {\mathcal P}^{N-1}_{\rm em}({\mathbb R}^n)$, where ${\mathcal P}^{N-1}_{\rm em}({\mathbb R}^n) $ is only a subset of ${\mathcal P}_2(\mathbb{R}^n)$. We still formally denote the derivative $ \delta_\mu V^\phi(t, x^i, \mu;y)$, which is interpreted
in the following sense: for any $\nu\in{\cal P}_{\rm em}^{N-1}({\mathbb R}^n)$,
\begin{align*}
&V^\phi(t,x^i, \nu)-V^\phi(t,x^i, \mu)
= \int_y \delta_\mu V^\phi(t,x^i, \mu;y) (\nu-\mu) (dy)+ o(W_2(\nu, \mu)).
\end{align*}
We give an example to illustrate.
\begin{example}
Suppose $h(t, x^{i}, \mu^{{-i}})= x^{iT}\Pi_1(t) x^i + \bar x^{(-i)T} \Pi_2(t)\bar x^{(-i)}+ 2 x^{iT} \Pi_{12}(t) \bar x^{(-i)}  $, where $ \bar x^{(-i)}\coloneqq \langle y \rangle_{\mu^{-i}}=\frac{1}{N-1}\sum_{j=1,j\ne i}^N x^j $ and $\Pi_1(t)$ and $\Pi_2(t)$ are symmetric. Then  $$
\delta_\mu h(t, x^i, \mu^{-i}; y)= 2 y^T\Pi_2(t) \langle y \rangle_{\mu^{-i}} + 2 x^{iT} \Pi_{12}(t) y  +\chi(t,x^i, \mu^{-i}),
$$
 where $\chi(t,x^i, \mu^{-i})$ is a normalizing term.
\end{example}

We further formally denote the second order derivative $\delta_{\mu\mu} V^\phi(t, x^i, \mu; y,z)$.
 We make formal usage of  partial derivatives
$\partial_t V^\phi$, $\partial_{x^i} V^\phi$, $\partial_{x^i}^2 V^\phi$,
$\partial_{x^i}\delta_\mu V^\phi(t, x^i, \mu;y)$, $\partial_{x^iy}\delta_\mu V^\phi(t, x^i, \mu;y)$, $\partial_{y}^2\delta_\mu V^\phi(t, x^i, \mu;y)$, $\partial_{yz}\delta_{\mu\mu} V^\phi(t, x^i, \mu;y,z)$.

\subsection{Cost estimate of ${\cal A}_i$}

Under \eqref{dxjapp}--\eqref{dxiapp},
let $J_i$ be given by \eqref{Jidel}.
Taking a fixed constant control $u^i$ (given the initial condition $(t, x^i, \mu^{-i})$) on the small interval
$[t, t+\epsilon]$,
we have the approximation
\begin{align}
J_i=\ & L(x^i, u^i, \mu^{-i}) \deltae + V^\phi(t, x^i, \mu^{-i}) \label{jiex} \\
& + \partial_t V^\phi(t, x^i, \mu^{-i}) \deltae
+  \partial_{x^i}V^\phi (t, x^i, \mu^{-i})
f(x^i, u^i, \mu^{-i}) \deltae \notag  \\
&+ \frac{1}{2} {\rm Tr}
[\partial_{x^i}^2 V^\phi(t, x^i, \mu^{-i}) \Sigma  (x^i,u^i,\mu^{-i}  ) ] \deltae \notag  \\
&+ \frac{1}{2} {\rm Tr}
[\partial_{x^i}^2 V^\phi(t, x^i, \mu^{-i}) \Sigma_0  (x^i,u^i,\mu^{-i}  ) ] \deltae \notag  \\
&+
\EE\langle \Delta\mu^{-i}(dy) ,  \delta_{\mu} V^\phi (t, x^i, \mu^{-i}; y) \rangle  \notag    \\
&+ \EE \langle \Delta\mu^{-i}(dy) ,   \partial_{x^i} \delta_\mu V^\phi (t, x^i, \mu^{-i}; y)\xi_t^i \rangle \notag \\
&+ \frac{1}{2}\EE\int_z\int_y \delta_{\mu\mu} V^\phi (t, x^i, \mu^{-i}; y,z)\Delta \mu^{-i} (dy) \Delta\mu^{-i} (dz) \notag \\
&+ o(\deltae), \notag
\end{align}
where $\Delta\mu^{-i}\coloneqq \mu_{t+\deltae}^{-i}-\mu^{-i}  $ and
$$
\xi_t^i\coloneqq \sigma(x^i,u^i,\mu^{-i})(W_{t+\epsilon}^i-W_t^i)+
\sigma_0(x^i, u^i,\mu^{-i})(W_{t+\deltae}^0-W_t^0).
$$
The three expectations in \eqref{jiex} are needed since the  empirical distribution  $\mu_{t+\epsilon}^{-i} $  is random.
 Within $J_i$, denote
\begin{align*}
{\mathcal K}_1(t,x^i, \mu^{-i}, u^i)\coloneqq \ &L(x^i, u^i, \mu^{-i})+
\partial_{x_i} V^\phi (t, x^i, \mu^{-i})
f(x^i, u^i, \mu^{-i})\\
& + \frac{1}{2} {\rm Tr}
[\partial_{x^i}^2 V^\phi(t, x^i, \mu^{-i}) \Sigma  (x^i,u^i,\mu^{-i}  ) ] \\
&+ \frac{1}{2} {\rm Tr}
[\partial^2_{x^i} V^\phi(t, x^i, \mu^{-i}) \Sigma_0  (x^i,u^i,\mu^{-i}  ) ]  .
\end{align*}
The sum ${\cal K}_1$ explicitly depends on $u^i$.
The remaining components in $J_i$ receive either no or negligible impact from $u^i$.
The optimal choice of $u^i$, however, is not to simply minimize ${\cal K}_1$.
Instead, one should take into account the impact of $u^i$ on all other agents, which  we call the social impact.

We  check the double integral term in  $J_i$.
As $N\to \infty$,  we obtain the approximation
\begin{align*}
&  \frac{1}{2}\EE\int_z\int_y \delta_{\mu\mu} V^\phi (t, x^i, \mu^{-i}; y,z)\Delta \mu^{-i} (dy) \Delta\mu^{-i} (dz)\\
 &\qquad \approx
 \frac{\deltae}{2}\int_z\int_y {\rm Tr}\Big[\partial_{yz}\delta_{\mu\mu} V^\phi (t, x^i, \mu; y,z)\\
&\qquad\qquad\qquad\cdot
\sigma_0(z,\phi(t,z,\mu),\mu)\sigma_0^T(y,\phi(t,y,\mu),\mu)\Big]\mu(dy)\mu(dz),
\end{align*}
where $\mu^{-i}$ has been approximated by $\mu$.

\subsection{Cost estimate for all other agents}

We check how $u_t^i\equiv u^i$ on $[t, t+\deltae]$ affects the cost of
agent ${\cal A}_j$, $j\ne i$. By symmetry of the dynamics of the $N$ agents on  $[t+\deltae, T]$, the cost of
${\cal A}_j$ on $[t, T]$  may be written as
\begin{align}
J_j(t, x^j, \mu^{-j}, {\mathbf u}(\cdot))&=\EE\Big[ \int_t^{t+\deltae}
L(X_s^j,u_s^j,
\mu_s^{-j} )ds+
V^\phi({t+\deltae}, X^j_{t+\deltae}, \mu^{-j}_{t+\deltae} ) \Big],\label{jjV}
\end{align}
where $u_s^j=\phi(s, X_s^j, \mu_s^{-j})$
on the whole interval $[t,T]$. In general $J_j(t,x^j,\mu^{-j}, {\mathbf u}(\cdot))\ne V^\phi(t, x^j, \mu^{-j})$ since $u^i_s$ may differ from  $\phi(s, X_s^i, \mu_s^{-i})$ on $[t, t+\deltae]$.
Then we take the expansion of \eqref{jjV} to formally  obtain
\begin{align}
J_j =\ & L(x^j,  \phi(t,x^j,\mu^{-j}),\mu^{-j}) \deltae+ V^\phi(t, x^j,
\mu^{-j})  \notag   \\
  & +  \partial_t V^\phi(t, x^j, \mu^{-j}) \deltae
  + \partial_{x^j} V^\phi (t, x^j, \mu^{-j})
f(x^j, \phi(t,x^j,\mu^{-j}), \mu^{-j}) \deltae  \notag \\
  & + \frac{1}{2} {\rm Tr}
[\partial^2_{x^j} V^\phi(t, x^j, \mu^{-j}) \Sigma  (x^j,\phi(t,x^j,\mu^{-j}),\mu^{-j}  ) ]\deltae  \notag \\
&+ \frac{1}{2} {\rm Tr}
[\partial^2_{x^j} V^\phi(t, x^j, \mu^{-j}) \Sigma_0  (x^j,\phi(t,x^j,\mu^{-j}),\mu^{-j}  ) ]\deltae  \notag  \\
&+ \EE\langle \Delta\mu^{-j}(dy), \delta_{\mu} V^\phi(t, x^j, \mu^{-j}; y) \rangle \notag  \\
&+  \EE \langle \Delta\mu^{-j} (dy),   \partial_{x^j}\delta_\mu V^\phi (t, x^j, \mu^{-j}; y)\xi_t^j \rangle  \notag  \\
&+ \frac{1}{2}\EE\int_z\int_y \delta_{\mu\mu} V^\phi (t, x^j, \mu^{-j}; y,z)\Delta \mu^{-j} (dy) \Delta\mu^{-j} (dz)+o(\deltae)  \notag \\
  \eqqcolon & J_{j,1}+\cdots+J_{j,9}+o(\deltae), \notag
\end{align}
where
$
\xi_t^j\coloneqq \sigma(x^j,u^j,\mu^{-j})(W_{t+\epsilon}^j-W_t^j)+ \sigma_0(x^j, u^j,\mu^{-j})(W_{t+\deltae}^0-W_t^0)
$ with $u^j=\phi(t, x^j, \mu^{-j})$. The nine terms $J_{j,k}$ are identified by their order of appearance. Note that the first six terms in the sum are not affected by   $u^i\in \mathbb{U}$.

In the following we use the notation $\delta_\mu V^\phi(t, x, \mu;y)$ and $\delta_{\mu\mu} V^\phi(t, x, \mu; y,z)$. Thus $\partial_{yz}\delta_{\mu\mu} V^\phi(t, x^j, \mu; x^k,x^i) $ stands for $ \partial_{yz}\delta_{\mu\mu} V^\phi(t, x^j, \mu; y,z)|_{y=x^k, z=x^i} $.
We have
\begin{align*}
J_{j,7}\coloneqq& \ \EE\langle (\mu_{t+\deltae }^{-j}-\mu^j)(dy), \delta_\mu V^\phi(t, x^j, \mu^{-j}; y) \rangle\\
=&\frac{1}{N-1}\sum_{k \in {\mathcal N} - \{i,j\}  }\EE[\delta_{\mu} V^\phi(t, x^j, \mu^{-j}; X^k_{t+\deltae}) - \delta_{\mu} V^\phi(t, x^j, \mu^{-j}; x^k)] \\
&+ \frac{1}{N-1}\EE[\delta_{\mu} V^\phi(t, x^j, \mu^{-j}; X^i_{t+\deltae})- \delta_\mu V^\phi(t, x^j, \mu^{-j}; x^i)],
\end{align*}
where  the second to last line has a much smaller dependence
on $u^i$ than the last line has.
 Specifically, when $u^i$ has a  change of magnitude $O(1)$, it results in an $O(\deltae)$ change of $X_{t+\deltae}^i$, which in turn causes an $O(\deltae^2/N)$ change of $X_{t+\deltae}^j$.
For $J_{j,7}$, we estimate the second component in the above sum and have
\begin{align*}
J_{j,7,i} &\coloneqq \frac{1}{N-1} \EE[\delta_{\mu} V^\phi(t, x^j, \mu^{-j}; X^i_{t+\deltae})- \delta_{\mu} V^\phi(t, x^j, \mu^{-j}; x^i)] \\
&= \frac{1}{N-1}\Big\{ \partial_{x^i}\delta_{\mu} V^\phi (t, x^j, \mu^{-j}; x^i) f(x^i,  u^i,\mu^{-i}) \deltae  \\
&\quad + \frac{1}{2} {\rm Tr} [\partial^2_{x^i} \delta_{\mu} V^\phi (t, x^j, \mu^{-j}; x^i)\Sigma(x^i,u^i,\mu^{-i}  )]\deltae\\
& \quad +  \frac{1}{2} {\rm Tr} [\partial ^2_{x^i} \delta_{\mu} V^\phi (t, x^j, \mu^{-j}; x^i)\Sigma_0(x^i,u^i,\mu^{-i}  )]\deltae +o(\deltae)\Big\}.
\end{align*}

Next we have
\begin{align*}
J_{j,8}\coloneqq&\ \EE \langle \Delta\mu^{-j}(dy) ,
 \partial_{x^j}\delta_{\mu}V^\phi (t, x^j, \mu^{-j}; y)\xi_t^j \rangle \\
=&\frac{1}{N-1}\EE \sum_{k\in {\mathcal N}_{-j}   }
[  \partial_{x^j}\delta_{\mu}V^\phi (t, x^j, \mu^{-j};
X_{t+\deltae}^k)-\partial_{x^j}\delta_{\mu} V^\phi (t, x^j,
\mu^{-j}; x^k)]\xi_t^j \\
=& \frac{1}{N-1} \sum_{k\in {\mathcal N}_{-j}} {\rm Tr}\Big[\partial_{x^jx^k} \delta_{\mu} V^\phi(t,x^j, \mu^{-j}; x^k)\\
&\qquad\qquad\qquad\qquad\cdot \sigma_0(x^k, u^k, \mu^{-k})\sigma_0^T(x^j, u^j, \mu^{-j})\Big]\deltae+o(\deltae).
\end{align*}
In the last summation $u^l=\phi(t, x^l, \mu^{-l})$ for all $l\in {\cal N}_{-i}$.
Within the above sum for $J_{j,8}$, denote its component with $k=i$, which depends on $u^i$:
\begin{align*}
J_{j,8,i}\coloneqq &\  \frac{\epsilon}{N-1} {\rm Tr}[\partial_{x^jx^i} \delta_{\mu} V^\phi(t,x^j, \mu^{-j}; x^i) \cdot \sigma_0(x^i, u^i, \mu^{-i})\sigma_0^T(x^j, u^j, \mu^{-j})].
\end{align*}

To analyze $J_{j,9}$,
denote
\begin{align*}
\xi_{kl} \coloneqq \ &
\delta_{\mu\mu} V^\phi (t, x^j, \mu^{-j}; X_{t+\deltae}^k,X_{t+\deltae}^l)-\delta_{\mu\mu} V^\phi (t, x^j, \mu^{-j}; x^k,X_{t+\deltae}^l) \\
&-\delta_{\mu\mu} V^\phi (t, x^j, \mu^{-j}; X_{t+\deltae}^k,x^l)+\delta_{\mu\mu} V^\phi (t, x^j, \mu^{-j}; x^k,x^l) .
\end{align*}
We have
\begin{align*}
J_{j,9}:= &\frac{1}{2}\EE\int_z\int_y \delta_{\mu\mu} V^\phi (t, x^j, \mu^{-j}; y,z)\Delta \mu^{-j} (dy) \Delta\mu^{-j} (dz)  \\
= & \frac{1}{2(N-1)^2}
\sum_{ k,l\in {\mathcal N}_{-j} } \EE \xi_{kl}\\
\hskip -1cm =&\frac{1}{2(N-1)^2} \sum_{ k\ne l; k, l\in{\mathcal N}_{-j} }
\Big\{{\rm Tr}\Big[\partial_{yz}\delta_{\mu\mu} V^\phi (t, x^j, \mu^{-j}; x^k,x^l)\\
&\qquad\qquad\qquad\qquad \cdot \sigma_0(x^l,u^l,\mu^{-l})
\sigma_0^T(x^k,u^k,\mu^{-k})\Big]\deltae+o(\deltae)\Big\}\\
& + \frac{1}{2(N-1)^2} \sum_{ k\in {\mathcal N}_{-j} }^N
\Big\{{\rm Tr}\Big[\partial_{yz}\delta_{\mu\mu} V^\phi (t, x^j, \mu^{-j}; x^k,x^k)\\
&\qquad\qquad\qquad\qquad\qquad \cdot(\Sigma+ \Sigma_0)(x^k,u^k,\mu^{-k})\Big]
\deltae+o(\deltae)\Big\},
\end{align*}
where   $u^l=\phi(t,x^l, \mu^{-l})$ for each $l\in {\cal N}_{-j}$.
Subsequently, within the expression of $J_{j,9}$,
we have the following $u^i$ dependent components:
\begin{align*}
 J_{j,9,i}&\coloneqq
\frac{\deltae}{2(N-1)^2} \sum_{ k\in {\cal N} - \{i,j\} }
{\rm Tr}\Big[\partial_{yz}\delta_{\mu\mu} V^\phi (t, x^j, \mu^{-j}; x^k,x^i) \\
&\hskip 4cm \cdot \sigma_0(x^i,u^i,\mu^{-i})\sigma_0^T(x^k,\phi(t,x^k,\mu^{-k}),\mu^{-k})
\Big]\\
&+ \frac{\deltae}{2(N-1)^2} \sum_{l \in{\mathcal N} - \{i,j\} }
{\rm Tr}\Big[\partial_{yz}\delta_{\mu\mu} V^\phi (t, x^j, \mu^{-j}; x^i,x^l) \\&\hskip 4cm\cdot \sigma_0(x^l,\phi(t,x^l,\mu^{-l}),\mu^{-l})\sigma_0^T(x^i,u^i,\mu^{-i})
\Big]\\
&+ \frac{\deltae}{2(N-1)^2} {\rm Tr}[\partial_{yz}\delta_{\mu\mu} V^\phi (t, x^j, \mu^{-j}; x^i,x^i) \cdot (\Sigma_0+\Sigma)(x^i,u^i,\mu^{-i})].
 \\
\end{align*}

\subsection{Approximation}

For large $N$, all the empirical distributions $\mu^{-j}$, $1\le j\le N$, may be approximated by a common $\mu\in P_{\rm em}^{N-1}$.
For $J_{j,7,i}$, denote \begin{align*}
\Gamma_{7}(t,y,\mu, x^i, u^i)
\coloneqq &\ \partial_{x^i} \delta_{\mu} V^\phi (t, y, \mu; x^i) f(x^i,  u^i,\mu)   \\
& + \frac{1}{2} {\rm Tr} [\partial^2_{x^i} \delta_{\mu} V^\phi (t, y, \mu; x^i)\Sigma(x^i,u^i,\mu  )]\\
& +  \frac{1}{2} {\rm Tr} [\partial^2_{x^i} \delta_{\mu} V^\phi (t, y, \mu; x^i)\Sigma_0(x^i,u^i,\mu  )].
\end{align*}
  We take the approximation
  \begin{align}
\sum_{j\in {\mathcal N}_{-i} } J_{j,7,i}
&\approx\deltae
\int_{{\mathbb R}^n}\Gamma_{7}(t,y,\mu, x^i, u^i)\mu(dy)\eqqcolon {\cal K}_2(t,x^i,u^i, \mu)\deltae,  \notag  \\
\sum_{j\in {\mathcal N}_{-i} } J_{j,8,i}&
\approx \deltae\int_{{\mathbb R}^n}{\rm Tr}[\partial_{yx^i} \delta_{\mu} V^\phi(t,y, \mu; x^i)  
\sigma_0(x^i, u^i, \mu)\sigma_0^T(y, \phi(t,y,\mu), \mu)]\mu(dy)\nonumber\\
& \eqqcolon {\cal K}_3(t,x^i,u^i,\mu)\deltae . \notag
\end{align}
Next,
\begin{align*}
& \sum_{j\in {\mathcal N}_{-i} } J_{j,9,i}\\
\approx &  \frac{\deltae }{2N}\sum_{j\ne i}\int_y
{\rm Tr}[\partial_{yx^i}\delta_{\mu\mu} V^\phi (t, x^j, \mu; y,x^i) \cdot
\sigma_0(x^i,u^i,\mu)\sigma_0^T(y,\phi(t,y,\mu),\mu)]\mu(dy)\\
&+\frac{\deltae }{2N}\sum_{j\ne i} \int_z{\rm Tr}[\partial_{x^iz}\delta_{\mu\mu} V^\phi (t, x^j, \mu; x^i,z) \cdot \sigma_0(z,\phi(t,z,\mu),\mu)\sigma_0^T(x^i,u^i,\mu)] \mu(dz)\\
\approx&\frac{\deltae}{2}\int_w\int_y
{\rm Tr}[\partial_{yx^i}\delta_{\mu\mu} V^\phi (t, w, \mu; y,x^i) \cdot
\sigma_0(x^i,u^i,\mu)\sigma^T_0(y,\phi(t,y,\mu),\mu)] \mu(dy)\mu(dw)\\
&+\frac{\deltae }{2} \int_w\int_z{\rm Tr}[\partial_{x^iz}\delta_{\mu\mu} V^\phi (t, w, \mu; x^i,z) \cdot \sigma_0(z,\phi(t,z,\mu),\mu)\sigma_0^T(x^i,u^i,\mu)]\mu(dz)\mu(dw)\\
\eqqcolon& {\cal K}_4(t,x^i,u^i, \mu)\deltae.
\end{align*}

\subsection{Cooperative optimizer selection}

Within the mean field limit, we consider $\mu\in {\mathcal P}_2({\mathbb R}^n)$, and the control law $\phi(t, x, \mu)$.
The function $V^\phi(t, x, \mu)$
is defined on $[0,T]\times \mathbb{R}^n\times {\mathcal P}_2({\mathbb R}^n) $ and does not depend on $N$.
By adding up ${\cal K}_1, {\cal K}_2, {\cal K}_3, $ and ${\cal K}_4$, we define
\begin{align}
& \hskip -0.5cm \Phi^\phi(t, x, u^i,\mu,   V^\phi(\cdot) )\nonumber\\
  \coloneqq &   L(x, u^i, \mu)+
\partial_{x} V^\phi (t, x, \mu)
f(x, u^i, \mu)\nonumber\\
& + \frac{1}{2} {\rm Tr}
[\partial^2_{x} V^\phi(t, x, \mu) \Sigma  (x,u^i,\mu  ) ]+ \frac{1}{2} {\rm Tr}
[\partial^2_{x} V^\phi(t, x, \mu) \Sigma_0  (x,u^i,\mu  ) ] \nonumber \\
&+\Big\langle \mu(dy), \Big\{ \partial_{x} \delta_\mu V^\phi(t, y, \mu; x) f(x,  u^i,\mu) \nonumber  \\
& +{\rm Tr}[ \partial_{xy}\delta_\mu V^\phi(t,y,\mu;x)
\sigma_0(y, \phi(t,y,\mu),\mu)\sigma_0^T(x, u^i, \mu) ]\nonumber \\
& + \frac{1}{2} {\rm Tr} [\partial^2_{x} \delta_\mu V^\phi (t, y, \mu; x)\Sigma(x,u^i,\mu  )]\nonumber\\
& +  \frac{1}{2} {\rm Tr} [\partial^2_{x} \delta_\mu V^\phi (t, y, \mu; x)\Sigma_0(x,u^i,\mu  )]  \Big\}\Big\rangle \nonumber\\
&+\frac{1}{2}\langle \mu^{\otimes 2} (dydw),
{\rm Tr}[\partial_{yx}\delta_{\mu\mu} V^\phi (t, w, \mu; y,x) \cdot
\sigma_0(x,u^i,\mu)\sigma^T_0(y,\phi(t,y,\mu),\mu)]\rangle \nonumber \\
&+\frac{1}{2}\Big\langle \mu^{\otimes 2} (dzdw), {\rm Tr}\Big[\partial_{xz}\delta_{\mu\mu} V^\phi (t, w, \mu; x,z)\nonumber  \\
&\qquad\qquad\qquad\qquad \cdot \sigma_0(z,\phi(t,z,\mu),\mu)\sigma_0^T(x,u^i,\mu)\Big]\Big\rangle,\label{K123}
\end{align}
where the control law $\phi$ has been used by other agents.
We take $\hat u^i$ as a minimizer of $\Phi^\phi$ to obtain
\begin{align}
 \Phi^\phi (t, x, \hat u^i (t,x, \mu), \mu,  V^\phi(\cdot) )
=&\min_{u^i}\Phi^\phi ( t,x,u^i, \mu, V^\phi (\cdot)) .  \label{cpu}
\end{align}
The selection of $ u^i$ takes into account its impact on both $J_i$ and  all other agents' costs.
The remaining step is to specify $\phi$ by a consistency condition to be introduced below.

\subsection{The master equation}
Combining \eqref{Jidel}, \eqref{jiex} and the optimizer selection rule
\eqref{cpu}, we introduce the master equation (as a special HJB equation):
\begin{align*}
&\hskip -1cm - \partial_t V(t, x,\mu)\\
= \ &   V_x (t, x, \mu)
f(x, \hat u^i, \mu) + \frac{1}{2} {\rm Tr}
[\partial^2_{x} V(t, x, \mu) \Sigma  (x,\hat u^i,\mu  ) ] \\
&+ \frac{1}{2} {\rm Tr}
[\partial_{x}^2 V(t, x, \mu) \Sigma_0  (x,\hat u^i,\mu  ) ] + L(x, \hat u^i, \mu)\\
& + \langle\mu(dy),  \partial_y\delta_\mu V(t,x, \mu; y) f( y, \phi(t, y, \mu),\mu)\rangle  \\
&+\frac{1}{2}\langle \mu(dy), {\rm Tr}[
\partial_{y}^2 \delta_\mu
V(t, x, \mu; y) \Sigma(y,
\phi(t,y,\mu),\mu)]\rangle \\
& +\frac{1}{2}\langle \mu(dy), {\rm Tr}[
\partial_{y}^2 \delta_\mu
V(t, x, \mu; y) \Sigma_0(y, \phi(t,y,\mu),\mu)]\rangle \\
&+\langle\mu(dy) ,  {\rm Tr} [ \partial_{xy}\delta_\mu V(t,x, \mu;y) \cdot \sigma_0(y, \phi(t,y,\mu),\mu)\sigma_0^T(x,\hat u^i, \mu) ]\rangle  \\
&+\frac{1}{2}\Big\langle \mu^{\otimes 2}(dydz),  {\rm Tr}\Big[\partial_{yz}\delta_{\mu\mu} V (t, x, \mu; y,z)\\
&\qquad\qquad \quad\qquad\cdot
\sigma_0(z,\phi(t,z,\mu),\mu)\sigma^T_0(y,\phi(t,y,\mu),\mu)\Big]\Big\rangle ,
\end{align*}
where
 $(t,x,\mu)\in[0,T]\times \mathbb{R}^n\times
 {\cal P}_2(\mathbb{R}^n)  $ and the  terminal condition is
$$
V(T, x, \mu)=g(x, \mu).
$$
Moreover, the control law $\phi$ is required to be equal to  the
 optimizer $\hat u^i$, i.e.,
\begin{align}
\hat u^i(t,x,\mu)=&\arg \min_{u^i\in \mathbb{U}}\Phi^\phi(t, x, u^i,\mu,   V(\cdot) ), \\
  \phi (t, x, \mu)=& \hat u^i(t, x, \mu),   \label{pcc}
\end{align}
where \eqref{pcc} is called the consistency condition and
\begin{align}
&\hskip -0.5cm \Phi^\phi(t, x, u^i,\mu,   V(\cdot) )\nonumber\\
  \coloneqq &\   L(x, u^i, \mu)+
 V_x (t, x, \mu)
f(x^i, u^i, \mu)\nonumber\\
& + \frac{1}{2} {\rm Tr}
[\partial^2_{x} V(t, x, \mu) \Sigma  (x,u^i,\mu  ) ]+ \frac{1}{2} {\rm Tr}
[\partial^2_{x} V(t, x, \mu) \Sigma_0  (x,u^i,\mu  ) ] \nonumber \\
&+\Big\langle \mu(dy), \Big\{ \partial_{x} \delta_\mu V(t, y, \mu; x) f(x,  u^i,\mu) \nonumber  \\
&  +{\rm Tr}[ \partial_{xy}\delta_\mu V(t,y,\mu;x)
\sigma_0(y, \phi(t,y,\mu),\mu)\sigma_0^T(x, u^i, \mu) ]\nonumber \\
& + \frac{1}{2} {\rm Tr} [\partial^2_{x} \delta_\mu V (t, y, \mu; x)\Sigma(x,u^i,\mu  )]\nonumber\\
& +  \frac{1}{2} {\rm Tr} [\partial_{x}^2 \delta_\mu V (t, y, \mu; x)\Sigma_0(x,u^i,\mu  )]\Big \}\Big\rangle\nonumber\\
&+\frac{1}{2} \langle \mu^{\otimes2}(dydw),
{\rm Tr}[\partial_{yx}\delta_{\mu\mu} V (t, w, \mu; y,x) \cdot
\sigma_0(x,u^i,\mu)\sigma^T_0(y,\phi(t,y,\mu),\mu)]\rangle \nonumber \\
&+\frac{1}{2}\langle \mu^{\otimes2}(dzdw), {\rm Tr}[\partial_{xz}\delta_{\mu\mu} V (t, w, \mu; x,z) \cdot \sigma_0(z,\phi(t,z,\mu),\mu)\sigma_0^T(x,u^i,\mu)]\rangle,\label{K123b}
\end{align}
where the control law $\phi$ in $\Phi^\phi$ has been used by other agents.

\begin{remark}
 The value function $V$ corresponds to a representative agent  ${\cal A}_i$ interacting with an infinite population.
  \end{remark}

\begin{remark}
When $\sigma$ and $\sigma_0$ are constant matrices, $\Phi^\phi$ in \eqref{K123b} reduces to the simpler form $\Phi$ in Section \ref{sec:me}.
\end{remark}

\section*{Appendix B: Preliminary lemmas and semi-symmetry property}
\renewcommand{\theequation}{B.\arabic{equation}}
\setcounter{equation}{0}
\renewcommand{\thetheorem}{B.\arabic{theorem}}
\setcounter{theorem}{0}
\renewcommand{\thesection}{B.\arabic{theorem}}
\setcounter{theorem}{0}
\renewcommand{\thesubsection}{B.\arabic{subsection}}
\setcounter{subsection}{0}

\begin{lemma} \label{lemma:psixy}
Suppose the function $\psi(x,y)$ from $\mathbb{R}^n\times \mathbb{R}^n$ to $\mathbb{R}$ has continuous partial derivatives $\partial_{xy}\psi(x,y)$ and $\partial_{yx}\psi(x,y)$, and denote $h(x,y)=\partial_{xy}\psi(x,y)$. Then
\begin{align}
&\partial_{yx}\psi(x,y)= (h(x,y))^T, \label{ps1} \\
&\partial_{xy} \psi(y,x) =( h(y,x))^T. \label{ps2}
\end{align}
\end{lemma}
\begin{proof}
 The $(i,j)$-th entry in $\partial_{xy}\psi(x,y)$ is $\partial_{x_iy_j}\psi(x,y)$ and  the $(j,i)$-th entry in $\partial_{yx} \psi(x,y)$ is $\partial_{y_jx_i}\psi(x,y)$.
Now \eqref{ps1} follows from Schwarz's theorem \cite{HW08}. Next \eqref{ps1} yields
\eqref{ps2} by switching  $x$ and $y$.
\end{proof}

\begin{lemma}\label{lemma:trDyz}
Let $\bar\Sigma\in \mathbb{R}^{n\times n }$ be symmetric. Suppose
$ \partial_{zy}\delta_\mu V(t, y, \mu; z) $ and $\partial_{yz}\delta_\mu V(t, y, \mu; z)$ are continuous in $(y,z)$.
Then
\begin{align}
&\mbox{\rm Tr}[\partial_{zy}\delta_\mu V(t, y, \mu; z)\bar \Sigma] =\mbox{\rm Tr}[\partial_{yz}\delta_\mu V(t, y, \mu; z)\bar \Sigma], \label{VSig}\\
& \mbox{\rm Tr}[\partial_{yz}\delta_\mu V(t, y, \mu; z)\bar \Sigma]|_{z=y} =
\mbox{\rm Tr}[\partial_{yz}\delta_\mu V(t, z, \mu; y)\bar \Sigma]|_{z=y}.
\label{dzyVS}
\end{align}
\end{lemma}
\begin{proof}
 If $D$ and $\bar \Sigma$ are  $n\times n$ matrices and $\bar \Sigma$ is symmetric, then we have
\begin{align}
 \mbox{Tr} ( D\bar \Sigma )=\mbox{Tr} ( \bar \Sigma D) = \mbox{Tr}[ ( \bar \Sigma D)^T]= \mbox{Tr} (D^T\bar \Sigma),\label{Dtr}
\end{align}
which together with  Lemma \ref{lemma:psixy}  yields \eqref{VSig}.
Denote $h(y,z)\coloneqq \partial_{yz} \delta_\mu V(t, y,\mu;z)$.
By Lemma \ref{lemma:psixy}, $\partial_{yz}\delta_\mu V(t, z, \mu; y)=(h(z,y))^T$,  which combined with \eqref{Dtr} implies \eqref{dzyVS}.
\end{proof}

Consider a continuous function $\psi(t,s)$: $R_c=[0,T]\times [0, c) \to \mathbb{R}$, where $c>0$. Define the partial derivatives  $\partial_t\psi(t,s)$, $\partial_s\psi(t,s)$ and $\partial_t\partial_s\psi(t,s)$  on $R_c$, where each partial derivative is interpreted as a one-sided derivative when we take $\partial_t$ ($\partial_s $, resp.) at $t=0$ or $t=T$ (at $s=0$, resp.).
For instance,  for $0\le t\le T$, we have
$$\partial_s\psi (t,0)\coloneqq  \lim_{\epsilon\to 0^+} \frac{\psi(t, \epsilon)- \psi(t, 0)}{\epsilon}.
$$

The next lemma extends the symmetry property of second order partial derivatives in Schwarz's theorem \cite{HW08} to the case of  boundary points of a region.
The proof uses essentially the same argument as in \cite[p.317]{HW08} and is omitted here.
\begin{lemma}\label{lemma:sch}
Suppose $\psi$ satisfies the following conditions:

(i) The partial derivatives
$\partial_t\psi(t,s)$, $\partial_s\psi(t,s)$, and $\partial_t\partial_s\psi(t,s)$ exist on $R_c$.

(ii) $\partial_t\partial_s\psi(t,s)$ is continuous at the point $(t_0, 0)$
for $t_0\in [0,T]$.

Then $ \partial_s\partial_t\psi(t_0,0)$ exists and
$\partial_s\partial_t\psi(t_0,0)= \partial_t\partial_s\psi(t_0,0).
$
\end{lemma}

\begin{lemma}\label{lemma:dtdmu}
Suppose $\hat V(t,x,\mu)$ is a function from
$ [0,T]\times \mathbb{R}^n\times {\mathcal P}_2(\mathbb{R}^n)$ to
$\mathbb{R} $, satisfying the following conditions:

(i) $\partial_t\hat V(t,x,\mu)$, $\delta_\mu \hat V(t, x, \mu;y)  $ and $\partial_t \delta_\mu \hat V(t, x, \mu; y)$  exist and are jointly continuous in
$(t, x, \mu) \in [0,T]\times \mathbb{R}^{n} \times {\mathcal P}_2(\mathbb{R}^n)$ and $(t, x, y,\mu) \in [0,T]\times \mathbb{R}^{2n} \times {\mathcal P}_2(\mathbb{R}^n)$, respectively;

(ii) for each constant  $K>0$,
\begin{align*}
| \delta_\mu \hat V(t, x, \mu; y)|,\
|\partial_t \delta_\mu \hat V(t, x, \mu; y)|\le C_K (1+|y|^2)
\end{align*}
holds for all $t\in [0,T]$, $x,y\in \mathbb{R}^n$, $\mu\in {\mathcal P}_2(\mathbb{R}^n)$ whenever $|x|\le K$ and $W_2(\mu,\delta_0)\le K$ (equivalently, $\langle \mu, |y|^2\rangle\le K^2$), where $C_K$ is a constant depending  on $K$.

 Then the derivative  $ \delta_\mu (\partial_t\hat V(t, x, \mu) )(y)$ exists, and moreover,
 \begin{align}\label{Vuttu}
 \delta_\mu (\partial_t\hat V(t, x, \mu) )(y)= \partial_t \delta_\mu \hat V(t, x, \mu; y)
\end{align}
for all
$t\in [0,T],x,y\in \mathbb{R}^n,\mu\in {\mathcal P}_2(\mathbb{R}^n)$.
\end{lemma}

\begin{proof}

Step 1. We show the normalization property
of $ \partial_t \delta_\mu \hat V(t, x, \mu; y)$.
By the normalization property of $\delta_\mu \hat V$, we have
$\langle \mu(dy), \delta_\mu \hat V(t, x, \mu; y)\rangle=0$,   
which implies
\begin{align*}
0&=\partial_t \langle \mu(dy), \delta_\mu \hat V(t, x, \mu; y)\rangle= \langle \mu(dy),\partial_t \delta_\mu \hat V(t, x, \mu; y)\rangle,
\end{align*}
due to condition (ii) and an application of the dominated convergence theorem.

Step 2.
Let $\mu_1$ and $\mu_2$ both from ${\mathcal P}_2(\mathbb{R}^n)$ be fixed.
For $t\in [0,T]$ and $s\in [0,1]$, set
\begin{align}
\psi(t,s) \coloneqq  \hat V(t, x, \mu_1+s\nu), \notag 
\end{align}
where $\nu= \mu_2-\mu_1$. We have
\begin{align}
\partial_t \psi(t,s)= \partial_t \hat V(t,x,\mu_1+s\nu) \label{ptstV}
\end{align}
for $(t,s)\in [0,T]\times [0,1]$.
We proceed to check $\partial_s \psi(t,s)$.
Suppose $s\in (0, 1)$. Then $\delta_\mu \hat V$ evaluated  at $\mu_1+s\nu$ gives
\begin{align*}
\lim_{\epsilon \downarrow 0}\frac{\psi(t, s+\epsilon)-\psi(t, s) }{\epsilon }=\langle \nu(dy), \delta_\mu \hat V(t,x, \ \mu_1+s\nu; y)\rangle.
\end{align*}
Next, we have
\begin{align*}
\lim_{\epsilon \downarrow 0}\frac{\psi(t, s-\epsilon)-\psi(t, s) }
{-\epsilon }= &-\lim_{\epsilon \downarrow 0} \frac{\hat V(t, x, \mu_1+s\nu+\epsilon (\mu_1-\mu_2))- \hat V(t, x, \mu_1+s\nu)   }{\epsilon }   \\
=&-\langle (\mu_1-\mu_2)(dy), \delta_\mu \hat V(t,x , \mu_1+s\nu; y)\rangle \\
=& \langle \nu(dy), \delta_\mu \hat V(t,x ,\mu_1+s\nu; y)\rangle.
\end{align*}
The second equality results from the definition of $\delta_\mu \hat V  $.
We  similarly obtain the one-sided derivative of $\psi(t,\cdot) $ as $\langle \nu(dy), \delta_\mu \hat V(t,x \, \mu_1+s\nu; y)\rangle$ at $s=0$ and $s=T$. Therefore for all $s\in [0,T]$,   $\partial_s \psi(t,s)$ exists   and
\begin{align}
\partial_s \psi(t,s)= \langle \nu(dy), \delta_\mu \hat V(t,x , \mu_1+s\nu; y)\rangle. \label{pspsits}
\end{align}
Subsequently, by \eqref{pspsits}, for each fixed $t\in (0,T)$, we have
\begin{align}
\partial_t\partial_s \psi(t,s)&=\partial_t \langle \nu(dy), \delta_\mu \hat V(t,x , \mu_1+s\nu; y)\rangle \notag  \\
&= \langle \nu(dy),\partial_t \delta_\mu \hat V(t,x ,
\mu_1+s\nu; y)\rangle, \quad s\in [0,T], \label{ptsmu1}
\end{align}
where the second equality results from   condition (ii). We similarly obtain \eqref{ptsmu1} for  $t=0$ and $t=T$.

Step 3. As $(t',s')\to (t,s)\in [0,T]\times [0,1]$, we may show $W_2(\mu_1+s'\nu, \mu_1+s\nu)\to 0$ using \cite[Theorem 6.9]{V09}, and therefore $\partial_t\delta_\mu \hat V (t',x, \mu_1+s'\nu; y)\to \partial_t\delta_\mu \hat V(t,x,\mu_1+s\nu;y)$ by condition (i). Subsequently, $\partial_t\partial_s \psi(t',s') \to\partial_t\partial_s \psi(t,s) $ by
\eqref{ptsmu1} and the dominated convergence theorem under condition (ii). Hence $\partial_t\partial_s \psi(t,s)$ is  continuous in $(t,s)\in [0,T]\times [0,1]$.  Therefore by Lemma \ref{lemma:sch}, $\partial_s\partial_t \psi(t,s)$ exists  at each point $(t, 0)$, $t\in [0,T]$, and moreover,
\begin{align}
 \partial_s\partial_t \psi(t,0) = \partial_t\partial_s \psi(t,0)=\langle \nu(dy),\partial_t \delta_\mu \hat V(t,x \, \mu_1; y)\rangle,
\label{ppstV1}
\end{align}
where the second equality follows from \eqref{ptsmu1}.
By \eqref{ptstV} and \eqref{ppstV1}, we have
\begin{align}
\lim_{\epsilon \downarrow 0 } \frac{1}{\epsilon}[ \partial_t \hat V(t, x, \mu_1+\epsilon \nu)-\partial_t \hat V(t, x, \mu_1) ]
& =\partial_s\partial_t \psi(t,0)\notag \\
& =\langle \nu(dy),\partial_t \delta_\mu \hat V(t,x \, \mu_1; y)\rangle \notag
\end{align}
for all $\mu_1\in {\mathcal P}_2(\mathbb{R}^n)$.
So $\delta_\mu (\partial_t \hat V(t,x,\mu_1))(y)$ exists. Recalling Step 1,
 we obtain \eqref{Vuttu}.
\end{proof}

\begin{proposition}  (\cite{CDLL19})
Suppose that both $\delta_{\mu}\psi(\mu;y)$ and $\psi_{\mu\mu}(\mu;y,z)$
exist and are jointly continuous in  their arguments and
that for each $K>0$,  $|\delta_{\mu}\psi(\mu;y)|, |\delta_{\mu\mu} \psi(\mu; y,z)|\le C_K(1+|y|^2+|z|^2)$ holds for all $y, z\in \mathbb{R}^n$ whenever $\mu$ satisfies $W_2(\mu, \delta_0)\le K$.
Then
\begin{align}\label{ssp}
\delta_{\mu\mu}\psi(\mu; y,z) =\delta_{\mu\mu} \psi(\mu; z,y)
+\delta_\mu \psi(\mu;y) -\delta_\mu \psi(\mu; z).
\end{align}
\end{proposition}

The equality \eqref{ssp} will be called the semi-symmetry property of $\delta_{\mu\mu}\psi $ and has been proved in \cite{CDLL19} for bounded derivatives. But the proof there can be easily adapted to the quadratic growth case.

\section*{Appendix C: Proof of Theorem \ref{theorem:U}}
\renewcommand{\theequation}{C.\arabic{equation}}
\setcounter{equation}{0}
\renewcommand{\thetheorem}{C.\arabic{theorem}}
\setcounter{theorem}{0}
\renewcommand{\thesubsection}{C.\arabic{subsection}}
\setcounter{subsection}{0}

It is clear that \eqref{defUU} satisfies the terminal condition \eqref{UTm}.
We will next differentiate both sides of equation \eqref{MEVwoCN} of $V$, and use further transformations to generate several equations. By adding up these equations, we can verify the solution \eqref{defUU} for $U$.

\subsection{The  equation of $V$}
 We redisplay the master equation of $V$:
\begin{align}
0=&\ \partial_t V(t, x, \mu)+   V_x (t, x, \mu)
f^*(t,x,  \mu)
 \notag \\
&+ \frac{1}{2} {\rm Tr}
[ V_{xx}(t, x, \mu)(\Sigma+ \Sigma_0) ] + L^*(t,x,  \mu)\notag \\
& +\langle  \mu(dz), \partial_z\delta_\mu V(t,x, \mu; z) f^*(t, z,\mu)\rangle \notag  \\
& +\frac{1}{2}\langle  \mu(dz), {\rm Tr}[
\partial_{z}^2 \delta_\mu
V(t, x, \mu; z) (\Sigma+\Sigma_0)]\rangle \notag \\
&+ \langle \mu(dz), {\rm Tr} [ \D_{xz}\delta_\mu V(t,x, \mu;z)  \Sigma_0 ]\rangle \notag  \\
&+\frac{1}{2}\langle \mu^{\otimes2} (dzdw), {\rm Tr}[\D_{zw}\delta_{\mu\mu} V (t, x, \mu; z,w)
\Sigma_0]\rangle ,\notag
\end{align}
where
 $(t,x,\mu)\in[0,T]\times \mathbb{R}^n\times
 {\cal P}_2(\mathbb{R}^n)  $.

\subsection{The equation  of $\partial_t\delta_\mu V$}

Using $x_i$ within $x=(x_1, \cdots, x_n)^T$ in place of
$t$ in Lemma \ref{lemma:dtdmu} and recalling the growth conditions of ${\cal C}_V$, we can show
$$
\partial_x\delta_\mu V(t,x,\mu;y)=(\delta_\mu \partial_x V(t,x,\mu))(y), \quad  \partial_{x}^2\delta_\mu V(t,x,\mu;y)=(\delta_\mu \partial_x^2 V(t,x,\mu))(y).
$$

Taking measure differentiation of the master equation of $V$ above and using $y$ as the newly generated variable, in view of Lemma \ref{lemma:dtdmu},
we have
{\allowdisplaybreaks
\begin{align}
0=& \partial_t\delta_\mu V(t, x, \mu;y) \label{pdedelV}  \\
  &+  \partial_{x} \delta_\mu V (t, x, \mu;y)
f^*(t,x,  \mu) +    V_x (t, x, \mu)
\delta_\mu f^*(t,x,  \mu;y)
 \notag \\
&+ \frac{1}{2} {\rm Tr}
[\partial_{x}^2\delta_\mu V(t, x, \mu;y)(\Sigma+ \Sigma_0) ]     +\delta_\mu L^*(t,x,  \mu;y)  \notag \\
& +\langle  \mu(dz), \partial_z\delta_{\mu\mu} V(t,x, \mu; z,y) f^*(t, z,\mu)\rangle \notag  \\
&+\langle  \mu(dz), \D_z\delta_\mu V(t,x, \mu; z)\delta_\mu f^*(t, z,\mu;y)\rangle \notag  \\
&\hskip -0cm +\frac{1}{2}\langle  \mu(dz), {\rm Tr}[
\D_{z}^2 \delta_{\mu\mu}
V(t, x, \mu; z,y) (\Sigma+\Sigma_0)]\rangle \notag \\
&+ \langle \mu(dz), {\rm Tr} [ \D_{xz}\delta_{\mu\mu} V(t,x, \mu;z,y)  \Sigma_0 ]\rangle \notag  \\
&+\frac{1}{2}\langle\mu^{\otimes2}(dzdw), {\rm Tr}[\D_{zw}\delta_{\mu\mu\mu}
 V (t, x, \mu; z,w,y)
\Sigma_0]\rangle ,\notag\\
&+ \D_y\delta_\mu V(t,x, \mu; y) f^*(t, y,\mu) +\chi_1(t,x,\mu) \notag  \\
& + \frac{1}{2} {\rm Tr}[
\D_{y}^2 \delta_\mu
V(t, x, \mu; y) (\Sigma+\Sigma_0)]+\chi_2(t,x,\mu)\notag  \\
&+  {\rm Tr} [ \D_{xy}\delta_\mu V(t,x, \mu;y)  \Sigma_0 ]+\chi_3(t,x,\mu)  \notag  \\
& + \frac{1}{2}\langle\mu(dw), {\rm Tr}[\D_{yw}\delta_{\mu\mu} V (t, x, \mu; y,w)
\Sigma_0]\rangle+\chi_4(t,x,\mu)\notag  \\
& + \frac{1}{2}\langle\mu(dz), {\rm Tr}[\D_{zy}\delta_{\mu\mu} V (t, x, \mu; z,y)
\Sigma_0]\rangle+\chi_5(t,x,\mu). \notag
\end{align}
}
In the above, the measure differentiation can get inside integration by dominated convergence.

\subsection{The summed equation}

By switching variables in \eqref{pdedelV} and integration, we  obtain two more  equations for
$ \partial_t \langle\mu(dy), \delta_\mu V(t, y, \mu;x)\rangle$
and $  \partial_t \langle\mu(dy), \delta_\mu V(t, y, \mu;y)\rangle$, respectively.
We have
\begin{align*}
0=&\partial_t V(t, x, \mu)+\partial_t M(t, \mu)
  + \partial_t \langle\mu(dy), \delta_\mu V(t, y, \mu;x)\rangle - \partial_t \langle\mu(dy), \delta_\mu V(t, y, \mu;y)\rangle\\
&+{\Psi}_V,
\end{align*}
where
{\allowdisplaybreaks
\begin{align*}
\Psi_V\coloneqq &
   \  V_x (t, x, \mu)
f^*(t,x,  \mu) \notag \\
&+ \frac{1}{2} {\rm Tr}
[ V_{xx}(t, x, \mu)(\Sigma+ \Sigma_0 ) ] + L^*(t,x,  \mu)\notag \\
& +\langle  \mu(dz), \D_z\delta_\mu V(t,x, \mu; z) f^*(t, z,\mu)\rangle \notag  \\
& +\frac{1}{2}\langle  \mu(dz), {\rm Tr}[
\D_{z}^2 \delta_\mu
V(t, x, \mu; z) (\Sigma+\Sigma_0)]\rangle \notag \\
&+ \langle \mu(dz), {\rm Tr} [ \D_{xz}\delta_\mu V(t,x, \mu;z)  \Sigma_0 ]\rangle \notag  \\
&+\frac{1}{2}\langle\mu^{\otimes2}(dzdw), {\rm Tr}[\D_{zw}\delta_{\mu\mu} V (t, x, \mu; z,w)
\Sigma_0]\rangle \notag\\
& \\
&+\langle  \mu(dy) ,  \partial_y \delta_\mu  M(t, \mu; y) f^*(t, y, \mu)     \rangle \\
& +\frac12 \langle \mu(dy), \mbox{Tr}[ \partial_{y}^2\delta_\mu M (t,\mu;y) (\Sigma+\Sigma_0)]\rangle \nonumber \\
& +\frac12 \langle \mu^{\otimes2}(dydz), \mbox{Tr}[ \partial_{yz}\delta_{\mu\mu} M (t,\mu;y,z) \Sigma_0]\rangle \nonumber \\
& + \langle \mu^{\otimes 2}(dydz),  \partial_{y}\delta_\mu V (t,z,\mu;y) [ \delta_\mu f^*(t,y,\mu;z)- \delta_\mu f^*(t,y,\mu;y) ]\rangle \nonumber \\
&+ \frac12 \langle \mu^{\otimes2}(dydw), \mbox{Tr}\{  [\partial_{yz}\delta_{\mu\mu} V (t,w,\mu;y,z)]|_{z=y} \Sigma\}\rangle,  \nonumber\\
&\\
  &+\langle\mu(dy),  \D_{y} \delta_\mu V (t, y, \mu;x)
f^*(t,y,  \mu) +  \D_{y}  V (t, y, \mu)
\delta_\mu f^*(t,y,  \mu;x)\rangle
 \notag \\
&+ \frac{1}{2} \langle \mu(dy),  {\rm Tr}
[\D_{y}^2\delta_\mu V(t, y, \mu;x)(\Sigma+ \Sigma_0  ) ]\rangle \\
& + \langle \mu(dy), \delta_\mu L^*(t,y,  \mu;x) \rangle \notag \\
& +\langle \mu^{\otimes2}(dydz), \D_z\delta_{\mu\mu} V(t,y, \mu; z,x) f^*(t, z,\mu)\rangle \notag  \\
&+\langle  \mu^{\otimes2}(dydz), \D_z\delta_\mu V(t,y, \mu; z)\delta_\mu f^*(t, z,\mu;x)\rangle \notag  \\
&\hskip -0cm +\frac{1}{2}\langle  \mu^{\otimes2}(dydz), {\rm Tr}[
\D_{z}^2 \delta_{\mu\mu}
V(t, y, \mu; z,x) (\Sigma+\Sigma_0)]\rangle \notag \\
&+ \langle \mu^{\otimes2}(dydz), {\rm Tr} [ \D_{yz}\delta_{\mu\mu} V(t,y, \mu;z,x)  \Sigma_0]\rangle \notag  \\
&+\frac{1}{2}\langle\mu^{\otimes3}(dydzdw), {\rm Tr}[\D_{zw}\delta_{\mu\mu\mu} V (t, y, \mu; z,w,x)
\Sigma_0]\rangle ,\notag\\
&+\langle\mu(dy), \D_x\delta_\mu V(t,y, \mu; x) f^*(t, x,\mu)\rangle \\
& + \frac{1}{2}\langle\mu(dy), {\rm Tr}[
\D_{x}^2 \delta_\mu
V(t, y, \mu; x) (\Sigma+\Sigma_0)]\rangle \\
&+\langle\mu(dy),  {\rm Tr} [ \D_{yx}\delta_\mu V(t,y, \mu;x)  \Sigma_0 ]\rangle  \\
& + \frac{1}{2}\langle\mu^{\otimes 2}(dwdy), {\rm Tr}[\D_{xw}\delta_{\mu\mu} V (t, y, \mu; x,w)
\Sigma_0]\rangle\\
& + \frac{1}{2}\langle\mu^{\otimes 2}(dzdy), {\rm Tr}[\D_{zx}\delta_{\mu\mu} V (t, y, \mu; z,x)
\Sigma_0]\rangle\\
&\\
 &-\langle\mu(dy),  [\D_{y} \delta_\mu V (t, y, \mu;x)]|_{x=y}
f^*(t,y,  \mu) +    V_y (t, y, \mu)
\delta_\mu f^*(t,y,  \mu;y)\rangle
 \notag \\
&- \frac{1}{2} \langle \mu(dy),  {\rm Tr}
\{[\D_{y}^2\delta_\mu V(t, y, \mu;x)]|_{x=y}(\Sigma+ \Sigma_0  ) \}\rangle \\
& - \langle\mu(dy),\ \delta_\mu L^*(t,y,  \mu;y) \rangle \notag \\
& -\langle \mu^{\otimes2}(dydz), \D_z\delta_{\mu\mu} V(t,y, \mu; z,y) f^*(t, z,\mu)\rangle \notag  \\
&-\langle  \mu^{\otimes2}(dydz), \D_z\delta_\mu V(t,y, \mu; z)\delta_\mu f^*(t, z,\mu;y)\rangle \notag  \\
&-\frac{1}{2}\langle  \mu^{\otimes2}(dydz), {\rm Tr}[
\D_{z}^2 \delta_{\mu\mu}
V(t, y, \mu; z,y) (\Sigma+\Sigma_0)]\rangle \notag \\
&- \langle \mu^{\otimes2}(dydz), {\rm Tr}\{ [ \D_{yz}\delta_{\mu\mu} V(t,y, \mu;z,x)]|_{x=y}  \Sigma_0 \}\rangle \notag  \\
&-\frac{1}{2}\langle\mu^{\otimes3}(dydzdw), {\rm Tr}[\D_{zw}\delta_{\mu\mu\mu} V (t, y, \mu; z,w,y)
\Sigma_0]\rangle ,\notag\\
&-\langle\mu(dy), [\D_x\delta_\mu V(t,y, \mu; x)]|_{x=y} f^*(t, y,\mu)\rangle \\
& - \frac{1}{2}\langle\mu(dy), {\rm Tr}\{[
\D_{x}^2 \delta_\mu
V(t, y, \mu; x)]|_{x=y} (\Sigma+\Sigma_0)\}\rangle \\
&-\langle\mu(dy),  {\rm Tr}\{ [ \D_{yx}\delta_\mu V(t,y, \mu;x)]|_{x=y}  \Sigma_0 \}\rangle  \\
& - \frac{1}{2}\langle\mu^{\otimes2}(dwdy), {\rm Tr}\{[\D_{xw}\delta_{\mu\mu} V (t, y, \mu; x,w)]|_{x=y}
\Sigma_0\}\rangle\\
& - \frac{1}{2}\langle\mu^{\otimes2}(dzdy), {\rm Tr}\{[\D_{zx}\delta_{\mu\mu} V (t, y, \mu; z,x)]|_{x=y}
\Sigma_0\}\rangle.
\end{align*}}
We have split the above sum into several groups for ease of reading.

On the other hand,
using the expressions  \eqref{defUUb} and \eqref{defUU} for  $\overline U$  and $U$,
we evaluate the right hand side of \eqref{MEUa}, excluding $\partial_t U$,
to get the sum
$$
\Psi_U\coloneqq \sum_{k=1}^9 G_k,
$$
with the components:
{\allowdisplaybreaks
\begin{align*}
G_1 \coloneqq \quad   & V_x(t,x, \mu) f^*(t,x,\mu) + \langle \mu(dy),  \partial_x \delta_\mu V(t, y,\mu; x)      \rangle f^*(t,x,\mu), \\
&  \\
G_2 \coloneqq \quad  &\frac12 \mbox{Tr}\{[  V_{xx}(t,x, \mu)+ \langle \mu(dy),   \partial_x^2\delta_\mu V(t, y,\mu; x)      \rangle] (\Sigma+\Sigma_0)\}+L^*(t, x, \mu),\\
& \\
G_3 \coloneqq  \quad & \langle \mu(dy), [ \partial_y \delta_\mu V(t, x, \mu;y) + \partial_y\delta_\mu M(t, \mu;y)] f^*(t,y,\mu)\rangle\\
&+ \langle \mu^{\otimes2}(dwdy), [ \partial_y\delta_{\mu\mu}  V(t, w,\mu; x,y)-\partial_y \delta_{\mu\mu}  V(t, w,\mu; w,y) ] f^*(t,y,\mu)\rangle  \\
&+\langle \mu(dy), \partial_{y}[    \delta_\mu  V(t, y,\mu; x)- \delta_\mu  V(t, y,\mu; y)  ] f^*(t,y,\mu)\rangle, \\
& \\
G_4 \coloneqq  \quad &  \frac12\langle \mu(dy),\mbox{Tr} \{[ \partial_y^2 \delta_\mu V(t, x, \mu;y) + \partial_y^2\delta_\mu M(t, \mu;y)] (\Sigma+\Sigma_0)\}\rangle\\
&+ \frac12\langle \mu^{\otimes2}(dwdy),\mbox{Tr}\{ [ \partial_y^2\delta_{\mu\mu}  V(t, w,\mu; x,y)-\partial_y^2 \delta_{\mu\mu}  V(t, w,\mu; w,y) ] (\Sigma+\Sigma_0)\}
 \rangle  \\
&+\frac12\langle \mu(dy),\mbox{Tr}\{ \partial_{y}^2[    \delta_\mu  V(t, y,\mu; x)- \delta_\mu  V(t, y,\mu; y)  ] (\Sigma+\Sigma_0)\} \rangle,\\
& \\
G_5 \coloneqq \quad & \langle \mu(dy),\mbox{Tr} [ \partial_{xy} \delta_\mu V(t, x, \mu;y)  \Sigma_0]\rangle 
+ \langle \mu^{\otimes2}(dwdy),\mbox{Tr} [ \partial_{xy}\delta_{\mu\mu}  V(t, w,\mu; x,y) \Sigma_0] \rangle  \\
&+\langle \mu(dy),\mbox{Tr}[ \partial_{xy}    \delta_\mu  V(t, y,\mu; x)\Sigma_0] \rangle, \\
&\\
G_6 \coloneqq  \quad &  \frac12\Big\langle\mu^{\otimes2}(dydz), {\rm Tr}\Big\{  \partial_{yz}\Big[  \delta_{\mu\mu} V(t, x, \mu;y,z) + \delta_{\mu\mu} M(t, \mu;y,z)\\
&+\langle\mu(dw),   \delta_{\mu\mu\mu}  V(t, w,\mu; x,y,z)- \delta_{\mu\mu\mu}  V(t, w,\mu; w,y,z) \rangle  \\
&+  \delta_{\mu\mu}  V(t, z,\mu; x,y)- \delta_{\mu\mu}  V(t, z,\mu; z,y)    \\
&+    \delta_{\mu\mu}  V(t, y,\mu; x,z)- \delta_{\mu\mu}  V(t, y,\mu; y,z)  \Big]\Sigma_0 \Big\}\Big\rangle, \\
 &\\
G_7 \coloneqq \quad  &  \langle \mu(dy) , \   \delta_\mu L^* (t,y,\mu; x) - \delta_\mu L^*(t,y,\mu; y) \rangle ,\nonumber  \\
&\\
G_8 \coloneqq \quad &  \langle \mu(dy) , \  \partial_y[ V(t,y,\mu)  +\langle \mu(dw), \delta_\mu V(t, w, \mu;y) \rangle ] \\
&\hskip 1.5cm \cdot    [\delta_\mu f^* (t,y,\mu; x) - \delta_\mu f^*(t,y,\mu; y)] \rangle ,\nonumber \\
&\\
G_9\coloneqq\quad  &   \frac{1}{2}\Big\langle \mu(dy),    \mbox{Tr} \Big\{\Big[\partial_{yz}\delta_\mu V(t, y, \mu; z) +\partial_{yz}\delta_\mu V(t, z, \mu; y) \\
& \hskip 4.5cm+\langle \mu(dw),  \partial_{yz}\delta_{\mu\mu} V (t,w,\mu;y,z) \rangle\Big]\big|_{z=y}\Sigma\Big\}  \Big\rangle.
\end{align*}
}
We have followed Section  \ref{sec:sub:no} for the notation of partial derivatives. For instance, in the last line of $G_3$, $\partial_y$ acts on two places of $\delta_\mu V(t,y,\mu;y)$.

\subsection{Reduction to simpler equations}
It suffices to show
\begin{align}
\Psi_V=\Psi_U. \label{psivu}
\end{align}
Both sides of \eqref{psivu} share many common components.
We can cancel out $G_1$, $G_2$, $G_7$ and $G_8$ from both sides. We further cancel out $G_5$ from both sides after rewriting $G_5$ using \eqref{VSig} and the semi-symmetry property \eqref{ssp}.
Subsequently, we further cancel out several terms in
$G_3$, $G_4$, $G_6$ and $G_9$ that are shared by $\Psi_V$.
Now to show \eqref{psivu}, it suffices to show
\begin{align}
\Psi_V^1=\Psi_U^1, \label{pvu1}
\end{align}
where
{\allowdisplaybreaks

\begin{align*}
\Psi_V^1\coloneqq& \langle \mu^{\otimes2}(dydz), \D_z\delta_{\mu\mu} V(t,y, \mu; z,x) f^*(t, z,\mu)\rangle \notag  \\
&\hskip -0cm +\frac{1}{2}\langle  \mu^{\otimes2}(dydz), {\rm Tr}[
\D_{z}^2 \delta_{\mu\mu}
V(t, y, \mu; z,x) (\Sigma+\Sigma_0)]\rangle \notag \\
&+ \langle \mu^{\otimes2}(dydz), {\rm Tr} [ \D_{yz}\delta_{\mu\mu} V(t,y, \mu;z,x)  \Sigma_0]\rangle \notag  \\
&+\frac{1}{2}\langle\mu^{\otimes3}(dydzdw), {\rm Tr}[\D_{zw}\delta_{\mu\mu\mu} V (t, y, \mu; z,w,x)
\Sigma_0]\rangle \notag\\
&\\
 &-\langle\mu(dy),  [\D_{y} \delta_\mu V (t, y, \mu;x)]|_{x=y}
f^*(t,y,  \mu) \rangle
 \notag \\
&- \frac{1}{2} \langle \mu(dy),  {\rm Tr}
\{[\D_{y}^2\delta_\mu V(t, y, \mu;x)]|_{x=y}(\Sigma+ \Sigma_0  ) \}\rangle \\
& -\langle \mu^{\otimes2}(dydz), \D_z\delta_{\mu\mu} V(t,y, \mu; z,y) f^*(t, z,\mu)\rangle \notag  \\
&-\frac{1}{2}\langle  \mu^{\otimes2}(dydz), {\rm Tr}\{
\D_{z}^2 \delta_{\mu\mu}
V(t, y, \mu; z,y) (\Sigma+\Sigma_0)\}\rangle \notag \\
&- \langle \mu^{\otimes2}(dydz), {\rm Tr}\{ [ \D_{yz}\delta_{\mu\mu} V(t,y, \mu;z,x)]|_{x=y}  \Sigma_0 \}\rangle \notag  \\
&-\frac{1}{2}\langle\mu^{\otimes3}(dydzdw), {\rm Tr}[\D_{zw}\delta_{\mu\mu\mu} V (t, y, \mu; z,w,y)
\Sigma_0]\rangle ,\notag\\
&-\langle\mu(dy), [\D_x\delta_\mu V(t,y, \mu; x)]|_{x=y} f^*(t, y,\mu)\rangle \\
& - \frac{1}{2}\langle\mu(dy), {\rm Tr}\{[
\D_{x}^2 \delta_\mu
V(t, y, \mu; x)]|_{x=y} (\Sigma+\Sigma_0)\}\rangle \\
&-\langle\mu(dy),  {\rm Tr}\{ [ \D_{yx}\delta_\mu V(t,y, \mu;x)]|_{x=y}  \Sigma_0 \}\rangle  \\
& - \frac{1}{2}\langle\mu^{\otimes2}(dwdy), {\rm Tr}\{[\D_{xw}\delta_{\mu\mu} V (t, y, \mu; x,w)]|_{x=y}
\Sigma_0\}\rangle\\
& - \frac{1}{2}\langle\mu^{\otimes2}(dzdy), {\rm Tr}\{[\D_{zx}\delta_{\mu\mu} V (t, y, \mu; z,x)]|_{x=y}
\Sigma_0\}\rangle
\end{align*}
}
and
$$
\Psi_U^1\coloneqq G_3^1+G_4^1+G_6^1+G_9^1,
$$
with the components:
{\allowdisplaybreaks
\begin{align*}
G_3^1 \coloneqq  \quad & \langle \mu^{\otimes2}(dwdy), [ \partial_y\delta_{\mu\mu}  V(t, w,\mu; x,y)-\partial_y \delta_{\mu\mu}  V(t, w,\mu; w,y) ] f^*(t,y,\mu)\rangle  \\
&-\langle \mu(dy), \partial_{y}[   \delta_\mu  V(t, y,\mu; y)  ]
f^*(t,y,\mu)\rangle ,\\
& \\
G_4^1 \coloneqq  \quad
& \frac12\langle \mu^{\otimes2}(dwdy),\mbox{Tr}\{ [ \partial_y^2\delta_{\mu\mu}  V(t, w,\mu; x,y)-\partial_y^2 \delta_{\mu\mu}  V(t, w,\mu; w,y) ] (\Sigma+\Sigma_0)\}\rangle  \\
&-\frac12\langle \mu(dy),\mbox{Tr}\{ \partial_{y}^2[  \delta_\mu  V(t, y,\mu; y)  ] (\Sigma+\Sigma_0)\} \rangle,\\
& \\
G_6^1 \coloneqq  \quad &  \frac12\Big\langle\mu^{\otimes2}(dydz), {\rm Tr} \Big\{ \partial_{yz}\Big[
\langle\mu(dw),   \delta_{\mu\mu\mu}  V(t, w,\mu; x,y,z)- \delta_{\mu\mu\mu}  V(t, w,\mu; w,y,z) \rangle  \\
&+  \delta_{\mu\mu}  V(t, z,\mu; x,y)- \delta_{\mu\mu}  V(t, z,\mu; z,y)    \\
&+    \delta_{\mu\mu}  V(t, y,\mu; x,z)- \delta_{\mu\mu}  V(t, y,\mu; y,z)  \Big]\Sigma_0  \Big\}\Big\rangle, \\
&\\
G_9^1\coloneqq\quad  &   \frac{1}{2}\Big\langle \mu(dy),    \mbox{Tr} \Big\{\Big[\partial_{yz}\delta_\mu V(t, y, \mu; z) +\partial_{yz}\delta_\mu V(t, z, \mu; y) \Big]\Big|_{z=y}\Sigma\Big\} \Big \rangle.
\end{align*}
}

We use the relation
$$
\partial_y [\delta_\mu  V(t, y,\mu; y)]=\partial_y [\delta_\mu  V(t, y,\mu; x)]_{x=y}+\partial_x [\delta_\mu  V(t, y,\mu; x)]_{x=y}
$$
and next the semi-symmetry property \eqref{ssp} of $\delta_{\mu\mu}V$ to rewrite $G_3^1$.
This eliminates $G_3^1$ as a common part of both sides of \eqref{pvu1}.
Now we only need to show
\begin{align}\label{pp22}
\Psi_V^2=\Psi_U^2,
\end{align}
where
{\allowdisplaybreaks

\begin{align*}
\Psi_V^2\coloneqq
&\hskip -0cm \frac{1}{2}\langle  \mu^{\otimes2}(dydz), {\rm Tr}[
\D_{z}^2 \delta_{\mu\mu}
V(t, y, \mu; z,x) (\Sigma+\Sigma_0)]\rangle \notag \\
&+ \langle \mu^{\otimes2}(dydz), {\rm Tr} [ \D_{yz}\delta_{\mu\mu} V(t,y, \mu;z,x)  \Sigma_0]\rangle \notag  \\
&+\frac{1}{2}\langle\mu^{\otimes3}(dydzdw), {\rm Tr}[\D_{zw}\delta_{\mu\mu\mu} V (t, y, \mu; z,w,x)
\Sigma_0]\rangle \notag\\
&\\
&- \frac{1}{2} \langle \mu(dy),  {\rm Tr}
\{[\D_{y}^2\delta_\mu V(t, y, \mu;x)]|_{x=y}(\Sigma+ \Sigma_0  ) \}\rangle \\
&-\frac{1}{2}\langle  \mu^{\otimes2}(dydz), {\rm Tr}\{
\D_{z}^2 \delta_{\mu\mu}
V(t, y, \mu; z,y) (\Sigma+\Sigma_0)\}\rangle \notag \\
&- \langle \mu^{\otimes2}(dydz), {\rm Tr}\{ [ \D_{yz}\delta_{\mu\mu} V(t,y, \mu;z,x)]|_{x=y}  \Sigma_0 \}\rangle \notag  \\
&-\frac{1}{2}\langle\mu^{\otimes3}(dydzdw), {\rm Tr}[\D_{zw}\delta_{\mu\mu\mu} V (t, y, \mu; z,w,y)
\Sigma_0]\rangle \notag\\
& - \frac{1}{2}\langle\mu(dy), {\rm Tr}\{[
\D_{x}^2 \delta_\mu
V(t, y, \mu; x)]|_{x=y} (\Sigma+\Sigma_0)\}\rangle \\
&-\langle\mu(dy),  {\rm Tr}\{ [ \D_{yx}\delta_\mu V(t,y, \mu;x)]|_{x=y}  \Sigma_0 \}\rangle  \\
& - \frac{1}{2}\langle\mu^{\otimes2}(dwdy), {\rm Tr}\{[\D_{xw}\delta_{\mu\mu} V (t, y, \mu; x,w)]|_{x=y}
\Sigma_0\}\rangle\\
& - \frac{1}{2}\langle\mu^{\otimes2}(dzdy), {\rm Tr}\{[\D_{zx}\delta_{\mu\mu} V (t, y, \mu; z,x)]|_{x=y}
\Sigma_0\}\rangle
\end{align*}
}
and
$$
\Psi_U^2\coloneqq G_4^1+G_6^1+G_9^1
$$
with the components $G_4^1$, $G_6^1$ and $G_9^1$ taking the same form as in $\Psi_U^1$.

We rewrite $G_4^1$ using the semi-symmetry property \eqref{ssp} of $\delta_{\mu\mu} V$
 and next the relation
 $$
\partial_y^2 [\psi(y,y)]=[\partial_{xx} \psi(x,y) + \partial_{yy} \psi(x,y)
+ \partial_{xy} \psi(x,y)+\partial_{yx} \psi(x,y)]_{x=y}.
 $$
Subsequently, we cancel out $G_4^1+G_9^1$ as a common part of
$\Psi_V^2$ and $\Psi_U^2$.
Then to show \eqref{pp22}, it suffices to show
\begin{align}\label{pg31}
\Psi^3_V=G_6^1,
\end{align}
where $G_6^1$ is the same as in $\Psi_U^1$ and
{\allowdisplaybreaks
\begin{align*}
\Psi_V^3 \coloneqq
& \langle \mu^{\otimes2}(dydz), {\rm Tr} [ \D_{yz}\delta_{\mu\mu} V(t,y, \mu;z,x)  \Sigma_0]\rangle \notag  \\
&+\frac{1}{2}\langle\mu^{\otimes3}(dydzdw), {\rm Tr}[\D_{zw}\delta_{\mu\mu\mu} V (t, y, \mu; z,w,x)
\Sigma_0]\rangle \notag\\
&- \langle \mu^{\otimes2}(dydz), {\rm Tr}\{ [ \D_{yz}\delta_{\mu\mu} V(t,y, \mu;z,x)]|_{x=y}  \Sigma_0 \}\rangle \notag  \\
&-\frac{1}{2}\langle\mu^{\otimes3}(dydzdw), {\rm Tr}[\D_{zw}\delta_{\mu\mu\mu} V (t, y, \mu; z,w,y)
\Sigma_0]\rangle \notag\\
& - \frac{1}{2}\langle\mu^{\otimes2}(dwdy), {\rm Tr}\{[\D_{xw}\delta_{\mu\mu} V (t, y, \mu; x,w)]|_{x=y}
\Sigma_0\}\rangle\\
& - \frac{1}{2}\langle\mu^{\otimes2}(dzdy), {\rm Tr}\{[\D_{zx}\delta_{\mu\mu} V (t, y, \mu; z,x)]|_{x=y}
\Sigma_0\}\rangle.
\end{align*}
}

Using the semi-symmetry property \eqref{ssp} twice, we obtain
\begin{align*}
\delta_{\mu\mu\mu} V(t, w, \mu; x,y,z) =& \delta_{\mu\mu\mu} V(t, w, \mu; y,z,x)+ \delta_{\mu\mu} V(t, w, \mu; y,x)\\
&- 2 \delta_{\mu\mu} V(t, w, \mu; y,z)+\delta_{\mu\mu} V(t, w, \mu; x,z),
\end{align*}
and subsequently,
\begin{align*}
&\partial_{yz}[\delta_{\mu\mu\mu} V(t, w, \mu; x,y,z)-\delta_{\mu\mu\mu} V(t, w, \mu; w,y,z)]\\
&=\partial_{yz}[\delta_{\mu\mu\mu} V(t, w, \mu; y,z,x)-\delta_{\mu\mu\mu} V(t, w, \mu; y,z,w)].
\end{align*}
Now to show \eqref{pg31}, we only need to show
$$
\Psi^4_V= G_6^2,
$$
where
{\allowdisplaybreaks

\begin{align*}
\Psi_V^4 \coloneqq
& \langle \mu^{\otimes2}(dydz), {\rm Tr} [ \D_{yz}\delta_{\mu\mu} V(t,y, \mu;z,x)  \Sigma_0]\rangle \notag  \\
&- \langle \mu^{\otimes2}(dydz), {\rm Tr}\{ [ \D_{yz}\delta_{\mu\mu} V(t,y, \mu;z,x)]|_{x=y}  \Sigma_0 \}\rangle \notag  \\
& - \frac{1}{2}\langle\mu^{\otimes2}(dwdy), {\rm Tr}\{[\D_{xw}\delta_{\mu\mu} V (t, y, \mu; x,w)]|_{x=y}
\Sigma_0\}\rangle\\
& - \frac{1}{2}\langle\mu^{\otimes2}(dzdy), {\rm Tr}\{[\D_{zx}\delta_{\mu\mu} V (t, y, \mu; z,x)]|_{x=y}
\Sigma_0\}\rangle
\end{align*}
}
and
\begin{align*}
G_6^2 \coloneqq  \quad &  \frac12\Big\langle\mu^{\otimes2}(dydz), {\rm Tr}\Big\{  \partial_{yz}\Big[
  \delta_{\mu\mu}  V(t, z,\mu; x,y)- \delta_{\mu\mu}  V(t, z,\mu; z,y)    \\
&+    \delta_{\mu\mu}  V(t, y,\mu; x,z)- \delta_{\mu\mu}  V(t, y,\mu; y,z)  \Big]\Sigma_0\Big\}\Big\rangle.
\end{align*}
After transforming the $x$-dependent terms of $G_6^2$ using the semi-symmetry property \eqref{ssp} and comparing with $\Psi_V^4$ to get rid of shared terms,
 to show \eqref{pg31}, we only need to show
\begin{align}
& \langle \mu^{\otimes2}(dydz), {\rm Tr}\{ [ \D_{yz}\delta_{\mu\mu} V(t,y, \mu;z,x)]|_{x=y}  \Sigma_0 \}\rangle \label{muvc6}  \\
& + \frac{1}{2}\langle\mu^{\otimes2}(dwdy), {\rm Tr}\{[\D_{xw}\delta_{\mu\mu} V (t, y, \mu; x,w)]|_{x=y}
\Sigma_0\}\rangle \notag \\
&  +\frac{1}{2}\langle\mu^{\otimes2}(dzdy), {\rm Tr}\{[\D_{zx}\delta_{\mu\mu} V (t, y, \mu; z,x)]|_{x=y}
\Sigma_0\}\rangle \notag  \\
=&\frac12\big\langle\mu^{\otimes2}(dydz), {\rm Tr}\{  \partial_{yz}[
  \delta_{\mu\mu}  V(t, z,\mu; z,y) + \delta_{\mu\mu}  V(t, y,\mu; y,z)  ]\Sigma_0\}\big\rangle \notag \\
&+ \langle \mu^{\otimes2}(dydz), \  {\rm Tr}[  \partial_{yz}
  \delta_{\mu}  V(t, y,\mu;z)\Sigma_0 ]  \rangle. \notag
\end{align}
 After changing the notation on the  LHS and expanding the partial differentiation on the RHS, \eqref{muvc6} is equivalently written as
\begin{align*}
& \langle \mu^{\otimes2}(dydz), {\rm Tr}\{ [ \D_{yz}\delta_{\mu\mu} V(t,y, \mu;z,x)]|_{x=y}  \Sigma_0 \}\rangle \notag  \\
& + \frac{1}{2}\langle\mu^{\otimes2}(dydz), {\rm Tr}\{[\D_{yz}\delta_{\mu\mu} V (t, w, \mu; y,z)]|_{w=y}
\Sigma_0\}\rangle\\
&  +\frac{1}{2}\langle\mu^{\otimes2}(dydz), {\rm Tr}\{[\D_{zy}\delta_{\mu\mu} V (t, w, \mu; z,y)]|_{w=y}
\Sigma_0\}\rangle\\
=&\frac12\Big\langle\mu^{\otimes2}(dydz), {\rm Tr}\Big\{ \Big( [\partial_{yz}
  \delta_{\mu\mu}  V(t, x,\mu; z,y)]|_{x=z} +[\partial_{yz} \delta_{\mu\mu}  V(t, z,\mu; x,y)]|_{x=z}\\
  &\qquad\qquad +[\partial_{yz} \delta_{\mu\mu}  V(t, x,\mu; y,z)]|_{x=y} +[\partial_{yz}\delta_{\mu\mu}  V(t, y,\mu; x,z)]|_{x=y}  \Big)\Sigma_0\Big\}\Big\rangle \\
&+ \langle \mu^{\otimes2}(dydz), \  {\rm Tr} [ \partial_{yz}
  \delta_{\mu}  V(t, y,\mu;z)\Sigma_0 ]  \rangle,
\end{align*}
which, after cancellation of the last two terms of LHS and next rewriting the first term of LHS  by the semi-symmetry property \eqref{ssp}, reduces to
\begin{align*}
& \langle \mu^{\otimes2}(dydz), {\rm Tr}\{ \D_{yz}[\delta_{\mu\mu} V(t,y, \mu;x,z)  +\delta_{\mu} V(t,y,\mu;z) - \delta_{\mu} V(t,y,\mu;x)  ]  \Sigma_0 \}|_{x=y}\rangle \notag  \\
=&\frac12\Big\langle\mu^{\otimes2}(dydz), {\rm Tr} \Big\{\Big( [\partial_{yz} \delta_{\mu\mu}  V(t, z,\mu; x,y)]|_{x=z}  +[\partial_{yz}\delta_{\mu\mu}  V(t, y,\mu; x,z)]|_{x=y}  \Big)\Sigma_0\Big\}\Big\rangle \\
&+ \langle \mu^{\otimes2}(dydz), \  {\rm Tr[}  \partial_{yz}
  \delta_{\mu}  V(t, y,\mu;z)\Sigma_0   ]\rangle.
\end{align*}
The above equality is equivalent to
\begin{align*}
& \langle \mu^{\otimes2}(dydz), {\rm Tr}[ \D_{yz}\delta_{\mu\mu} V(t,y, \mu;x,z)     \Sigma_0 ]|_{x=y}\rangle \notag  \\
 &=\frac12\big\langle\mu^{\otimes2}(dydz), {\rm Tr}\big\{ \big( [\partial_{yz} \delta_{\mu\mu}  V(t, z,\mu; x,y)]|_{x=z}  +[\partial_{yz}\delta_{\mu\mu}  V(t, y,\mu; x,z)]|_{x=y}  \big)\Sigma_0\big\}\big\rangle .
\end{align*}
The last equality indeed holds in view of Lemma \ref{lemma:psixy} and equality \eqref{Dtr}.
We conclude that \eqref{psivu} holds, which completes the proof.

\section*{Appendix D: Derivation of  Equations \eqref{MEUa} and \eqref{MEUb} for $U$ and $\overline U$}

\renewcommand{\theequation}{D.\arabic{equation}}
\setcounter{equation}{0}
\renewcommand{\thetheorem}{D.\arabic{theorem}}
\setcounter{theorem}{0}
\renewcommand{\thesubsection}{D.\arabic{subsection}}
\setcounter{subsection}{0}

Let $U_{\rm soc}^N(t,x^1, \mu^{-1})$
be the social cost in \eqref{JNphi} and \eqref{U1nrep}.
In the formal derivation below, we suppose the functions behave well so that the small error terms $o(\epsilon)$ holds uniformly with respect to all $N$.
We have
\begin{align}
U_{\rm soc}^N(t,x^1, \mu^{-1})& =\EE\Big\{
\int_t^{t+\epsilon}\Big[ L^*(s, X_s^1,
\mu^{-1}_s)+\sum_{k=2}^N L^*(s, X_s^k,
\mu^{-k}_s)\Big ]ds \label{Udeco} \\
& \quad +U^N(t+\epsilon, X_{t+\epsilon}^1, \mu^{-1}_{t+\epsilon}) + (N-1) \overline U(t+\epsilon, \mu^{-1}_{t+\epsilon})\Big\} \notag  \\
&= L^*(t,x^1, \mu^{-1}) \epsilon + \sum_{k=2}^N L^*(t,x^k, \mu^{-k}) \epsilon
+No(\epsilon)\notag  \\
&  \quad +\EE U^N(t+\epsilon, X_{t+\epsilon}^1, \mu^{-1}_{t+\epsilon}) + (N-1)\EE \overline U(t+\epsilon, \mu^{-1}_{t+\epsilon}).\notag
\end{align}

We make the expansion
\begin{align*}
\sum_{k=2}^N L^*(s, x^k,
\mu^{-k}) =&
\sum_{k=2}^N  L^*(s, x^k,
\mu^{-1})\\
& +\sum_{k=2}^N [\langle ( \mu^{-k}-\mu^{-1})(dy) , \delta_\mu L^*( t, x^k, \mu^{-1};y) \rangle + o(1/N)]\\
=&  (N-1)\langle \mu^{-1} (dy) , L^*(t,y, \mu^{-1})  \rangle\\
& +\sum_{k=2}^N \langle (\mu^{-k}-\mu^{-1})(dy) ,\ \delta_\mu L^*( t, x^k, \mu^{-1};y) \rangle \quad(\eqqcolon\xi_1)  \\
&+o(1),
\end{align*}
where $o(1)\to 0$ as $N\to \infty$.
We have
\begin{align*}
\xi_1&= \frac{1}{N-1}\sum_{k=2}^N [\delta_\mu L^*( t, x^k, \mu^{-1};x^1)  -  \delta_\mu L^*( t, x^k, \mu^{-1};x^k)   ]\\
 &=\langle\mu^{-1}(dy), \    \delta_\mu L^*( t, y, \mu^{-1};x^1)  -  \delta_\mu L^*( t, y, \mu^{-1};y)    \rangle .
\end{align*}

In the following we have such designated variables $x,y,z$ as in the functions
$U^N(t,x,\mu)$, $\partial_\mu U^N(t, x, \mu; y) $,
$\partial_{\mu\mu} U^N(t, x, \mu; y,z)$, $\delta_\mu\overline U(t,\mu; y)$ and
$ \delta_{\mu\mu} \overline U(t, \mu; y,z) $.
Denote $\Delta \mu^{-k}\coloneqq \mu_{t+\epsilon}^{-k}- \mu^{-k} $ for all $1\le k\le N $.
Now we  formally make the expansion
\begin{align*}
&\hskip -0.5cm  U^N(t+\epsilon, X_{t+\epsilon}^1, \mu^{-1}_{t+\epsilon})\\
=\ &
U^N(t, x^1, \mu^{-1}) + \partial_t U^N(t, x^1, \mu^{-1}) \epsilon \\
& + \partial_x U^N(t,x^1, \mu^{-1})f^*(t, x^1, \mu^{-1}) \epsilon  + \frac12 \mbox{Tr}[ \partial_x^2 U^N(t, x^1, \mu^{-1}) (\Sigma+\Sigma_0)] \epsilon \\
&+ \langle \Delta\mu^{-1} (dy), \delta_\mu U^N(t, x^1, \mu^{-1}; y)     \rangle \quad ( \eqqcolon \xi_2) \\
& +\langle \Delta\mu^{-1} (dy),  \partial_x\delta_\mu  U^N(t, x^1, \mu^{-1}; y)\\
 & \hskip 3.2cm  \cdot[\sigma_0 (W_{t+\epsilon }^0-W_t^0)+\sigma(W^1_{t+\epsilon}-W^1_t)]    \rangle
 \quad (\eqqcolon \xi_3)    \\
&+\frac12 \langle \Delta\mu^{-1} (dy)\Delta\mu^{-1} (dz),  \delta_{\mu\mu} U^N(t, x^1, \mu^{-1}; y,z)
\rangle \quad (\eqqcolon \xi_4) \\
&+o(\epsilon).
\end{align*}
We similarly have
\begin{align*}
 \overline U (t+\epsilon, \mu^{-1}_{t+\epsilon})& =
 \overline U(t, \mu^{-1}) +\partial_t  \overline U (t, \mu^{-1)})\epsilon  \\
 & +
\langle \Delta\mu^{-1}(dy), \delta_\mu
\overline U (t, \mu^{-1};y)  \rangle \quad (\eqqcolon\xi_5)  \\
&+\frac12
\langle \Delta\mu^{-1} (dy)\Delta\mu^{-1} (dz) , \delta_{\mu\mu} \overline U (t,\mu^{-1};y,z ) \rangle \quad (\eqqcolon\xi_6) \\
& +o(\epsilon).
\end{align*}

We have
\begin{align*}
\EE\xi_2 &=\frac{1}{N-1}\EE\sum_{k=2}^N [\delta_\mu U^N(t, x^1, \mu^{-1}; X^k_{t+\epsilon}) -\delta_\mu U^N(t, x^1, \mu^{-1}; x^k) ]\\
&=\frac{1}{N-1}\sum_{k=2}^N \Big\{\partial_y \delta_\mu U^N(t, x^1, \mu^{-1}; x^k)
f^*(t,x^k, \mu^{-k}) \epsilon\\
&\qquad\qquad\qquad + \frac12 \mbox{Tr} [\partial_y^2 \delta_\mu U^N(t, x^1, \mu^{-1}; x^k) (\Sigma+\Sigma_0)]\epsilon +o(\epsilon)\Big\}\\
&= \epsilon  \langle  \mu^{-1}(dy), \  \partial_y \delta_\mu U^N(t, x^1, \mu^{-1}; y)
f^*(t,y, \mu^{-1})    \rangle\\
&  \quad  +\epsilon  \langle\mu^{-1}(dy) ,\   \frac12 \mbox{Tr} [\partial_y^2 \delta_\mu U^N(t, x^1, \mu^{-1}; y) (\Sigma+\Sigma_0)]  \rangle +o(\epsilon)  ,
\end{align*}
and
\begin{align*}
\EE\xi_3=& \frac{1}{N-1} \sum_{k=2}^N\Big\{[ \partial_x\delta_\mu  U^N(t, x^1, \mu^{-1}; X^k_{t+\epsilon}) - \partial_x\delta_\mu  U^N(t, x^1, \mu^{-1}; x^k) ] \\
 & \hskip 3.2cm  \cdot[\sigma_0 (W_{t+\epsilon }^0-W_t^0)+\sigma(W^1_{t+\epsilon}-W^1_t)]\Big\}      \\
& =\langle \mu^{-1}(dy),\ \mbox{Tr} [\partial_{xy}\delta_\mu U^N(t,x^1, \mu; y)\Sigma_0 ]\rangle \epsilon +o(\epsilon).
\end{align*}

Next we have
\begin{align*}
\xi_4=&  \frac{1}{2(N-1)^2} \sum_{j,k\ge 2}\Big[ \delta_{\mu\mu}U^N(t, x^1, \mu^{-1}, X_{t+\epsilon}^j, X_{t+\epsilon}^k)- \delta_{\mu\mu}U^N(t, x^1, \mu^{-1}, X_{t}^j, X_{t+\epsilon}^k)\\
&\qquad\qquad\qquad\quad - \delta_{\mu\mu}U^N(t, x^1, \mu^{-1}, X_{t+\epsilon}^j, X_{t}^k)+\delta_{\mu\mu}U^N(t, x^1, \mu^{-1}, X_{t}^j, X_{t}^k)\Big],
\end{align*}
and
\begin{align*}
\EE \xi_4&= \frac{1}{2(N-1)^2}\EE\sum_{j,k\ge 2}\Big\{ \mbox{Tr}\Big[\partial_{yz} \delta_{\mu\mu} U^N(t, x^1, \mu^{-1};x^j,x^k)\\
&\hskip 3cm\cdot (X_{t+\epsilon}^k-X_t^k)(X_{t+\epsilon}^j-X_t^j)^T\Big] +o(\epsilon)\Big\}  \\
&= \frac{1}{2(N-1)^2}\EE\sum_{j,k\ge 2}\Big\{ \mbox{Tr}[\partial_{yz} \delta_{\mu\mu} U^N(t, x^1, \mu^{-1};x^j,x^k) \Sigma_0]\epsilon +o(\epsilon)\Big\} \\
&\quad + \frac{1}{2(N-1)^2}\EE\sum_{k\ge 2}\Big\{ \mbox{Tr}[\partial_{yz} \delta_{\mu\mu} U^N(t, x^1, \mu^{-1};x^k,x^k) \Sigma] \epsilon +o(\epsilon)\Big\}\\
& = \frac12\langle\mu^{-1}(dy) \mu^{-1}(dz), \    \mbox{Tr[}\partial_{yz} \delta_{\mu\mu} U^N(t, x^1, \mu^{-1};y,z) \Sigma_0  ]\rangle\epsilon \\
&\qquad  +  \frac{1}{2(N-1)}\langle\mu^{-1}(dy) , \    \mbox{Tr}\{[\partial_{yz} \delta_{\mu\mu} U^N(t, x^1, \mu^{-1};y,z)]|_{z=y} \Sigma \} \rangle\epsilon
+o(\epsilon).
\end{align*}

We have {\allowdisplaybreaks
\begin{align*}
\EE\xi_5 & = \frac{1}{N-1}\EE\sum_{k=2}^N
[\delta_\mu \overline U(t, \mu^{-1}; X^k_{t+\epsilon})-\delta_\mu \overline U(t, \mu^{-1}; x^k)] \\
&= \frac{1}{N-1}\sum_{k=2}^N
\Big\{\partial_y\delta_\mu \overline U(t, \mu^{-1}; x^k) f^*(t, x^k, \mu^{-k})\\ & \qquad \qquad \qquad +\frac12 \mbox{Tr} [\partial_{y}^2\delta_\mu \overline U(t, \mu^{-1}; x^k)(\Sigma+\Sigma_0) ]\Big\}\epsilon +o(\epsilon)\\
&= \frac{1}{N-1}\sum_{k=2}^N
\Big\{\partial_y\delta_\mu \overline U(t, \mu^{-1}; x^k) f^*(t, x^k, \mu^{-1}) \\
&\qquad \qquad \qquad   +\frac12 \mbox{Tr} [\partial_{y}^2\delta_\mu \overline U(t, \mu^{-1}; x^k)(\Sigma+\Sigma_0) ]\Big\}\epsilon +o(\epsilon)\\
&\quad +  \frac{1}{N-1}\sum_{k=2}^N \partial_y\delta_\mu \overline U(t, \mu^{-1}; x^k)\\
&\hskip 3cm\cdot [\langle (\mu^{-k}-\mu^{-1})(dy),\  \delta_\mu f^*(t, x^k, \mu^{-1}; y) \rangle + o(1/N)]\epsilon \\
&=\frac{1}{N-1}\sum_{k=2}^N
\Big\{\partial_y\delta_\mu \overline U(t, \mu^{-1}; x^k) f^*(t, x^k, \mu^{-1}) \\
&\qquad \qquad \qquad    +\frac12 \mbox{Tr}[\partial_{y}^2\delta_\mu \overline U(t, \mu^{-1}; x^k)(\Sigma+\Sigma_0) ]\Big\}\epsilon \\
&\quad +  \frac{1}{(N-1)^2}\sum_{k=2}^N\Big\{ \partial_y\delta_\mu \overline U(t, \mu^{-1}; x^k)\\
& \hskip 3cm \cdot  [  \delta_\mu f^*(t, x^k, \mu^{-1}; x^1)- \delta_\mu f^*(t, x^k, \mu^{-1}; x^k) ]\epsilon \Big\}\\
&\quad   + o(1/N)\epsilon+o(\epsilon).
\end{align*}
}
So we have
\begin{align*}
(N-1) \EE \xi_5= & \epsilon (N-1)
\Big\langle \mu^{-1}(dy), \ \partial_y\delta_\mu \overline U(t, \mu^{-1}; y) f^*(t, y, \mu^{-1})\\
&\qquad\qquad \qquad\qquad +\frac12 \mbox{Tr}[\partial_{y}^2\delta_\mu \overline U(t, \mu^{-1}; y)(\Sigma+\Sigma_0) ] \Big \rangle \\
&+\epsilon \langle \mu^{-1}(dy),\   \partial_y\delta_\mu \overline U(t, \mu^{-1}; y)[  \delta_\mu f^*(t, y, \mu^{-1}; x^1)- \delta_\mu f^*(t, y, \mu^{-1}; y) ]  \rangle \\
&  + o(1) \epsilon+  N o(\epsilon),
\end{align*}
where $o(1)\to 0$ as $N\to \infty$.

Similarly, we have
\begin{align*}
\EE \xi_6&=\frac{1}{2(N-1)^2}\EE\sum_{j,k\ge 2}\Big\{ \mbox{Tr}\Big[\partial_{yz} \delta_{\mu\mu} \overline U(t,  \mu^{-1};x^j,x^k)\\
& \hskip 3.5cm \cdot (X_{t+\epsilon}^k-X_t^k)(X_{t+\epsilon}^j-X_t^j)^T\Big] +o(\epsilon)\Big\}  \\
&= \frac{1}{2(N-1)^2}\EE\sum_{j,k\ge 2}\Big\{ \mbox{Tr[}\partial_{yz} \delta_{\mu\mu} \overline U(t,  \mu^{-1};x^j,x^k) \Sigma_0]\epsilon +o(\epsilon)\Big\} \\
&\quad + \frac{1}{2(N-1)^2}\EE\sum_{k\ge 2}\Big\{ \mbox{Tr[}\partial_{yz} \delta_{\mu\mu} \overline U(t,  \mu^{-1};x^k,x^k) \Sigma] \epsilon +o(\epsilon)\Big\}\\
& = \frac12\langle\mu^{-1}(dy) \mu^{-1}(dz), \    \mbox{Tr}[\partial_{yz} \delta_{\mu\mu}\overline U(t, \mu^{-1};y,z) \Sigma_0  ]\rangle\epsilon \\
&\quad +  \frac{1}{2(N-1)}\langle\mu^{-1}(dy) , \    \mbox{Tr}\{[\partial_{yz} \delta_{\mu\mu} \overline U(t, \mu^{-1};y,z)]|_{z=y} \Sigma\}  \rangle\epsilon
+o(\epsilon).
\end{align*}
Therefore,  we have
\begin{align*}
(N-1) \EE \xi_6 &=   \frac12 (N-1)\langle\mu^{-1}(dy) \mu^{-1}(dz), \    \mbox{Tr}[\partial_{yz} \delta_{\mu\mu}\overline U(t, \mu^{-1};y,z) \Sigma_0]  \rangle\epsilon \\
& +  \frac12\langle\mu^{-1}(dy) , \    \mbox{Tr}\{[\partial_{yz} \delta_{\mu\mu} \overline U(t, \mu^{-1};y,z)]|_{z=y} \Sigma \} \rangle\epsilon
+No(\epsilon).
\end{align*}

Using the local expansion of the right hand side of \eqref{Udeco}
and next letting $\epsilon \to 0$,
we obtain the following  equation
{\allowdisplaybreaks
\begin{align*}
0=&  L^*(t,x^1, \mu^{-1}) + (N-1)\langle \mu^{-1} (dy) , L^*(t,y, \mu^{-1})  \rangle\\
 &+\langle\mu^{-1}(dy), \    \delta_\mu L^*( t, y, \mu^{-1};x^1)  -  \delta_\mu L^*( t, y, \mu^{-1};y)    \rangle \\
&+ \partial_t U^N(t, x^1, \mu^{-1})   + \partial_x U^N(t,x^1, \mu^{-1})f^*(t, x^1, \mu^{-1})  \\
& + \frac12 \mbox{Tr}[ \partial_x^2 U^N(t, x^1, \mu^{-1}) (\Sigma+\Sigma_0)]  \\
&+  \langle  \mu^{-1}(dy), \  \partial_y \delta_\mu U^N(t, x^1, \mu^{-1}; y)
f^*(t,y, \mu^{-1})    \rangle\\
&  +\frac12  \langle\mu^{-1}(dy) ,\    \mbox{Tr} [\partial_y^2 \delta_\mu U^N(t, x^1, \mu^{-1}; y) (\Sigma+\Sigma_0)]  \rangle\\
&+\langle \mu^{-1}(dy),\ \mbox{Tr} [\partial_{xy}\delta_\mu U^N(t,x^1, \mu; y)\Sigma_0 ]\rangle \\
& + \frac12\langle\mu^{-1}(dy) \mu^{-1}(dz), \    \mbox{Tr[}\partial_{yz} \delta_{\mu\mu} U^N(t, x^1, \mu^{-1};y,z) \Sigma_0  ]\rangle \\
&  +  \frac{1}{2(N-1)}\langle\mu^{-1}(dy) , \    \mbox{Tr}\{[\partial_{yz} \delta_{\mu\mu} U^N(t, x^1, \mu^{-1};y,z)]|_{z=y} \Sigma \} \rangle
\\
&+ (N-1)\partial_t \overline U (t, \mu^{-1})  \\
&+(N-1)
\langle \mu^{-1}(dy), \ \partial_y\delta_\mu \overline U(t, \mu^{-1}; y) f^*(t, y, \mu^{-1})\rangle \\
& +\frac12(N-1)\langle\mu^{-1}(dy),\   \mbox{Tr}[\partial_{y}^2\delta_\mu \overline U(t, \mu^{-1}; y)(\Sigma+\Sigma_0) ]  \rangle \\
&+ \langle \mu^{-1}(dy),\   \partial_y\delta_\mu \overline U(t, \mu^{-1}; y)[  \delta_\mu f^*(t, y, \mu^{-1}; x^1)- \delta_\mu f^*(t, y, \mu^{-1}; y) ]  \rangle \\
&+ \frac12 (N-1)\langle\mu^{-1}(dy) \mu^{-1}(dz), \    \mbox{Tr}[\partial_{yz} \delta_{\mu\mu}\overline U(t, \mu^{-1};y,z) \Sigma_0  ]\rangle \\
& +  \frac12\langle\mu^{-1}(dy) , \    \mbox{Tr}\{[\partial_{yz} \delta_{\mu\mu} \overline U(t, \mu^{-1};y,z)]|_{z=y} \Sigma \} \rangle
+o(1) ,
\end{align*}
}
where $o(1)\to 0$ as $N\to\infty$.

Collecting all components with coefficient $N-1$ together, we rewrite the above equation in the form
{\allowdisplaybreaks
\begin{align*}
0= &  (N-1)\Big\{ \partial_t \overline U (t, \mu^{-1)})  + \langle \mu^{-1} (dy) , L^*(t,y, \mu^{-1})  \rangle  \\
&+
\langle \mu^{-1}(dy), \ \partial_y\delta_\mu \overline U(t, \mu^{-1}; y) f^*(t, y, \mu^{-1})\rangle \\
& +\frac12\langle\mu^{-1}(dy),\   \mbox{Tr}[\partial_{y}^2\delta_\mu \overline U(t, \mu^{-1}; y)(\Sigma+\Sigma_0) ]  \rangle \\
&+ \frac12 \langle\mu^{-1}(dy) \mu^{-1}(dz), \    \mbox{Tr[}\partial_{yz} \delta_{\mu\mu}\overline U(t, \mu^{-1};y,z) \Sigma_0  ]\rangle
\Big\} \\
&+ \partial_t U^N(t, x^1, \mu^{-1})  + \partial_x U^N(t,x^1, \mu^{-1})f^*(t, x^1, \mu^{-1})  \\
& + \frac12 \mbox{Tr}[ \partial_x^2 U^N(t, x^1, \mu^{-1}) (\Sigma+\Sigma_0)] + L^*(t,x^1, \mu^{-1}) \\
&+  \langle  \mu^{-1}(dy), \  \partial_y \delta_\mu U^N(t, x^1, \mu^{-1}; y)
f^*(t,y, \mu^{-1})    \rangle\\
&  +\frac12  \langle\mu^{-1}(dy) ,\    \mbox{Tr} [\partial_y^2 \delta_\mu U^N(t, x^1, \mu^{-1}; y) (\Sigma+\Sigma_0)]  \rangle\\
&+\langle \mu^{-1}(dy),\ \mbox{Tr} [\partial_{xy}\delta_\mu U^N(t,x^1, \mu; y)\Sigma_0 ]\rangle \\
& + \frac12\langle\mu^{-1}(dy) \mu^{-1}(dz), \    \mbox{Tr[}\partial_{yz} \delta_{\mu\mu} U^N(t, x^1, \mu^{-1};y,z) \Sigma_0  ]\rangle \\
&  +  \frac{1}{2(N-1)}\langle\mu^{-1}(dy) , \    \mbox{Tr}\{[\partial_{yz} \delta_{\mu\mu} U^N(t, x^1, \mu^{-1};y,z)]|_{z=y} \Sigma \} \rangle\\
 &+\langle\mu^{-1}(dy), \    \delta_\mu L^*( t, y, \mu^{-1};x^1)  -  \delta_\mu L^*( t, y, \mu^{-1};y)    \rangle \\
&+ \langle \mu^{-1}(dy),\   \partial_y\delta_\mu \overline U(t, \mu^{-1}; y)[  \delta_\mu f^*(t, y, \mu^{-1}; x^1)- \delta_\mu f^*(t, y, \mu^{-1}; y) ]  \rangle \\
& +  \frac12\langle\mu^{-1}(dy) , \    \mbox{Tr}\{[\partial_{yz} \delta_{\mu\mu} \overline U(t, \mu^{-1};y,z)]|_{z=y} \Sigma \} \rangle +o(1) .
\end{align*}
}

Provided that we have chosen the function $\overline U(t,\mu)$ to satisfy \eqref{MEUb},
then the above equation further implies
{\allowdisplaybreaks
\begin{align}
 0=&\partial_t U^N(t, x^1, \mu^{-1})  + \partial_x U^N(t,x^1, \mu^{-1})f^*(t, x^1, \mu^{-1}) \label{MEUN}  \\
& + \frac12 \mbox{Tr}[ \partial_x^2 U^N(t, x^1, \mu^{-1}) (\Sigma+\Sigma_0)] + L^*(t,x^1, \mu^{-1}) \notag \\
&+  \langle  \mu^{-1}(dy), \  \partial_y \delta_\mu U^N(t, x^1, \mu^{-1}; y)
f^*(t,y, \mu^{-1})    \rangle \notag  \\
&  +\frac12  \langle\mu^{-1}(dy) ,\    \mbox{Tr} [\partial_y^2 \delta_\mu U^N(t, x^1, \mu^{-1}; y) (\Sigma+\Sigma_0)]  \rangle \notag \\
&+\langle \mu^{-1}(dy),\ \mbox{Tr} [\partial_{xy}\delta_\mu U^N(t,x^1, \mu; y)\Sigma_0 ]\rangle \notag \\
& + \frac12\langle\mu^{-1}(dy) \mu^{-1}(dz), \    \mbox{Tr[}\partial_{yz} \delta_{\mu\mu} U^N(t, x^1, \mu^{-1};y,z) \Sigma_0  ]\rangle \notag  \\
&  +  \frac{1}{2(N-1)}\langle\mu^{-1}(dy) , \    \mbox{Tr}\{[\partial_{yz} \delta_{\mu\mu} U^N(t, x^1, \mu^{-1};y,z)]|_{z=y} \Sigma \} \rangle\notag  \\
 &+\langle\mu^{-1}(dy), \    \delta_\mu L^*( t, y, \mu^{-1};x^1)  -  \delta_\mu L^*( t, y, \mu^{-1};y)    \rangle \notag  \\
&+ \langle \mu^{-1}(dy),\   \partial_y\delta_\mu \overline U(t, \mu^{-1}; y)[  \delta_\mu f^*(t, y, \mu^{-1}; x^1)- \delta_\mu f^*(t, y, \mu^{-1}; y) ]  \rangle  \notag \\
& +  \frac12\langle\mu^{-1}(dy) , \    \mbox{Tr}\{[\partial_{yz} \delta_{\mu\mu} \overline U(t, \mu^{-1};y,z)]|_{z=y} \Sigma \} \rangle +o(1)  . \notag
\end{align}
}
Taking $N\to \infty$, we obtain  \eqref{MEUa} as the limiting form of \eqref{MEUN}.

\end{document}